\documentclass[smallextended]{svjour3}  
\smartqed  
\usepackage{graphicx}
\usepackage[utf8]{inputenc} 
\usepackage[T1]{fontenc} 
\usepackage{hyperref}  
\usepackage{url}  
\usepackage{microtype}
\usepackage{array}
\usepackage{color} 
\usepackage{tabularx}
\usepackage{amsmath,mathdots}
\usepackage{amssymb}
\usepackage{amsfonts}
\usepackage{txfonts}
\usepackage{xcolor}
\usepackage{bm}
\usepackage{soul}
\usepackage{float}
\usepackage{cite}
\usepackage{booktabs, multirow}
\usepackage{subfigure}
\usepackage{yhmath}
\usepackage{comment}
\usepackage{enumerate}

\DeclareMathOperator{\Tr}{Tr}
\hypersetup{
	colorlinks=true, 
	linkcolor=blue, 
	citecolor=blue, 
	urlcolor=black  
} 
\usepackage[colorinlistoftodos]{todonotes}

\hypersetup{
	pdftitle={},
	pdfsubject={},
	pdfauthor={R. Abgrall, M. Dumbser, P.H. Maire, E. Zampa},
}

\journalname{Journal of Scientific Computing}

\usepackage{xcolor}
\definecolor{darkgreen}{rgb}{0.1,0.6,0.1}

\newcommand{\halb}{\frac{1}{2}}

\newcommand{\A}{\mathbf{A}} 
\newcommand{\B}{\mathbf{B}} 
\newcommand{\E}{\mathbf{E}} 
 
\renewcommand{\H}{\mathbf{H}}

\newcommand{\p}{\mathbf{p}} 
\renewcommand{\v}{\mathbf{v}} 
\newcommand{\q}{\mathbf{q}} 
\newcommand{\f}{\mathbf{f}} 
\newcommand{\g}{\mathbf{g}} 
\newcommand{\x}{\mathbf{x}} 
\newcommand{\w}{\mathbf{w}} 
\newcommand{\n}{\mathbf{n}} 
\renewcommand{\u}{\mathbf{u}} 
\newcommand{\dx}{\; \mathtt{d}\mathbf{x}}
\newcommand{\dgamma}{\; \mathtt{d}\gamma}
\newcommand{\dd}{\mathtt{d}}
\newcommand{\dt}{\mathtt{d}t}

\newcommand{\Uh}{\mathcal{U}_h^N}
\newcommand{\Wh}{\mathcal{W}_h^{N+1}}

\newcommand{\uh}{\u_h}
\newcommand{\vh}{\v_h}
\newcommand{\Yh}{Y_h} 
\newcommand{\Bh}{\B_h} 
\newcommand{\gh}{\tilde{\g}_h}
\newcommand{\Zh}{\tilde{Z}_h}
\newcommand{\Ah}{\tilde{\A}_h}
\newcommand{\wh}{\tilde{\w}_h}
\newcommand{\fh}{\tilde{\f}_h}

\newcommand{\Pih}{\mathbf{\Pi}_h}
\newcommand{\PW}{\Pih}
\newcommand{\PU}{\mathbf{\pi}_h}

\newcommand{\dq}
{\boldsymbol{\phi}}
\newcommand{\dqt}{\boldsymbol{\psi}_h}
\newcommand{\norm}[1]{\lVert #1 \rVert}
\newcommand{\seminorm}[1]{\lvert #1 \rvert}

\allowdisplaybreaks

\begin{document}
	
	\title{On structure-preserving and pointwise conservative continuous DG schemes for hyperbolic systems}
	
	\titlerunning{Structure-preserving CG-DG schemes for hyperbolic systems}
	
	\author{R\'emi Abgrall \and 
		Michael Dumbser \and \\  
		Pierre-Henri Maire \and
		Enrico Zampa 
	}
	
	
	\institute{R\'emi Abgrall \at Institut f\"ur Mathematik, Universit\"at Z\"urich, Winterthurerstrasse 190, CH 8057 Z\"urich, Switzerland\\ \email{remi.abgrall@math.uzh.ch}\\
		Michael Dumbser$^*$ \at
		Laboratory of Applied Mathematics, DICAM, University of Trento, Via Mesiano 77, 38123 Trento, Italy.\\
		\email{michael.dumbser@unitn.it} \\
        Pierre-Henri Maire \at CEA CESTA, 33116 Le Barp, France.\\
        \email{pierre-henri.maire@cea.fr} \\ 
        Enrico Zampa \at Faculty of Mathematics, Universit\"at Wien, Oskar-Morgenstern-Platz 1, 1090 Vienna, Austria.\\
        \email{enrico.zampa@univie.ac.at}
	}
	
	\date{Received: date / Accepted: date}
	
	\maketitle


\begin{abstract}
	We present a new class of structure-preserving semi-discrete continuous-discontinuous Galerkin (CG-DG) finite element schemes for linear and nonlinear hyperbolic systems of partial differential equations on unstructured simplex meshes.
	The key ingredient here is the use of two different but compatible approximation spaces: the classical DG space $\Uh$ of discontinuous piecewise polynomials of degree up to $N$, and a classical finite element space $\Wh$ of globally continuous piecewise polynomials of degree $N+1$. The discrete solution $\u_h$ is sought in $\Uh$, while its reconstruction $\tilde{\w}_h$ and a suitable \textit{discrete flux field} $\tilde{\f}_h$ are computed in $\Wh$.
	 The new CG-DG schemes presented in this paper automatically satisfy the following properties: i) they are not only cellwise conservative, but also locally pointwise conservative everywhere, hence they satisfy the integral form of the conservation law on \textit{arbitrary control volumes} that do not have to coincide with the mesh at all; ii) with $\mathbf{A}$ and $Z$ a vector and a scalar potential, respectively, the new methods naturally satisfy the two basic vector calculus identities $\nabla \cdot \nabla \times \mathbf{A}=0$ and $\nabla \times \nabla Z=0$ exactly pointwise locally and globally everywhere on the discrete level; iii) for linear symmetric hyperbolic systems the quadratic energy of the reconstructed solution $\tilde{\w}_h \in \Wh$ obtained from $\u_h \in \Uh$ via $L^2$ projection is exactly conserved, i.e. the reconstructed solution is nonlinearly stable in the $L^2$ norm. 
	 For $N=0$ our new schemes are directly related to cell-centered finite volume schemes with suitable vertex-based fluxes.
	All claimed properties of the schemes are first mathematically proven and are then also verified via suitable numerical tests. We show applications of our approach to three hyperbolic systems: i) the equations of linear acoustics, in which the velocity field must remain curl-free if it was so at the initial time; ii) the vacuum Maxwell equations in the absence of charges, for which the magnetic and the electric fields must remain divergence-free for all times if they were initially divergence-free; 
	iii) the Euler equations of compressible gasdynamics. 
\end{abstract}

\keywords{
	continuous-discontinuous Galerkin (CG-DG) finite element schemes; 
linear and nonlinear hyperbolic systems; 
structure-preserving (SP) schemes; 
discrete flux field; 
discrete vector calculus identities; 
divergence-free and curl-free schemes;
applications to the equations of linear acoustics, 
the vacuum Maxwell equations 
and to the nonlinear compressible Euler equations
}


\section{Introduction}
\label{sec.intro}

First order hyperbolic conservation laws with stationary differential constraints (involutions) are of fundamental importance in continuum physics. Probably the best-known involution is the divergence-free property of the magnetic field in the case of the Maxwell and magnetohydrodynamics (MHD) equations. This mathematical property of the magnetic field is equivalent to the physical statement that there exist no magnetic monopoles. Similarly, in computational solid mechanics, the inverse deformation gradient, being a gradient, must remain curl-free for all times. 

The Yee scheme \cite{Yee66} was the first compatible discretization of the time-domain Maxwell equations and was based on a staggered mesh that allowed the definition of compatible discrete derivative operators so that the discrete magnetic field remains always divergence-free if it was initially divergence-free. 

Since the seminal work of Yee many different numerical schemes have been designed in order to satisfy the two classical vector calculus identities exactly also at the discrete level, namely that the discrete divergence of a discrete curl must be zero and that the discrete curl of a discrete gradient must vanish. Without pretending to be complete, we refer the interested reader to the well-known family of mimetic finite difference schemes  \cite{HymanShashkov1997,Margolin2000,Lipnikov2014}, compatible finite volume schemes \cite{BalsaraCED,BalsaraKaeppeli,HazraBalsara} and compatible finite elements \cite{Nedelec1,Nedelec2,Hiptmair,Arnold2006,Monk,Alonso2015,CPSo2016,ddrham,BBZ22,ZAR,Zampa1}. 

While the divergence-free property of the magnetic field has almost immediately attracted the interest of researchers in numerical analysis, compatible schemes that preserve the curl-free property of a vector field exactly on the discrete level have been developed much later, see e.g. \cite{JeltschTorrilhon2006,BalsaraCurlFree,SIGPR,SIST,Perrier3,CurlFreeTwoFluid}. Very recently, notable progress in the field of exactly constraint-preserving finite volume and discontinuous Galerkin schemes has also been achieved in \cite{Sidilkover2025,Barsukow2024} and \cite{Perrier1,Perrier2,CompatibleDG1}.  
Compatible discrete differential operators are also essential in the construction of staggered and cell-centered Lagrangian schemes, see e.g. \cite{Caramana1998,DepresMazeran2003,Maire2007,ShashkovCellCentered,LMR2016,Maire2020,Boscheri24_mhd} and references therein, as they allow for a natural compatibility with the so-called geometric conservation law (GCL). At first order of accuracy, the method presented in this paper is actually closely related to \cite{Maire2020,Boscheri24_mhd,HTCLagrange,HTCLagrangeGPR}, but on fixed meshes in Eulerian coordinates and not on moving meshes. These methods require the computation of numerical fluxes at the vertices of the mesh, which is linked to the construction of genuinely multidimensional Riemann solvers, see e.g. \cite{BalsaraMultiDRS,MUSIC1,MUSIC2,MultidOsher,Gallice2022,Delmas2025,DelGrosso2026,Abgrall2026} and their generalization to arbitrary co-dimension of the mesh via the composite finite volume approach introduced in \cite{HochComposite1,HochComposite2,HochComposite3}.  

The key ingredient of the new CG-DG schemes introduced in this paper is the use of two different but \textit{compatible} approximation spaces, as in \cite{CompatibleDG1}: the first one is a DG space $\Uh$ of discontinuous piecewise polynomials of degree up to $N$ and the second one is a classical finite element space $\Wh$ of globally continuous piecewise polynomials of degree $N+1$. In the new CG-DG schemes, the discrete solution $\u_h$ is sought in $\Uh$, while a suitable \textit{discrete flux field} $\tilde{\f}_h$ is computed in $\Wh$. In the 1D case our methods are closely related to the so-called flux reconstruction (FR) method \cite{FR1,FR2,FR3,FR4,FR5,FR6} but in the unstructured case treated in this paper, the two methods are completely different, see e.g. \cite{FR6} and related spectral difference methods on unstructured grids \cite{USSD1,USSD2}. Since different approximation polynomials are used for the representation of the state vector and of the flux tensor, the new CG-DG approach also presents
some common ingredients with $P_NP_M$ and reconstructed DG schemes, see e.g. \cite{Dumbser2008,luo1,luo2,luo5}. 
However, we view the new method rather as an extension and generalization of our previous work \cite{CompatibleDG1}.

The new CG-DG schemes introduced in this paper have the following features and provable properties: 
\begin{enumerate}[(i)]  
\item they rely on a \textit{discontinuous} piecewise polynomial representation of the state vector $\uh \in \Uh$ in the DG space of degree $N$ and on a \textit{globally continuous} representation of the discrete numerical flux field $\fh \in \Wh$ in the continuous finite element space of degree $N+1$; 
\item they are exactly pointwise conservative everywhere and also conservative on \textit{arbitrary} control volumes that do not need to coincide with the mesh; 
\item for linear symmetric hyperbolic systems the quadratic energy of the globally continuous reconstruction $\wh \in \Wh$ obtained from $\uh \in \Uh$ via $L^2$ projection is exactly conserved, i.e. $\wh$ is provably nonlinearly stable in the $L^2$ norm; 
\item the discrete differential operators appearing in the method satisfy a discrete Schwarz theorem, see also \cite{CompatibleDG1}, and therefore our CG-DG schemes satisfy the two classical vector calculus identities $\nabla \times \nabla Z = 0$ and $\nabla \cdot \nabla \times \A = 0$ exactly also on the discrete level, pointwise inside each element and also across element interfaces. 
\end{enumerate}

The rest of this paper is organized as follows: in Section \ref{sec.method} we introduce the numerical method and give more details for the special case $N=0$ in Section \ref{sec:Lagrangian}. In Section \ref{sec.prop} we analyze the scheme and prove its properties. In Section \ref{sec.effi} we show a more efficient way to compute the discrete flux field compared to the $L^2$ projection used in Section \ref{sec.method}. Numerical results are shown in Section \ref{sec.results} and the conclusions are drawn in Section \ref{sec.concl}.

\section{Numerical method}
\label{sec.method}

\subsection{Governing PDE under consideration} 

We present the numerical method for the general case of a nonlinear system of hyperbolic conservation laws 
\begin{equation}
	\label{eqn.pde}
	\partial_t \q + \nabla \cdot \f = 0,
\end{equation}
with $\q=\q(\x,t)$ the state vector and $\f=\f(\q)$ the flux tensor.
Furthermore, $t$ is the time and $\x$ the vector of spatial coordinates $\x \in \Omega \subset \mathbb{R}^d$, with $\Omega$ the computational domain and $d$ the number of space dimensions. 
Throughout this paper we assume the system \eqref{eqn.pde} to be hyperbolic. 
We will also consider the special case of linear symmetric-hyperbolic systems, in which $\f(\q)$ is a linear function of $\q$ and the corresponding matrices appearing in the flux are symmetric. 
The computational domain $\Omega$ is discretized at the aid of a set of conforming and non-overlapping simplex elements $T_k$ and the triangulation / tetrahedralization of $\Omega$ is denoted by $\mathcal{T}_h = \bigcup T_k$.  
In this paper we furthermore denote by the set $\mathcal{S}_h$ the skeleton of the mesh, i.e. the union of all edges (2D) or faces (3D) of all elements $T_k$.  

\subsection{Compatible approximation spaces} 
  
The essence of the CG-DG scheme is the use of two different approximation spaces: a classical discontinuous Galerkin finite element space $\Uh$ of discontinuous piecewise polynomials of degree $N$ and a classical continuous finite element space $\Wh$ of globally continuous piecewise polynomials of degree $N+1$. It is obvious that spatial derivatives of solutions in the continuous FEM space $\Wh$ can be \textit{exactly} represented in the DG space $\Uh$. This is indeed a fundamental property of the particular choice of $\Uh$ and $\Wh$. 

We now make the following ansatz for the discrete solutions of \eqref{eqn.pde}:
The vector $\q=\q(\x,t)$ will be approximated in $\Uh$ by a DG scheme
\begin{equation}
	\u_h(\x,t) = \sum \limits_j \phi_j(\x) \hat \u_j(t) := 
	\phi_j(\x) \hat \u_j(t) \quad \in \Uh,
	\label{eqn.uh} 
\end{equation} 
with spatial basis functions $\phi_j(\x) \in \Uh$ and time-dependent degrees of freedom $\hat \u_j(t)$. Since $\uh$ is \textit{discontinuous}, the particular choice of the DG basis functions is not crucial. Here we simply use classical nodal Lagrange basis functions of degree $N$ within each element. Throughout this paper we will make use of the Einstein summation convention which implies summation over two repeated indices. The $\u_h$ are allowed to \textit{jump} from one element $T_k$ to another. 
We then introduce the following auxiliary discrete solution represented in the continuous FEM space: 
\begin{equation}
	\wh(\x,t) = \sum \limits_j \psi_j(\x) \hat \w_j(t) := 
	\psi_j(\x) \hat \w_j(t) \quad \in \Wh,
	\label{eqn.wh} 
\end{equation} 
with spatial basis functions $\psi_j(\x) \in \Wh$ and time-dependent degrees of freedom $\hat \w_j(t)$. Since $\wh$ is \textit{globally continuous} the natural choice for the basis functions is the classical Lagrange basis of continuous finite elements of degree $M=N+1$. 
The auxiliary solution $\wh$ can be readily obtained from the DG solution $\uh$ via a suitable projector $\Pih$ from $\Uh$ to $\Wh$ as
\begin{equation}
	\wh = \Pih \uh. 
\end{equation}
One may choose, for example, $L^2$ projection as a natural projector $\Pih$, i.e. by requiring 
\begin{equation}
	\int \limits_{\Omega} \psi_i(\x) \wh(\x,t) d\x = 
	\int \limits_{\Omega} \psi_i(\x) \uh(\x,t) d\x, \qquad \forall \psi_i \in \Wh.  
	\label{eqn.L2proj}
\end{equation}
Inserting \eqref{eqn.uh} and \eqref{eqn.wh} into \eqref{eqn.L2proj} one obtains in index notation 
\begin{equation}
	\int \limits_{\Omega} \psi_i \psi_j d\x \, \hat{\w}_j(t) = 
	\int \limits_{\Omega} \psi_i \phi_j d\x \, \hat{\u}_j(t),  
	\label{eqn.L2proj.dof}
\end{equation}
which allows us to compute the degrees of freedom of the auxiliary FEM solution $\hat{\w}_j$ from the degrees of freedom $\hat{\u}_j$ of the DG solution. Note that this operation requires the inversion of the global mass matrix of the continuous finite elements. Since the global mass matrix is symmetric and positive definite, one can use a classical matrix-free conjugate gradient method to solve the problem iteratively and efficiently, see \cite{cgmethod}. To improve the efficiency of this procedure, see Section \ref{sec.effi} below. 

\subsection{Discrete flux field} 

We now also introduce a discrete representation of the flux tensor, denoted throughout this paper as the \textit{discrete flux field}, which is defined as follows: 
\begin{equation}
	\fh(\x,t) = \sum \limits_j \psi_j(\x) \hat \f_j(t) := 
	\psi_j(\x) \hat \f_j(t) \quad \in \Wh.
	\label{eqn.fh} 
\end{equation} 
Since the continuous finite element basis is a \textit{nodal} Lagrange basis, one simply carries out a \textit{pointwise flux evaluation} at the finite element nodes as
\begin{equation}
	\hat \f_j(t) = \f(\hat \w_j(t)).
	\label{eqn.nodalflux}
\end{equation} 
Equations \eqref{eqn.L2proj.dof} and \eqref{eqn.nodalflux} allow the computation of the degrees of freedom of the discrete flux field $\fh$ in terms of the degrees of freedom of the DG solution $\uh$. Alternatively, the degrees of freedom of the discrete flux field could also have been obtained via  $L^2$ projection as follows: 
\begin{equation}
	\int \limits_{\Omega} \psi_i(\x) \fh(\x,t) d\x = 
	\int \limits_{\Omega} \psi_i(\x) \f(\wh(\x,t)) d\x.  
	\label{eqn.L2proj.fh}
\end{equation}
However, this is a line of research that was not pursued further and throughout this paper we simply use the nodal evaluation \eqref{eqn.nodalflux}, i.e. \textit{interpolation} and \textit{not} $L^2$ projection of the discrete flux field.  

\subsection{New CG-DG scheme}

We now have all ingredients to introduce the new CG-DG scheme. Multiplication of \eqref{eqn.pde} with a test function $\phi_i \in \Uh$, integration over an element $T_k \in \mathcal{T}_h$ and replacing the state vector $\q$ by the discrete solution $\uh$ and the flux tensor $\f$ by the discrete flux field yields 
\begin{equation}
	\int \limits_{T_k} \phi_i \, \partial_t \uh \, d\x + \int \limits_{T_k} \phi_i \, \nabla \cdot \fh \, d\x = 0, \qquad 
	\forall \phi_i \in \Uh.
	\label{eqn.cgdg} 
\end{equation}
After inserting the ansatz \eqref{eqn.uh} and \eqref{eqn.fh} into \eqref{eqn.cgdg} one obtains in terms of the degrees of freedom 
\begin{equation}
	\int \limits_{T_k} \phi_i \phi_j \, d\x \, \frac{d\hat{\u}_j}{dt} + \int \limits_{T_k} \phi_i \, \nabla \psi_j \, d\x \, \cdot \hat{\f}_j = 0, \qquad 
	\forall \phi_i \in \Uh.
	\label{eqn.cgdg.dof} 
\end{equation}
Since the discrete flux field is \textit{globally continuous}, no integration by parts and no further steps are needed. The new semi-discrete CG-DG scheme is completely defined by \eqref{eqn.cgdg}, obviously making use of \eqref{eqn.L2proj.dof}, \eqref{eqn.fh} and \eqref{eqn.nodalflux}. As usual in RKDG schemes, time integration is performed at the aid of classical or TVD Runge-Kutta schemes of suitable order. 

\subsection{Compatible pair of discrete nabla operators}

Based on the ideas first outlined in \cite{CompatibleDG1} and making use of the two compatible approximation spaces $\Uh$ and $\Wh$, a pair of compatible discrete nabla operators is introduced. The subscripts $a$ and $c$ will be used for the DG degrees of freedom in a cell (triangle / tetrahedron $T_k$) and the subscripts $p$ and $q$ for the nodal degrees of freedom of the continuous FEM. Furthermore, tensor indices of physical fields are denoted by $i$, $j$, $k$, $l$ and $m$. 

\subsubsection{Discrete primary nabla operator} 

We define the rank three stiffness tensor as 
\begin{equation} \mathbb K = \left\{ \mathbf{K}_{cp} \right\} = \int \limits_{\Omega} \phi_c \nabla \psi_p d\mathbf{x}, \end{equation}  
or, equivalently, in Einstein index notation, 
\begin{equation}
K_{cpm} = \int \limits_{\Omega} \phi_c \partial_m \psi_p d\mathbf{x}.
\end{equation}
Note that this discrete operator directly appears in the CG-DG scheme \eqref{eqn.cgdg.dof} above.
To get the discrete primary nabla operator we multiply with the inverse of the element-local mass matrix of the DG scheme, which is easy to compute, in particular for orthogonal basis functions, see e.g. \cite{Dubiner},      
\begin{equation} D_{ac} = \int \limits_{\Omega} \phi_a \phi_c d\mathbf{x}  \end{equation}
i.e. 
\begin{equation} 
\nabla_c^p = D_{ca}^{-1} \mathbf{K}_{ap}, 
\qquad \textnormal{or, } \qquad  
\left(\partial_m\right)_c^p = D_{ca}^{-1} K_{apm}.
\label{eqn.primary.nabla}
\end{equation} 
Let now $Z$ be a scalar potential and let its gradient be denoted by $\v = \nabla Z$. The approximation of $Z$ in $\mathcal{W}^{N+1}_h$ is then denoted by $\Zh$, with degrees of freedom ${Z}_q$ and representation 
\begin{equation} 	
\Zh = \sum_q \psi_q(\mathbf{x}) {Z}_q:= \psi_q Z_q \qquad \in \mathcal{W}^{N+1}_h,
\label{eqn.Z.ansatz}
\end{equation} 	
while the discrete gradient is denoted by $\vh \in \mathcal{U}^N_h$ and representation
\begin{equation} 	
\vh = \sum_a \phi_a(\mathbf{x}) \v_a := \phi_a \v_a = \phi_a v_{ma} \qquad \in \mathcal{U}^N_h,
\label{eqn.v.ansatz}
\end{equation} 	
where the index $m$ in $v_{ma}$ refers to the components of the vector field $\v$ and the index $a$ to the degrees of freedom. 
The degrees of freedom of the discrete gradient then simply read 
\begin{equation}
\v_c = \nabla_c^p Z_p = \sum_{p} D_{ca}^{-1} \mathbf{K}_{ap} Z_p,  
\qquad 
\textnormal{ or, } 
\qquad 
v_{mc} = \left( \partial_m \right)_c^p Z_p = D_{ca}^{-1} K_{apm} Z_p.    
\end{equation}
Because of the particular choice of compatible discrete function spaces $\Uh$ and $\Wh$ for $\vh$ and $\Zh$, respectively, this primary discrete nabla operator is \textit{exact}, i.e. $\vh = \nabla_h \Zh = \phi_c \nabla_c^p Z_p = \nabla \Zh$, with $\nabla$ the exact nabla operator.  

The discrete primary nabla operator introduced above can also be applied to general discrete vector fields $\Ah = \psi_q \mathbf{A}_{q} = \psi_q A_{kq} \in \mathcal{W}^{N+1}_h$, with $k$ the tensor index of the physical field and $q$ the index for the degrees of freedom:
\begin{equation}
 \nabla_c^p \mathbf{A}_{p} = \sum_{p} D_{ca}^{-1} \mathbf{K}_{ap} \mathbf{A}_{p}, 
 \qquad \textnormal{ or, } \qquad 
 \left( \partial_m \right)_c^p A_{kp} = D_{ca}^{-1} K_{apm} A_{kp}.  
\end{equation}
For $N=0$ we have $\phi_c=1$, the inverse mass matrix is one divided by the cell volume and $\nabla \psi_p$ is the gradient of the continuous P1 Lagrange basis functions
and we obtain the discrete primary nabla operator already used in the Lagrangian schemes \cite{Despres2005,Despres2009,Maire2007,Maire2009,LMR2016,Maire2020,Boscheri_hyperelast_22,Boscheri24_mhd,HTCLagrange,HTCLagrangeGPR} as a special case of our method. 

The discrete primary nabla operator allows us now to rewrite the CG-DG scheme \eqref{eqn.cgdg.dof} in an even more compact form. Multiplying \eqref{eqn.cgdg.dof} with the inverse of the element mass matrix $D_{ac}$ of the DG scheme yields 
\begin{equation}
 \frac{d\hat{\u}_c}{dt} + \nabla_c^p \cdot \hat{\f}_p = 0,
 \label{eqn.cgdg.nabla}
\end{equation} 
which for nodal basis functions $\phi_c$ and $\psi_p$ could now even be interpreted as a kind of high order unstructured finite difference scheme. However, this is not the objective of this paper, but it can be related to the flux-reconstruction (FR) and spectral difference methods introduced in \cite{FR6,USSD1,USSD2}. 

\subsubsection{Discrete dual nabla operator} 

The discrete dual nabla operator \cite{CompatibleDG1} can be obtained via integration by parts 
\begin{eqnarray}
 \int \limits_\Omega \psi_p \nabla \phi_c \, d\x & = & 
 \int \limits_{\Omega \backslash \mathcal{S}_h} \psi_p \nabla \phi_c \, d\x + \int \limits_{\mathcal{S}_h} \psi_p \left( \phi_c^+ - \phi_c^- \right) \, \n \, dS 
 \nonumber \\
 & = & 
 - \int \limits_\Omega \nabla \psi_p  \, \phi_c \, d\x + \int \limits_{\partial \Omega}  \psi_p \phi_c \, \n \, dS.
 \label{eqn.dual.ibp}
\end{eqnarray}
Recall that $\mathcal{S}_h$ denotes the skeleton of the mesh, i.e. the union of all edges (2D) or faces (3D) of all elements $T_k$, while $\phi_c^-$ and $\phi_c^+$ denote the boundary-extrapolated values of $\phi_c$ on the two sides of a face of $\mathcal{S}_h$ and $\n$ is the unit normal vector on this face pointing from the $-$ side to the $+$ side.  
Throughout this paper, we assume suitable boundary conditions on $\partial \Omega$ that allow the boundary contributions on $\partial \Omega$ to be dropped, see \cite{CompatibleDG1} for details. Hence, in these situations, the above
expression simplifies to 
\begin{equation}
 \int \limits_\Omega \psi_p \nabla \phi_c d\x = 
 \int \limits_{\Omega \backslash \mathcal{S}_h} \psi_p \nabla \phi_c \, d\x + \int \limits_{\mathcal{S}_h} \psi_p \left( \phi_c^+ - \phi_c^- \right) \, \n \, dS 
 = 
 - \int \limits_\Omega \nabla \psi_p  \, \phi_c \, d\x,
 \label{eqn.dual.ibp.bc}
\end{equation}
which allows us to define a dual nabla operator from the DG space $\Uh$ into the FEM space $\Wh$ as follows:   
\begin{equation} 
\tilde{\nabla}_p^c = -\mathbf{K}_{cp}, 
\qquad \textnormal{or, in index notation,} \qquad  
(\tilde{\partial}_m)_p^c = - K_{cpm}.
\label{eqn.dual.nabla}
\end{equation} 
Note that this operator does not  contain the multiplication by the inverse of the global finite element mass matrix 
\begin{equation}
    M_{pq} = \int \limits_{\Omega} \psi_p \psi_q d\x.     
\end{equation}
In the case $N=0$ we have $\phi_c=1$, hence we recover the discrete dual nabla operator as in \cite{Despres2005,Despres2009,Maire2007,Maire2009,LMR2016,Maire2020,Boscheri_hyperelast_22,Boscheri24_mhd,HTCLagrange,HTCLagrangeGPR}, which, however, all employ mass lumping and therefore do not require the inversion of a global mass matrix.      
Let $Y$ be a scalar potential and let $\g = \nabla Y$ be its gradient. The approximation of $Y$ in $\mathcal{U}^N_h$ is denoted by $\Yh$, with degrees of freedom ${Y}_a$ and representation 
\begin{equation} 	
\Yh = \sum_a \phi_a(\mathbf{x}) {Y}_a:= \phi_a Y_a \qquad \in \mathcal{U}^N_h,
\label{eqn.Y.ansatz}
\end{equation} 	
while the discrete gradient is denoted by $\gh \in \mathcal{W}^{N+1}_h$ and with representation
\begin{equation} 	
\gh = \sum_p \psi_p(\mathbf{x}) \g_p := \psi_p \g_p = \psi_p g_{mp} \qquad \in \mathcal{W}^{N+1}_h,
\label{eqn.w.ansatz}
\end{equation} 	
where the index $m$ in $g_{mp}$ refers to the spatial components of the vector field and the index $p$ to the associated degrees of freedom of the discrete solution $\gh$. The discrete dual gradient $\tilde{\nabla}_h$ applied to $\Yh$ reads 
\begin{equation}
M_{pq} \g_q = \tilde{\nabla}_p^c Y_c = -\sum_{c} \mathbf{K}_{cp} Y_c,  
\qquad 
\textnormal{ or, }  
\qquad 
M_{pq} g_{mq} = ( \tilde{\partial}_m )_p^c Y_c = - K_{cpm} Y_c.    
\end{equation}
We also provide the discrete dual nabla operator applied to a vector field $\mathbf{A}_h = \phi_a A_{ka} \in \mathcal{U}^N_h$, with $k$ the tensor index of the physical field and $a$ the index for the degrees of freedom:
\begin{equation}
 \tilde{\nabla}_p^c \mathbf{A}_{c} = -\sum_{c} \mathbf{K}_{cp} \mathbf{A}_{c}, 
 \qquad \textnormal{ or, } \qquad 
 ( \tilde{\partial}_m )_p^c A_{kc} =  - K_{cpm} A_{kc}.  
\end{equation}

\subsection{Shock capturing via suitable artificial viscosity}\label{sec.shockcapturing}

In order to capture shocks and strong gradients, we employ a simple concept of artificial viscosity, which, however, is fully compatible with the principles of the CG-DG scheme outlined above, in particular with the concept of a discrete flux field. 

For general hyperbolic systems without involutions, the degrees of freedom of the discrete flux field are simply modified as 
\begin{equation}
	\hat{\f}_j = \f(\hat{\w}_j) - \epsilon_j \widehat{\nabla \u}_j.  
	\label{eqn.shock.capture} 
\end{equation} 
Here, $\epsilon_j$ is the artificial viscosity coefficient and the 
$\widehat{\nabla \u}_j$ are the degrees of freedom of the discrete dual nabla operator applied to the DG solution $\uh$, i.e. we have 
\begin{equation}
	\int \limits_\Omega \psi_i \tilde{\nabla} \uh \, d\x =  
	 \int \limits_\Omega \psi_i \psi_j \, d\x \, \widehat{\nabla \u}_j =  
	-\int \limits_\Omega \uh \otimes \nabla \psi_i \, d\x.  
\end{equation}
In this paper we simply use 
\begin{equation}
	\epsilon_j = \halb \chi_j \frac{h_j}{2N + 1} s^{\max}_j, 
\end{equation}
with $ s^{\max}_j$ a suitable upper bound of the maximum absolute value of the eigenvalues of the hyperbolic system in all normal directions, $h_j$ a characteristic mesh spacing and $\chi_j \in [0,1]$ an indicator which takes the values $\chi_j = 0$ in unlimited areas and $\chi_j = 1$ in troubled zones. In the test problems with discontinuities we simply set $\chi_j = 1$ as the default choice. However, for the circular Sod problem of Section~\ref{sec.ep2d} we compare this choice with a more sophisticated indicator $\chi_j$ based on the oscillation-eliminating (OE) sensor recently proposed in \cite{AbgrallLiuDg}. For state variables in $\q$ subject to involutions, the artificial viscosity specified above is not compatible and needs to be properly modified. For more details on compatible artificial viscosity, see Section~\ref{sec.results} below.   

\subsection{Comparison with classical DG schemes}

In the case of a standard RK-DG scheme \cite{cbs0,cbs3,cbs4} one makes the same ansatz for $\uh$ as above in \eqref{eqn.uh}, and the weak form of the PDE reads  
\begin{equation}
	\int \limits_{T_k} \phi_i \, \partial_t \uh \, d\x + \int \limits_{T_k} \phi_i \, \nabla \cdot \f\left( \uh \right) \, d\x = 0, \qquad 
	\forall \phi_i \in \Uh.
	\label{eqn.dg0} 
\end{equation}	 
But since $\uh$ and therefore also $\f(\uh)$ is discontinuous at element boundaries, one uses a numerical flux $\hat{\f}(\uh^-,\uh^+)$, with boundary-extrapolated states $\uh^-$ and $\uh^+$ from within the cell and within the neighbor cell, to obtain a unique flux on the element boundary. The standard DG scheme then reads 
\begin{equation}
	\int \limits_{T_k} \phi_i \, \partial_t \uh \, d\x 	
	- \int \limits_{\mathring{T}_k} \nabla \phi_i \cdot \f\left( \uh \right) \, d\x 
	+ \int \limits_{\partial T_k} \phi_i \, \hat{\f}(\uh^-,\uh^+) \cdot \n \, dS 
	= 0, \quad 
	\forall \phi_i \in \Uh,
	\label{eqn.dgw} 
\end{equation}	
with $\mathring{T}_k = T_k \backslash \partial T_k$ the interior of the cell, or, equivalently in \textit{fluctuation form}, after undoing integration by parts of the second term 
\begin{equation}
	\int \limits_{T_k} \phi_i \, \partial_t \uh \, d\x 	
	+ \int \limits_{\mathring{T}_k} \phi_i \nabla \cdot \f\left( \uh \right) \, d\x 
	+ \int \limits_{\partial T_k} \phi_i \, \left( \hat{\f}(\uh^-,\uh^+) -  \f\left( \uh^- \right) \right) \cdot \n \, dS
	= 0, \quad 
	\forall \phi_i \in \Uh,
	\label{eqn.dgs} 
\end{equation}	
which is completely different from our new CG-DG scheme \eqref{eqn.cgdg} based on a globally continuous flux field. In particular note that the discrete flux field in \eqref{eqn.dgw} is in general \textit{not} continuous, since it is $\f(\uh)$ inside the cell and $\hat{\f}(\uh^-,\uh^+)$ on the boundary. This important difference is illustrated in Figure~\ref{fig.comparison} by means of a simple one-dimensional example in terms of the underlying data representations, in which a given function $q(x)=x^{-2}\sin(2\pi x^2)$ is first approximated in the discontinuous DG data representation $u_h \in \Uh$ via element-local $L^2$ projection of $q(x)$, and from $u_h$ the globally continuous representation $\tilde{w}_h \in \Wh$ is then obtained via the global $L^2$ projection \eqref{eqn.L2proj.dof}, so that for the linear flux $f(q)=q$ the function $u_h$ coincides inside each element with the discrete flux field of a standard DG scheme, while $\tilde{w}_h$ coincides with the globally continuous discrete flux field $\tilde{f}_h$ of the new CG-DG method.  

\begin{figure}[!h]
	\centering
	\includegraphics[trim=50 20 50 20,clip,width=0.99\textwidth]{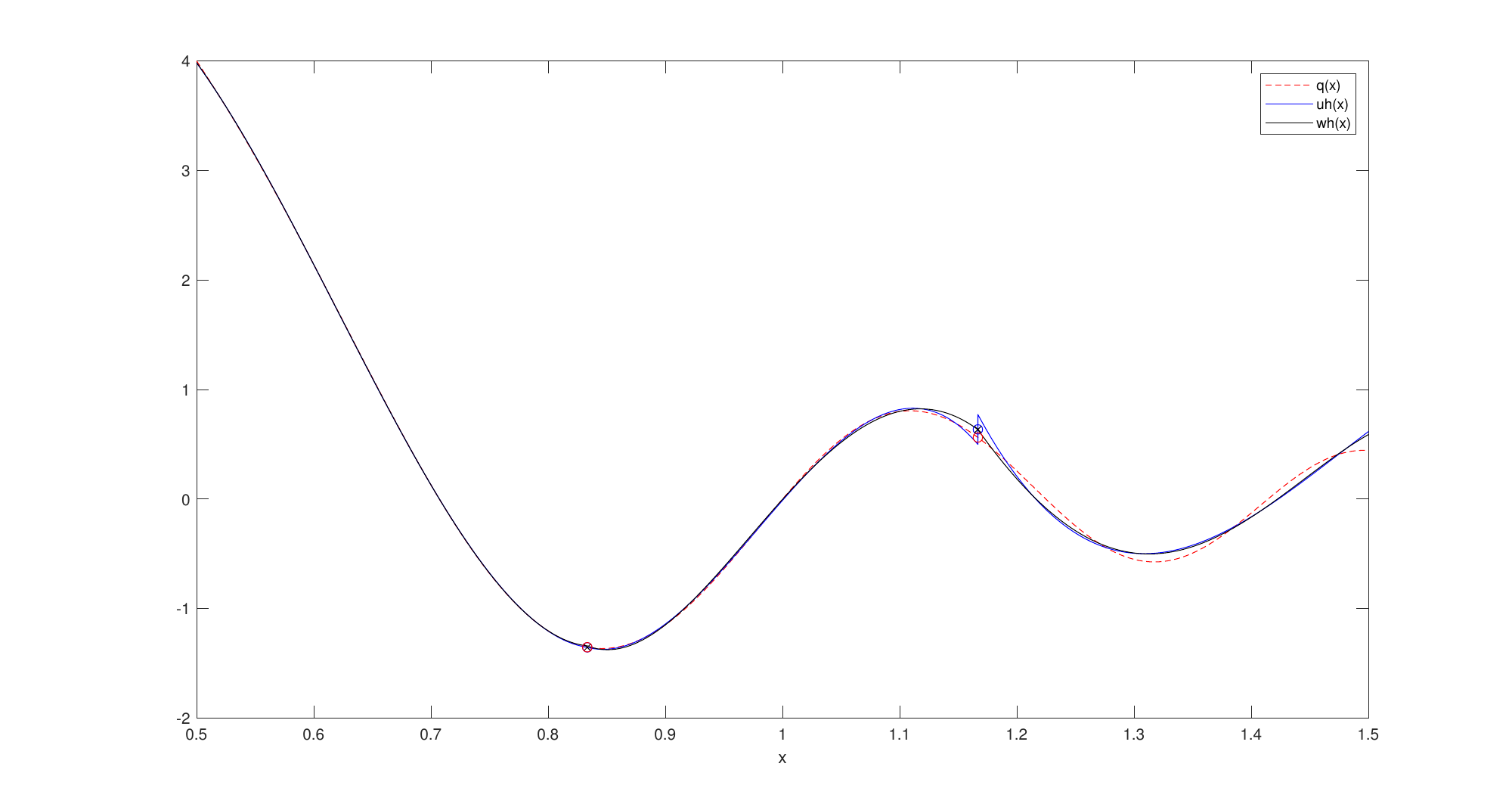}   
    \begin{tabular}{lr}        
\includegraphics[width=0.47\textwidth]{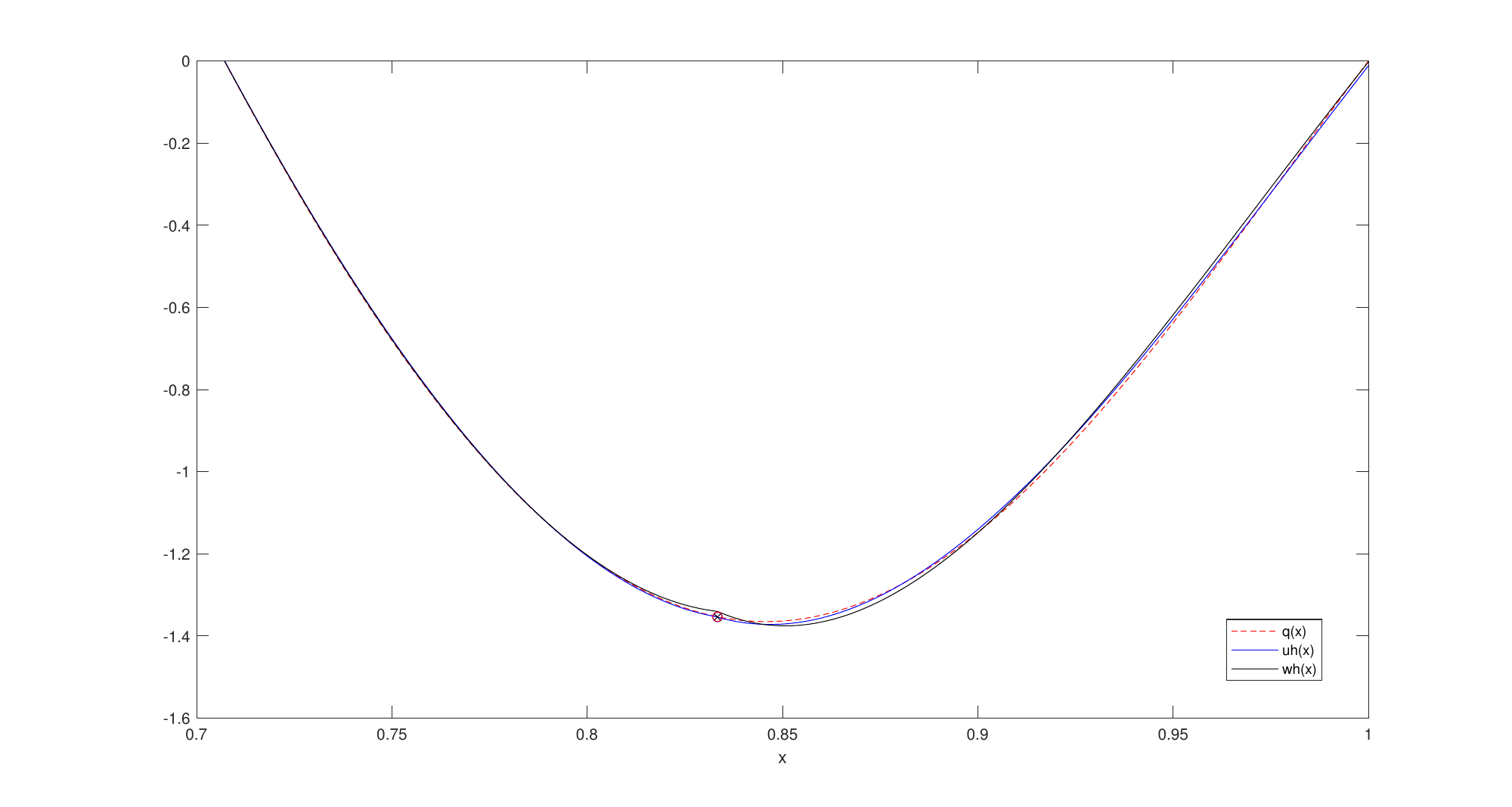}
    & 
    \includegraphics[clip,width=0.47\textwidth]{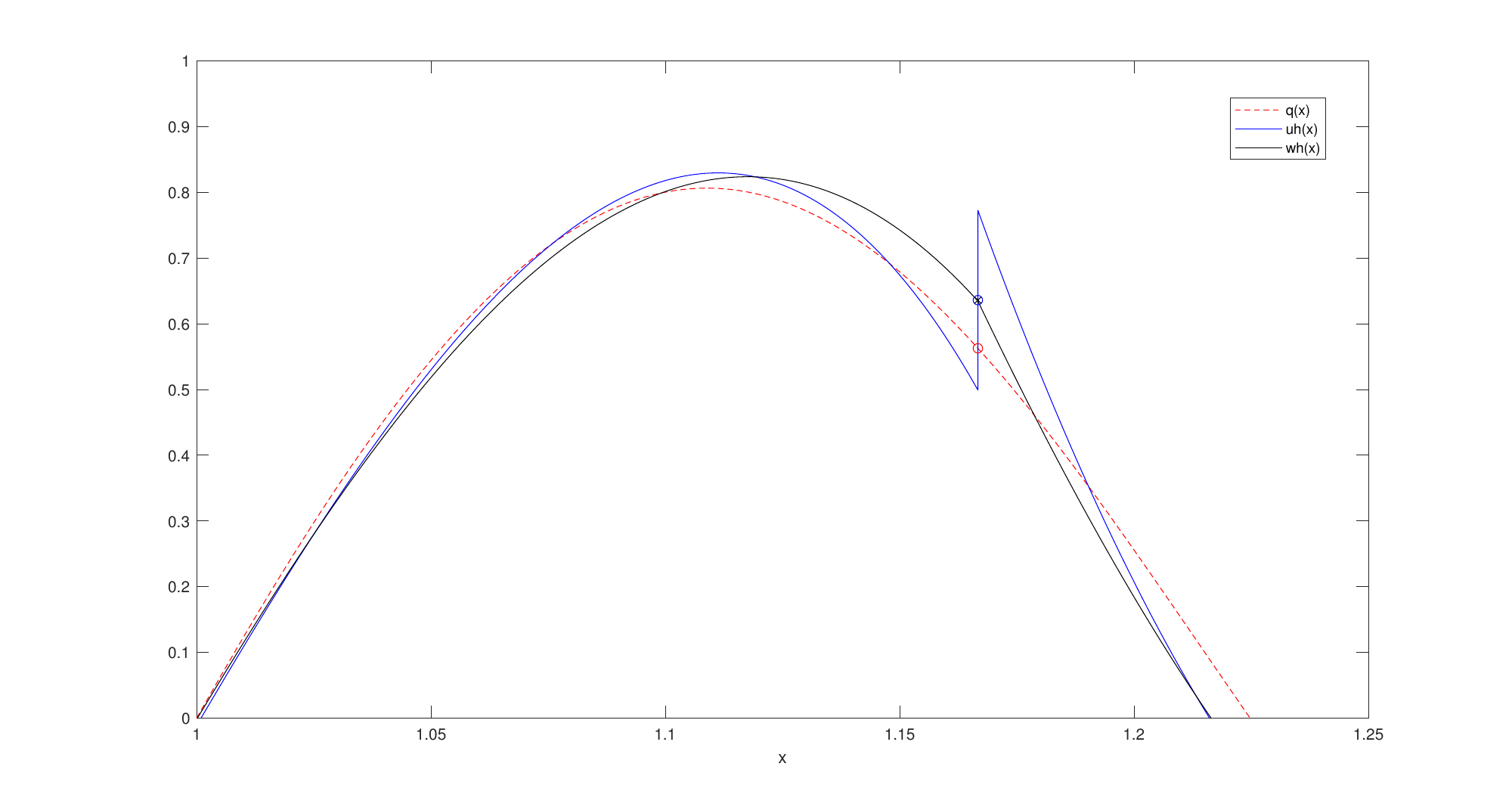}        
    \end{tabular}
	\caption{Original function $q(x) = x^{-2} \sin(2\pi x^2)$, its DG approximation $u_h(x) \in \Uh$ obtained via element-local $L^2$ projection of $q$ with $N=3$ and 3 elements on the interval $\Omega = [0.5, 1.5]$ and the continuous FEM approximation $\tilde{w}_h(x) \in \Wh$ obtained via $L^2$ projection \eqref{eqn.L2proj.dof} of $u_h$ (top)
    with two zoomed views (bottom). }
	\label{fig.comparison}
\end{figure}



\section{First-order Finite Volume discretization}
\label{sec:Lagrangian}
In this section, we explain how the method described in the previous Section~\ref{sec.method} "degenerates" when $N=0$, i.e. when piecewise constant polynomials are used in the cells and piecewise linear approximation for the flux. We hope this will help the reader to connect to more classical settings, and also to see more clearly when and where we depart from classical settings.
\subsection{Governing equations}
\label{ssec:Governing}
The computational domain $\Omega$ is a subset of the Euclidean space $\mathbb{R}^{\text{d}}$, where $\text{d}$ denotes the space dimension, which is paved with a collection of non-overlapping simplex elements. Throughout this paper, $\mathbf{x} \in \mathbb{R}^{\text{d}}$ denotes the vector position of a point and $t>0$ is the time. The system of conservation laws under consideration reads
\begin{equation}
\label{eq:syscl}
\partial_{t} \mathbf{q}+\nabla \cdot \mathbf{f}=0,
\end{equation}
where $\mathbf{q}=\mathbf{q}(\mathbf{x},t)$ is the vector of conservative variables and $\mathbf{f}=\mathbf{f}(\mathbf{q})$ the flux tensor. 

It is well known that systems of conservation laws of hyperbolic type might admit multiple discontinuous solutions, refer to \cite{Raviart1996}. To select the physically admissible solutions, we supplement them with an entropy inequality. To this end, we introduce the entropy, entropy flux pair $(\mathcal{E}, \mathbf{F})$ defined by
\begin{eqnarray}
 & & \mathcal{E}: \,\mathbf{q}\, \longmapsto \,\mathcal{E}(\mathbf{q})\;\text{is strictly concave}, \nonumber \\
 & & \mathbf{F}: \,\mathbf{q}\, \longmapsto \, \mathbf{F}(\mathbf{q})\; \text{is such that},\; \nabla \cdot \mathbf{F}=\partial_{\mathbf{q}} \mathcal{E} \cdot \nabla \cdot \mathbf{f}.
\end{eqnarray}
Note that here, we have chosen to work with a concave entropy so that in the case of gasdynamics we recover the physical entropy. Bearing this in mind, the smooth solutions of the system of conservation laws \eqref{eq:syscl} satisfy the supplementary conservation law
\begin{equation}
  \label{eq:entropcons}
  \partial_{t} \mathcal{E} +\nabla \cdot \mathbf{F}=0,
\end{equation}
whereas the admissible discontinuous, {\it i.e.}, weak, solutions have to fulfill the entropy inequality
\begin{equation}
  \label{eq:entropineq}
  \partial_{t} \mathcal{E} +\nabla \cdot \mathbf{F} \geq 0.
\end{equation}
So equipped with a concave entropy, the foregoing  system of conservation laws is assumed to be hyperbolic, that is the Jacobian matrix of the flux admits real eigenvalues and a set of linearly independent eigenvectors for all normal directions. To study the theoretical properties of the subsequent finite volume discretization, it is worth defining the entropy variables, $\mathbf{p}$, which are nothing but the gradient of the entropy with respect to the vector of conservative variables
\begin{equation}
  \mathbf{p}=\partial_{\mathbf{q}} \mathcal{E}.
\end{equation}
Following Tadmor \cite{Tadmor2003} we introduce the mapping
\begin{equation} \mathbf{p}\,\longmapsto \,\mathbf{q}(\mathbf{p}),\end{equation}
which is one-to-one by virtue of the concavity of entropy, {\it i.e.}, $\partial_{\mathbf{p}} \mathbf{q}=(\partial_{\mathbf{q}\mathbf{q}} \mathcal{E})^{-1} <0$. Namely, $\mathbf{H}=\partial_{\mathbf{p}} \mathbf{q}$ is nothing but the inverse of the Hessian matrix of the entropy and is thus symmetric negative definite. Finally, we introduce the potential functions
\begin{eqnarray}
 & & \phi(\mathbf{p})=\langle \mathbf{p},\mathbf{q}(\mathbf{p}) \rangle -\mathcal{E}(\mathbf{q}(\mathbf{p})), \nonumber \\
 & & \boldsymbol{\psi}(\mathbf{p})= \mathbf{f}^{\top}(\mathbf{q}(\mathbf{p})) \mathbf{p} -\mathbf{F}(\mathbf{q}(\mathbf{p})),
\end{eqnarray}
where $\langle .,. \rangle$ denotes the usual Euclidean inner product. 
\subsection{Finite Volume discretization}
\label{ssec:FVdiscretization}
We describe the cell-centered Finite Volume discretization of the systems of conservation laws \eqref{eq:syscl}. To ease the notation we describe the two-dimensional discretization knowing that the three-dimensional generalization does not present any difficulties.  
\subsubsection{Grid notation}
\label{sssec:Notation}
The two-dimensional Euclidean space is equipped with the direct orthonormal basis $(\mathbf{e}_x,\mathbf{e}_y)$ which is supplemented by $\mathbf{e}_z=\mathbf{e}_x \times \mathbf{e}_y$. The computational domain, $\Omega$, is paved by means of a collection of non-overlapping conformal triangles denoted $\omega_c$ and thus $\Omega=\cup_c \omega_c$. The triangular cell $\omega_c$ is characterized by the set of its vertices $\mathcal{P}(c)$. The generic vertex also named point is denoted using the label $p$ and its vector position is $\mathbf{x}_p$. In the counterclockwise ordered list of points of cell $\omega_c$, the vertex $p^{+}$ is the next vertex with respect to $p$ and $p^{-}$ is the previous one. The midpoints of $[p^{-},p]$ and $[p,p^{+}]$ are respectively $p^{-\frac{1}{2}}$ and $p^{+\frac{1}{2}}$. The quadrangle obtained joining successively the cell centroid $\mathbf{x}_c$ to $\mathbf{x}_{p^{-\frac{1}{2}}}$, $\mathbf{x}_p$, $\mathbf{x}_{p^{+\frac{1}{2}}}$ and the centroid defines the subcell $\omega_{pc}$, refer to Fig.~\ref{fig:trigrid}-(a). The set of subcells of a given cell constitutes a partition of this cell, that is $\omega_c=\cup_{p \in \mathcal{P}(c)} \omega_{pc}$. Gathering the subcells surrounding the generic vertex $p$ we define the dual cell, refer to Fig.~\ref{fig:trigrid}-(b)
\begin{equation}
  \label{eq:dualcell}
  \omega_p=\bigcup_{c \in \mathcal{C}(p)} \omega_{pc},
\end{equation}
where $\mathcal{C}(p)$ is the set of cells sharing vertex $p$.

\begin{figure}[htbp]
  \begin{center}
  \subfigure[]
    {\includegraphics[width=0.48\textwidth]{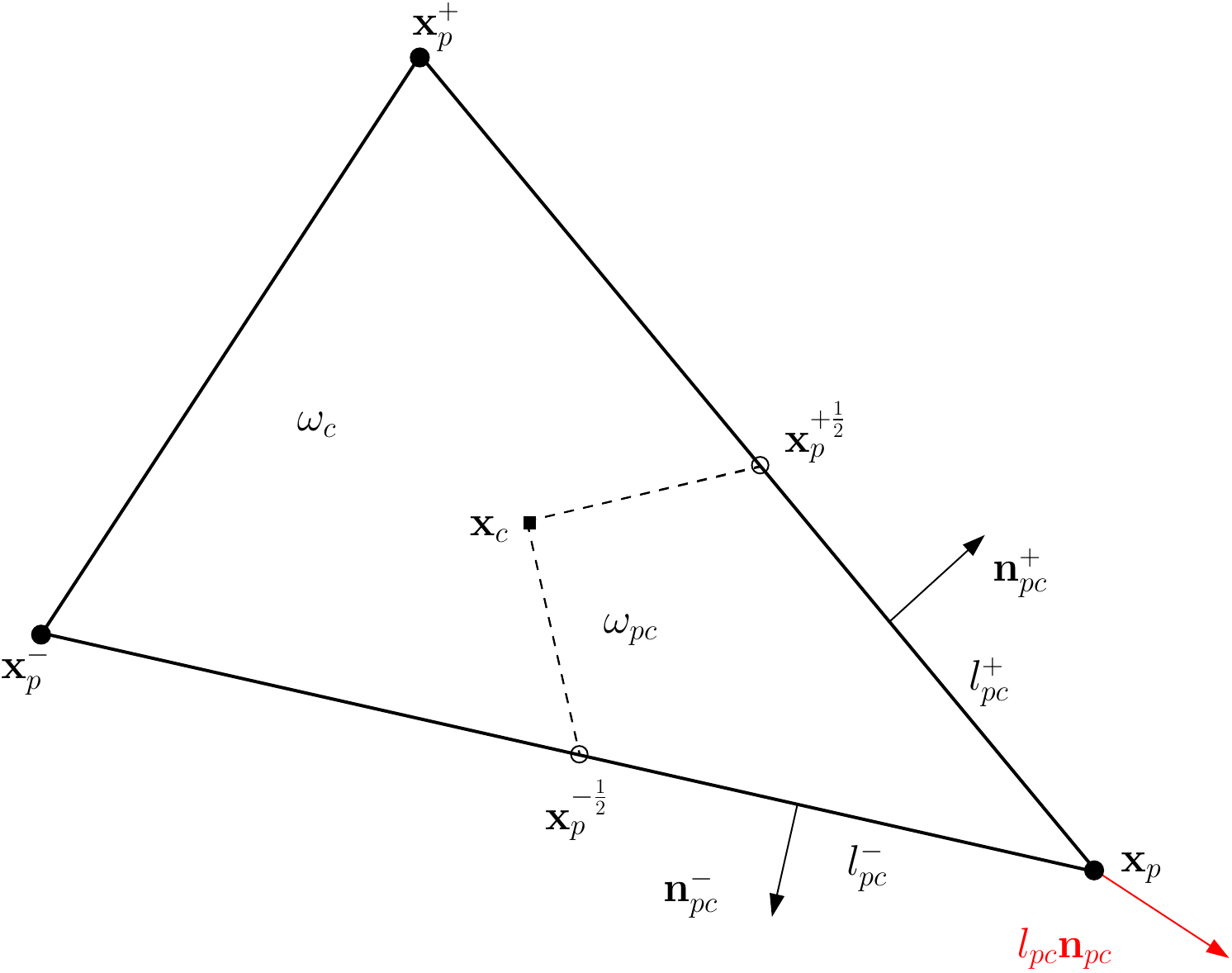}}
\subfigure[]{\includegraphics[width=0.48\textwidth]{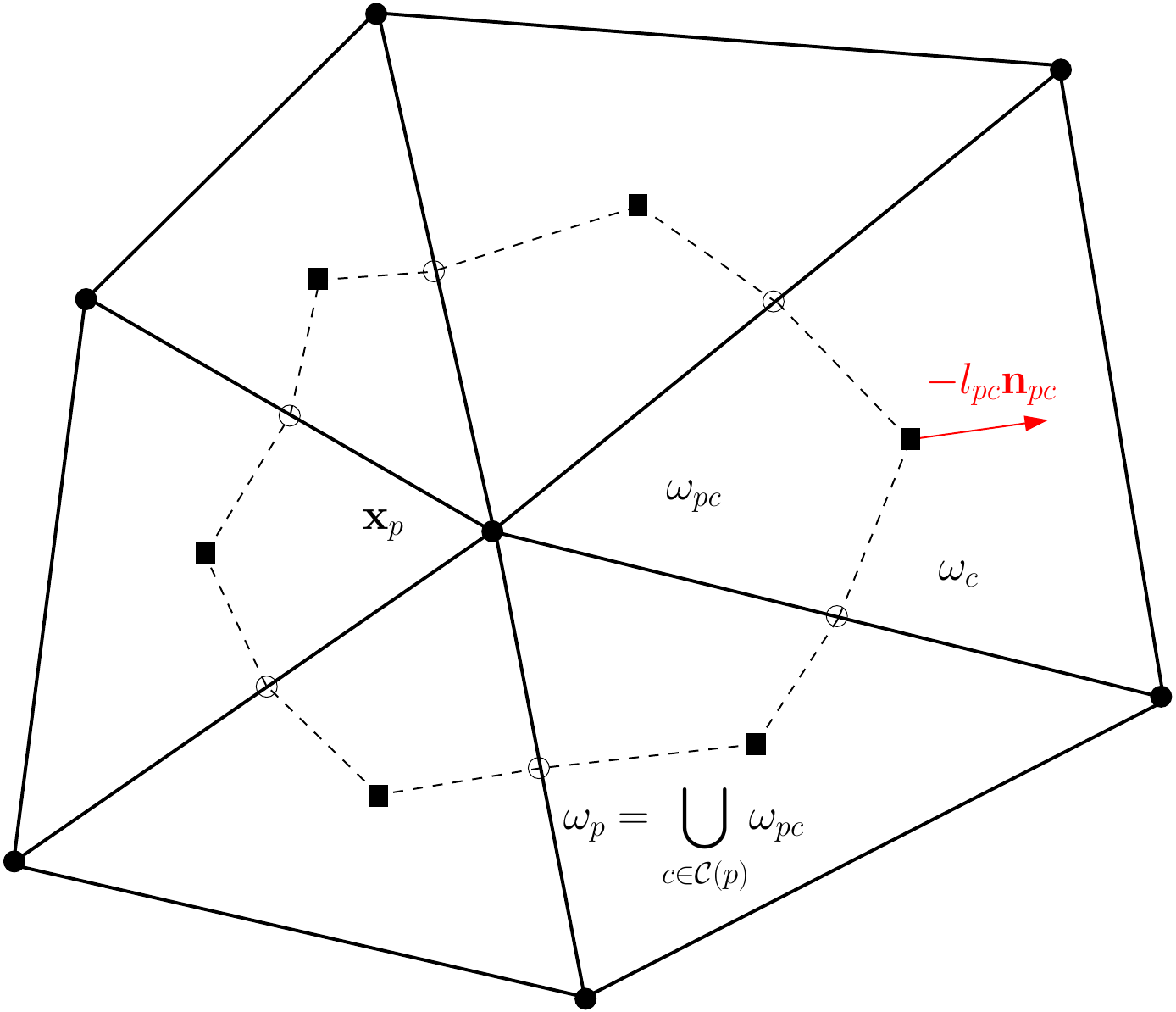}}
\end{center}
  \caption{Fragments of a triangular grid: (a) geometric quantities attached to the triangular cell $\omega_c$; (b) dual cell attached to point $p$.}
  \label{fig:trigrid}
\end{figure}

The unit outward normals to the faces of $\omega_c$ impinging at $p$ are respectively $\mathbf{n}_{pc}^{-}$ and $\mathbf{n}_{pc}^{+}$. We also introduce the lengths of the corresponding ``half-faces'' $l_{pc}^{-}=\frac{1}{2}|\mathbf{x}_{p}-\mathbf{x}_{p^{-}}|$ and $l_{pc}^{+}=\frac{1}{2}|\mathbf{x}_{p^{+}}-\mathbf{x}_{p}|$. The volume of the triangular cell $\omega_c$ might be expressed in terms of the position vectors of its vertices. This allows us to define the corner normal vector related to point $p$ and cell $c$ as the gradient of the cell volume $\omega_c$ with respect to the vector position $\mathbf{x}_p$
\begin{equation}
  \label{eq:defcornervec}
  l_{pc}\mathbf{n}_{pc}=\frac{\partial |\omega_c|}{\partial \mathbf{x}_p},
\end{equation}
where $|\mathbf{n}_{pc}|^{2}=1$. This geometrical object is fundamental for deriving staggered and cell-centered discretization of Lagrangian gasdynamics over moving grids, cf. \cite{LMR2016} and the references therein. Computing the cell volume gradient with respect to $\mathbf{x}_{p}$ yields the analytical expression of the corner normal
\begin{equation}
  \label{eq:fromcornervec}
  l_{pc}\mathbf{n}_{pc}=\frac{1}{2} [\mathbf{x}_{p^{+}}-\mathbf{x}_{p^{-}} ] \times \mathbf{e}_z=l_{pc}^{-}\mathbf{n}_{pc}^{-}+l_{pc}^{+}\mathbf{n}_{pc}^{+}.
\end{equation}  
This shows that the corner normal is the sum of the outward normals to the ``half-faces'' impinging at point $p$, refer to Fig.~\ref{fig:trigrid}-(a). We observe that the corner normal fulfills the fundamental geometrical identity
\begin{equation}
  \label{eq:geomiden}
  \sum_{p \in \mathcal{P}(c)} l_{pc}\mathbf{n}_{pc}=0.
\end{equation}
This result, by virtue of \eqref{eq:fromcornervec}, is a direct consequence of the fact that the sum of the outward normals to a closed contour is equal to zero.
In \cite{HTCLagrangeGPR} it was shown that the corner normal divided by the cell area $|\omega_c|$ is equal to the gradient of classical continuous P1 finite element basis functions, which establishes the link with the general method presented in the previous section for the special case $N=0$.

The corner vector is also the cornerstone to define discrete mathematical operators such as the discrete divergence operator following the approach introduced in \cite{Shashkov1996} in the more general framework of mimetic finite difference methods. For any vector field $\mathbf{q}:=\mathbf{q}(\mathbf{x}),\;\text{for}\; \mathbf{x} \in \omega_c$ we define the cell-centered discrete divergence 
\begin{equation}
  \label{eq:discretedivprim}
  \text{DIV}_c(\mathbf{q})=\frac{1}{|\omega_c|}\sum_{p \in \mathcal{P}(c)} \langle l_{pc}\mathbf{n}_{pc}, \mathbf{q}_{p} \rangle,
\end{equation}
where $\mathbf{q}_p=\mathbf{q}(\mathbf{x}_{p})$. Identity \eqref{eq:geomiden} implies that the divergence of constant vectors is equal to zero. Note that the discrete gradient operator over $\omega_c$ might be defined in a similar manner. The mathematical properties of the discrete operators and their application to construct numerical methods for solving different classes of partial differential equations are described in the review paper \cite{Lipnikov2014}. Now, considering the dual cell $\omega_p$, which is the union of the subcells attached to point $p$, refer to Fig.~\ref{fig:trigrid}, we define the discrete divergence and gradient operators of a vector as follows
\begin{align}
& \text{DIV}_p(\mathbf{q})=-\frac{1}{|\omega_p|}\sum_{c \in \mathcal{C}(p)} \langle l_{pc} \mathbf{n}_{pc}, \mathbf{q}_c \rangle,  \label{eq:discretedivdual} \\
&\text{GRAD}_p(\mathbf{q})=-\frac{1}{|\omega_p|}\sum_{c \in \mathcal{C}(p)} l_{pc} \mathbf{q}_{c} \otimes \mathbf{n}_{pc}. \label{eq:discretegraddual}
\end{align}
Here, $\otimes$ denotes the dyadic product and $\mathbf{q} \otimes \mathbf{n}$ is the second-order tensor whose $ij$ component in the Cartesian basis is $(\mathbf{q} \otimes \mathbf{n})_{ij}=q_i n_j$. We point out that we have used the fact that $-l_{pc}\mathbf{n}_{pc}$ is the outward normal to the contour of the dual cell $\omega_p$. 
\subsubsection{Semi-discrete node-based Finite Volume discretization}
\label{sssec:sdfcl}
Let us define the cell-averaged quantity
\begin{equation}\mathbf{q}_c=\frac{1}{|\omega_c|} \int_{\omega_c} \mathbf{q}\,\mathrm{d}v.\end{equation}
Integrating the system of conservation laws \eqref{eq:syscl} over $\omega_c$ and applying the Green-Gauss theorem we get
\begin{equation}|\omega_c| \frac{\mathrm{d} \mathbf{q}_c}{\mathrm{d}t} +\int_{\partial \omega_c} \mathbf{f}(\mathbf{q}) \mathbf{n}\,\mathrm{d}s=0.\end{equation}
Employing a node-based quadrature formula for the surface integral in the foregoing equation leads us to the following node-based approximation of the integral flux
\begin{equation}\int_{\partial \omega_c} \mathbf{f}(\mathbf{q})\mathbf{n}\,\mathrm{d}s \approx \sum_{p \in \mathcal{P}(c)} l_{pc} \mathbf{f}_{p} \mathbf{n}_{pc},\end{equation}
where $\mathbf{f}_{p}$ is the tensor flux approximation at point $p$. {\it A priori}, $\mathbf{f}_{p}$ should depend on all the states surrounding point $p$, namely $\mathbf{f}_p := \mathbf{f}_{p}(\mathbf{q}_1,\dots,\mathbf{q}_{|\mathcal{C}(p)|})$. Moreover, we assume that this multidimensional tensor flux approximation is consistent \cite{Abgrall2026}, that is when $\mathbf{q}_c=\mathbf{q}$ for all $c \in \mathcal{C}(p)$ then $\mathbf{f}_{p}(\mathbf{q},\dots,\mathbf{q})=\mathbf{f}(\mathbf{q})$. Combining the previous definitions, we arrive at the node-based Finite Volume discretization of the system of conservation laws \eqref{eq:syscl}
\begin{equation}
  \label{eq:semidis}
  |\omega_c| \frac{\mathrm{d} \mathbf{q}_c}{\mathrm{d} t}+\sum_{p \in \mathcal{P}(c)} l_{pc} \mathbf{f}_{p} \mathbf{n}_{pc}=0.
\end{equation}
We observe that this scheme is conservative by construction by virtue of the geometrical identity \eqref{eq:geomiden} and since the corner normal divided by the cell volume is the gradient of the continuous finite element P1 basis functions the above scheme is identical with \eqref{eqn.cgdg} for the special case $N=0$ if mass lumping is used. In addition, it is worth noting that the present Finite Volume scheme can be recast in the form of the following Residual Distribution scheme
\begin{equation}
  \label{eq:semidisRD}
  |\omega_c| \frac{\mathrm{d} \mathbf{q}_c}{\mathrm{d} t}+\sum_{p \in \mathcal{P}(c)} \boldsymbol{\Phi}_{c}^{p}=0,
\end{equation}
where the corner fluctuation $\boldsymbol{\Phi}_{c}^{p}$ reads
\begin{equation}\boldsymbol{\Phi}_{c}^{p}=l_{pc} \left [ \mathbf{f}_{p} -\mathbf{f}(\mathbf{q}_c) \right ] \mathbf{n}_{pc}.\end{equation} 
\subsubsection{Structural form of the node-based multidimensional flux}
We derive the structural form of the node-based multidimensional flux, $\mathbf{f}_p$, so that a semi-discrete entropy inequality is ensured, except at nodes where $\text{GRAD}_{p} (\mathbf{p})=0$, see below. First, the semi-discrete equation satisfied by the entropy of the cell average results from taking the dot product of \eqref{eq:semidisRD} by $\mathbf{p}_c$
\begin{equation}|\omega_c| \frac{\mathrm{d} \mathcal{E}(\mathbf{q}_c)}{\mathrm{d} t} +\sum_{p \in \mathcal{P}(c)} \langle \boldsymbol{\Phi}_{c}^{p}, \mathbf{p}_c \rangle =0.\end{equation}
Let us develop the term between brackets in the foregoing equation expressing the corner fluctuation in terms of the nodal flux
\begin{equation}
  \label{eq:entropfluctuat}
  \langle \boldsymbol{\Phi}_{c}^{p} ,\mathbf{p}_c\rangle=\langle l_{pc} \left [ \mathbf{f}_{p} -\mathbf{f}(\mathbf{q}_c) \right ] \mathbf{n}_{pc} , \mathbf{p}_c \rangle= \langle l_{pc} \mathbf{f}_{p} \mathbf{n}_{pc}, \mathbf{p}_c \rangle - \langle l_{pc} \mathbf{f}^{\top}(\mathbf{q}_c) \mathbf{p}_c, \mathbf{n}_{pc} \rangle.
\end{equation}
Before proceeding further, let us recall some useful mathematical notation. For any second-order tensor $\mathbf{A}$ and any pair of vectors $(\mathbf{n},\mathbf{p})$, one has
\begin{equation}
\langle \mathbf{A}\mathbf{n}, \mathbf{p}\rangle
= \Tr{(\mathbf{A}\mathbf{n} \otimes \mathbf{p})}
= \Tr{(\mathbf{A}(\mathbf{n} \otimes \mathbf{p}))}
= \mathbf{A} : (\mathbf{p} \otimes \mathbf{n}).
\end{equation}
Here, $\Tr{()}$ denotes the trace operator and $:$ the inner product of second-order tensors. These quantities are defined as follows \cite{Gurtin2009}: $\Tr{(\mathbf{A})}=\sum_{i=1\dots \mathrm{d}} A_{ii}$, where $A_{ij}$ denotes the $ij$ component of $\mathbf{A}$ in the Cartesian basis; and, for any second-order tensor $\mathbf{B}$, the inner product between $\mathbf{A}$ and $\mathbf{B}$ is given by $\mathbf{A} : \mathbf{B}=\Tr{(\mathbf{A}^{\top} \mathbf{B})}$.

Bearing this in mind, \eqref{eq:entropfluctuat} turns into
\begin{equation}\langle \boldsymbol{\Phi}_{c}^{p} ,\mathbf{p}_c\rangle=l_{pc} \mathbf{f}_{p} : (\mathbf{p}_c \otimes \mathbf{n}_{pc}) -\langle l_{pc} (\boldsymbol{\psi}_c +\mathbf{F}_{c}),\mathbf{n}_{pc} \rangle,\end{equation}
by virtue of the definition of the potential function $\boldsymbol{\psi}$. Summing the foregoing expression over all the cells surrounding point $p$ we get
\begin{align}
  \label{eq:sumentr}
  \sum_{c \in \mathcal{C}(p)} \langle \boldsymbol{\Phi}_{c}^{p} ,\mathbf{p}_c\rangle=&\mathbf{f}_{p} : \sum_{c \in \mathcal{C}(p)} l_{pc} (\mathbf{p}_c \otimes \mathbf{n}_{pc}) -\sum_{c \in \mathcal{C}(p)} \langle l_{pc} \boldsymbol{\psi}_{c}, \mathbf{n}_{pc} \rangle  -\sum_{c \in \mathcal{C}(p)} \langle l_{pc} \mathbf{F}_{c},\mathbf{n}_{pc} \rangle \nonumber \\
  = &-|\omega_p| \left [ \mathbf{f}_{p} : \text{GRAD}_{p} (\mathbf{p}) -\text{DIV}_{p}(\boldsymbol{\psi}) \right ] -\sum_{c \in \mathcal{C}(p)} \langle l_{pc} \mathbf{F}_c ,\mathbf{n}_{pc} \rangle.
\end{align}
Here, we have introduced the discrete gradient and divergence operators defined over the dual cell $\omega_p$ according to \eqref{eq:discretedivdual} and \eqref{eq:discretegraddual}.
Finally, we observe that the entropy consistency condition at the node \cite{Abgrall2026}
\begin{equation}
  \label{eq:entropcond}
  \sum_{c \in \mathcal{C}(p)} \langle \boldsymbol{\Phi}_{c}^{p}, \mathbf{p}_{c} \rangle \leq -\sum_{c \in \mathcal{C}(p)} \langle l_{pc} \mathbf{F}_{c} , \mathbf{n}_{pc} \rangle,
\end{equation}
is fulfilled if and only if
\begin{equation}
  \label{eq:tadmor}
  \mathbf{f}_{p}:\text{GRAD}_{p} (\mathbf{p}) -\text{DIV}_{p}(\boldsymbol{\psi}) \geq 0.
\end{equation}
The latter inequality can be interpreted as the multidimensional extension of Tadmor's condition \cite{Tadmor2003} for defining an entropy-stable flux in the one-dimensional case. In order to satisfy \eqref{eq:tadmor}, we first introduce the following {\it ansatz} for the multidimensional flux:
\begin{equation}
  \label{eq:multidimensional_flux}
  \mathbf{f}_{p}= \left \{ \mathbf{f} \right \}_{p}-\frac{1}{2} \varepsilon_p h_{p} a_p \mathbf{H}_{p} \text{GRAD}_{p} (\mathbf{p}).
\end{equation}
Here, $\left \{ \right \}_{p}$ denotes some average defined at point $p$ of the cell values of the tensor flux, $\varepsilon_p >0$ is a viscosity coefficient, $h_p$ is a length at node $p$, $a_p>0$ is a characteristic velocity at point $p$ and $\mathbf{H}_{p}$ is the inverse of the Hessian matrix of the entropy evaluated at point $p$. With this expression of the multidimensional flux we arrive at
\begin{eqnarray}
\mathbf{f}_{p}:\text{GRAD}_{p} (\mathbf{p}) -\text{DIV}_{p}(\boldsymbol{\psi}) & = & \left \{ \mathbf{f} \right \}_{p}:\text{GRAD}_{p} (\mathbf{p})-\text{DIV}_{p}(\boldsymbol{\psi}) \nonumber \\
 & & -\frac{1}{2}\varepsilon_p h_p a_p \mathbf{H}_p \text{GRAD}_{p} (\mathbf{p}) : \text{GRAD}_{p} (\mathbf{p}).
\end{eqnarray}
The last term on the right-hand side of the foregoing equation is positive by virtue of the negative definiteness of $\mathbf{H}_p$. Indeed, $-\mathbf{H}_p$ is symmetric positive definite and we can compute its square root, $\mathbf{S}_p$, which satisfies $\mathbf{S}_p^{2}=-\mathbf{H}_p$ and then
\begin{eqnarray}
  -\mathbf{H}_p\text{GRAD}_{p} (\mathbf{p}):\text{GRAD}_{p} (\mathbf{p}) & = & \Tr{\left [\text{GRAD}_{p}^{\top} (\mathbf{p}) \mathbf{S}_p \mathbf{S}_p \text{GRAD}_{p} (\mathbf{p})\right]} \nonumber \\
  & = & \Tr{\left [(\mathbf{S}_p \text{GRAD}_{p} (\mathbf{p}))^{\top} \mathbf{S}_p \text{GRAD}_{p} (\mathbf{p})\right ]} \nonumber \\
  & = & (\mathbf{S}_p \text{GRAD}_{p} (\mathbf{p})) : (\mathbf{S}_p \text{GRAD}_{p} (\mathbf{p})) \geq 0.
\end{eqnarray}
Finally, to achieve the control of the sign of the left-hand side of the multidimensional Tadmor's condition, we supplement the expression of the multidimensional nodal flux as follows
\begin{equation}
  \label{eq:multidimensional_fluxbis}
  \mathbf{f}_{p}= \left \{ \mathbf{f} \right \}_{p}-\frac{1}{2} \varepsilon_p h_{p} a_p \mathbf{H}_{p} \text{GRAD}_{p} (\mathbf{p})+\alpha_p\mathbf{H}_{p} \text{GRAD}_{p} (\mathbf{p}). 
\end{equation}
Here, the coefficient $\alpha_p$ is chosen so that the latter supplementary term cancels exactly with $\left \{ \mathbf{f} \right \}_{p}:\text{GRAD}_{p} (\mathbf{p})-\text{DIV}_{p}(\boldsymbol{\psi})$
\begin{equation}\alpha_p=-\frac{\left \{ \mathbf{f} \right \}_{p}:\text{GRAD}_{p} (\mathbf{p})-\text{DIV}_{p}(\boldsymbol{\psi})}{\mathbf{H}_p\text{GRAD}_{p} (\mathbf{p}):\text{GRAD}_{p} (\mathbf{p})}.\end{equation}
If $\text{GRAD}_{p} (\mathbf{p})=0$, we set $\alpha_p=0$, since in this case the left-hand side of \eqref{eq:tadmor} reduces to $-\text{DIV}_{p}(\boldsymbol{\psi})$, which does not depend on the nodal flux, so that \eqref{eq:tadmor} cannot in general be enforced at such nodes.

\section{Properties of the method}
\label{sec.prop}
Now we come back to the general case described in Section \ref{sec.method}.
\subsection{Basic properties}

We first recall two basic properties of the primary discrete nabla operator, as already discussed in \cite{CompatibleDG1}. The continuous finite element space $\mathcal{W}_h^{N+1}$ is \textit{globally continuous}, hence it is obvious that all \textit{tangential derivatives}, i.e. the tangential components of the gradient, are \textit{continuous} across elements. We also stress again that the two ansatz spaces $\Uh$ and $\Wh$ are deliberately chosen in such a way that the DG space $\mathcal{U}_h^N$ is rich enough to represent all first derivatives (the full gradient) of a function represented in the finite element space $\mathcal{W}_h^{N+1}$ \textit{exactly}. We can therefore state the following 
\begin{proposition}
	\label{prop.td} 
	For any discrete scalar field $\Zh \in \mathcal{W}_h^{N+1}$ and any pair of elements $T_b$ and $T_c$  the jump of the tangential derivatives vanishes across the common edge / face $\partial T_{bc} = T_b \cap T_c$, i.e.  
	\begin{equation}
		 \left( \nabla \Zh^+ - \nabla \Zh^- \right) \cdot \mathbf{t} = 0, 
 	\end{equation} 
	with $\mathbf{t}$ all vectors tangential to $\partial T_{bc}$ and $\Zh^+$, $\Zh^-$ the boundary-extrapolated values of $\Zh$ seen from elements $T_b$ and $T_c$, respectively.  
\end{proposition} 
\begin{proof}
	Due to the \textit{global continuity} of $\Zh \in \mathcal{W}_h^{N+1}$ one has clearly  
	\begin{equation} \Zh^+ = \Zh^-, \quad \textnormal{ or, equivalently } \quad \Zh^+ - \Zh^- = 0, \qquad \forall \mathbf{x} \in \partial T_{bc}.\end{equation}
	It is therefore obvious that for all vectors $\mathbf{t}$ tangential to $\partial T_{bc}$ one has 
	\begin{equation}\frac{\partial \Zh^+}{\partial \mathbf{t}} = \frac{\partial \Zh^-}{\partial \mathbf{t}}, \qquad \textnormal{ or, equivalently } \qquad 
	\left( \nabla \Zh^+ - \nabla \Zh^- \right) \cdot \mathbf{t} = 0, \qquad \forall \mathbf{x} \in \partial T_{bc}.
	\end{equation} 
\qed\end{proof}
As a direct \textit{consequence} of Proposition \ref{prop.td} we also immediately have the property that the \textit{normal component} of the curl of a vector field $\Ah \in \mathcal{W}_h^{N+1}$ is \textit{continuous} across elements. 
\begin{proposition}
	\label{prop.nd} 
	For any discrete vector field $\Ah \in \mathcal{W}_h^{N+1}$ and any pair of elements $T_b$ and $T_c$  the jump of the normal component of the curl vanishes across the common edge / face $\partial T_{bc} = T_b \cap T_c$, i.e.  
	\begin{equation}
		\left( \nabla \times \Ah^+ - \nabla \times \Ah^- \right) \cdot \mathbf{n} = 0,
	\end{equation} 
	with $\mathbf{n}$ the unit normal vector of $\partial T_{bc}$ pointing from $T_c$ to $T_b$ and $\Ah^+$, $\Ah^-$ the boundary-extrapolated values of $\Ah$ seen from elements $T_b$ and $T_c$, respectively.  
\end{proposition} 
\begin{proof}
	The proof becomes easier when using the Einstein index notation, i.e. here we will denote the components of the vector field $\Ah$ by $\tilde{A}_p$ and drop the subscript $h$ to ease notation. Furthermore, the components of the tangential vectors $\mathbf{t}$ will be denoted by $t_k$ and those of the unit normal vector $\mathbf{n}$ by $n_k$. 
	Due to the \textit{global continuity} of $\Ah \in \mathcal{W}_h^{N+1}$ one has clearly componentwise   
	\begin{equation} \tilde{A}_p^+ = \tilde{A}_p^-, \quad \textnormal{ or } \quad \tilde{A}_p^+ - \tilde{A}_p^- = 0, \qquad \forall \mathbf{x} \in \partial T_{bc}.\end{equation}
	It is therefore obvious that for all vectors $\mathbf{t}$ tangential to $\partial T_{bc}$ one has componentwise 
	\begin{equation}	\left( \partial_j \tilde{A}_p^+ - \partial_j \tilde{A}_p^- \right) t_j = 0, \qquad \forall \mathbf{x} \in \partial T_{bc}.
	\end{equation} 
	Hence, in general, the only non-vanishing component of the jump in the gradient of $\Ah$ at the element boundary $\partial T_{bc}$ is the jump of the derivative in the normal direction, i.e. 
	\begin{equation} \left( \partial_j  \tilde{A}_p^+ - \partial_j  \tilde{A}_p^- \right) = \partial_j \left( \tilde{A}_p^+ - \tilde{A}_p^- \right) = \partial_m \left( \tilde{A}_p^+ - \tilde{A}_p^- \right) n_m \, n_j. \end{equation}
	Using the above result and the usual fully antisymmetric Levi-Civita tensor $\epsilon_{ijk}$ in Einstein index notation it is now easy to see that 
	\begin{eqnarray}
	\left( \nabla \times \Ah^+ - \nabla \times \Ah^- \right) \cdot \mathbf{n} & = & 
	\epsilon_{ijk} \left( \partial_j \tilde{A}_k^+ - \partial_j \tilde{A}_k^- \right) n_i \nonumber \\
	& = & 
	\epsilon_{ijk} \left( \partial_m \tilde{A}_k^+ - \partial_m \tilde{A}_k^- \right) n_m n_i n_j = 0, 
	\end{eqnarray} 
	since the contraction of the anti-symmetric Levi-Civita tensor with the symmetric tensor $n_i n_j$ vanishes, i.e. $\epsilon_{ijk} n_i n_j = 0$. 
\qed\end{proof}

\subsection{Discrete Schwarz theorem and resulting discrete vector calculus identities}\label{sec.schwarz} 

As already shown in \cite{CompatibleDG1}, the two discrete nabla operators introduced before satisfy a discrete Schwarz theorem. To make this paper self-contained and for the reader's convenience, here we briefly summarize the main findings of \cite{CompatibleDG1}. The key feature of the two nabla operators is that they satisfy a discrete version of the Schwarz theorem. Then, as a direct consequence, the two discrete vector calculus identities follow automatically. 

\begin{theorem}
	\label{lemma.ds}
 The two discrete nabla operators \eqref{eqn.primary.nabla} and \eqref{eqn.dual.nabla} satisfy the following discrete Schwarz theorem
 \begin{equation}
 	  (\tilde{\partial}_j)_p^c  \left(\partial_k \right)_c^q Z_q = (\tilde{\partial}_k)_p^c  \left(\partial_j\right)_c^q Z_q.
 	  \label{eqn.ds.scalar} 
 \end{equation}
\end{theorem}
\begin{proof}
    For clarity we use the Einstein index notation. From the definitions \eqref{eqn.primary.nabla} and \eqref{eqn.dual.nabla} we have 
\begin{equation} 
\begin{split}
	 (\tilde{\partial}_j)_p^c \, (\partial_k)_c^q Z_q &= - \int \limits_{\Omega} \partial_j \psi_p \, \partial_k \Zh \dx \\&= \int \limits_{\Omega \backslash \mathcal{S}_h} \psi_p \partial_j \partial_k \Zh \dx 
	+ \int \limits_{\mathcal{S}_h} \psi_p  \left( \partial_k \Zh^+  - \partial_k \Zh^- \right)  n_j \, dS,
	\label{eqn.auxeq}
    \end{split}
\end{equation}
with $\mathcal{S}_h$ the skeleton of the mesh. 
Since $\Zh$ is smooth inside each simplex $T_k$, the classical continuous Schwarz theorem obviously holds, $\partial_j \partial_k \Zh = \partial_k \partial_j \Zh$, and due to the global continuity of $\Zh$ across element edges / faces $\nabla \Zh$ is allowed to jump across $\mathcal{S}_h$ only in the normal direction and not in the tangential directions, hence also the jump term, which can be rewritten as  
\begin{equation} \left( \partial_k \Zh^+  - \partial_k \Zh^- \right)  n_j = \left( \partial_m \Zh^+ - \partial_m \Zh^- \right) n_m n_k n_j = \left( \partial_j \Zh^+  - \partial_j \Zh^- \right)  n_k \end{equation} is clearly symmetric in the indices $j$ and $k$. 
Hence, the discrete Schwarz theorem follows: 
\begin{eqnarray} 
	&& (\tilde{\partial}_j)_p^c \, (\partial_k)_c^q Z_q = 
	\int \limits_{\Omega \backslash \mathcal{S}_h} \psi_p \partial_j \partial_k \Zh \dx 
	+ \int \limits_{\mathcal{S}_h} \psi_p  \left( \partial_k \Zh^+  - \partial_k \Zh^- \right)  n_j \, dS \nonumber \\ 
	&& =   
	\int \limits_{\Omega \backslash \mathcal{S}_h} \psi_p \partial_k \partial_j \Zh \dx 
	+ \int \limits_{\mathcal{S}_h} \psi_p  \left( \partial_j \Zh^+  - \partial_j \Zh^- \right)  n_k \, dS 
	= (\tilde{\partial}_k)_p^c \, (\partial_j)_c^q Z_q. 
	\label{eqn.discrete.schwarz}
\end{eqnarray}     	 
\qed\end{proof}
As a consequence we have 
\begin{corollary}
	\label{corr.ds}
	The two discrete nabla operators \eqref{eqn.primary.nabla} and \eqref{eqn.dual.nabla}  satisfy a discrete Schwarz theorem also for discrete vector fields 
	\begin{equation}
		(\tilde{\partial}_j)_p^c  \left(\partial_k \right)_c^q A_{mq} = (\tilde{\partial}_k)_p^c  \left(\partial_j\right)_c^q A_{mq}.
		\label{eqn.ds.vector} 
	\end{equation}
\end{corollary}
\begin{proof}
   Replacing $Z_q$ by $A_{mq}$ in the proof of the previous Theorem~\ref{lemma.ds} leads immediately to the  result. 
\qed\end{proof}

\begin{theorem}
Given the two discrete nabla operators \eqref{eqn.primary.nabla} and \eqref{eqn.dual.nabla}, 
a discrete scalar potential $\Zh \in \mathcal{W}^{N+1}_h$ with degrees of freedom $Z_q \in \mathbb{R}$ and a discrete vector potential $\Ah \in \mathcal{W}^{N+1}_h$ with degrees of freedom $\A_q \in \mathbb{R}^3$ and representation
\begin{equation}
    \Zh = \psi_q  Z_q \quad \in \mathcal{W}^{N+1}_h, \qquad 
    \Ah = \psi_q \A_q \quad \in \mathcal{W}^{N+1}_h, 
    \label{eqn.disc.potential}
\end{equation}
then the following discrete vector calculus identities hold:
\begin{equation}
   { \tilde{\nabla}_p^c \times \left( \nabla_c^q Z_q \right) = 0,}
    \label{eqn.disc.rotgrad}
\end{equation}
and
\begin{equation}
   { \tilde{\nabla}_p^c \cdot 
   \left( \nabla_c^q \times \mathbf{A}_q \right)} = 0.
    \label{eqn.disc.divrot}
\end{equation}
\end{theorem}
\begin{proof}
    We employ Einstein index notation for convenience and also make use of the classical fully antisymmetric Levi-Civita tensor $\epsilon_{ijk}$, which is well-known and whose definition does not need to be repeated here. 
    The identity \eqref{eqn.disc.rotgrad} is easily proven at the aid of the definition of the curl and of the divergence in index notation   
    \begin{equation}
        \tilde{\nabla}_p^c \times \nabla_c^q Z_q = 
        \epsilon_{ijk} \, (\tilde{\partial}_j)_p^c \, (\partial_k)_c^q Z_q 
        = 0. 
        \label{eqn.disc.rotgrad2}
    \end{equation}
    The result of the expression on the left hand side is zero, because the contraction of the antisymmetric Levi-Civita symbol
    $\epsilon_{ijk}$ with a symmetric tensor  is zero. The tensor $(\tilde{\partial}_j)_p^c \, (\partial_k)_c^q Z_q $  is symmetric thanks to the  discrete Schwarz theorem \eqref{eqn.discrete.schwarz} above. 
    The second identity \eqref{eqn.disc.divrot} can be proven in the same manner: 
    \begin{equation}
        \tilde{\nabla}_p^c \cdot \nabla_c^q \times \mathbf{A}_q = 
        (\tilde{\partial}_i)_p^c \, \epsilon_{ijk}  (\partial_j)_c^q A_{kq}
        = 0.
        \label{eqn.disc.divrot2}
    \end{equation}    
    The expression on the left hand side is zero for the same reasons as before, i.e. simply using 
    \eqref{eqn.primary.nabla} and \eqref{eqn.dual.nabla} in index notation, the discrete Schwarz theorem for vector fields \eqref{eqn.ds.vector} and using the fact that the contraction of the antisymmetric Levi-Civita symbol with a symmetric tensor is zero. 
\qed\end{proof}
\paragraph{Correct choice of initial data} 
In order to guarantee that divergence and curl involutions hold exactly at the discrete level, the discrete solution in the DG space $\mathcal{U}_h^N$ \textit{cannot} be simply initialized via classical element-local $L^2$ projection. Instead,  the initial condition for those state variables in the DG space $\mathcal{U}_h^N$ that have to satisfy involution constraints must also be properly computed as the discrete gradient or discrete curl of the corresponding discrete scalar or vector potentials in the continuous finite element space $\mathcal{W}_h^{N+1}$. As an example, a discretely divergence-free magnetic field 
$\B_h \in \Uh$ must be initialized as 
\begin{equation}
    \B_h(\x,0) = \nabla \times \tilde{\A}_h\left(\x,0\right), 
\end{equation}
with a discrete vector potential $\tilde{\A}_h(\x,0) \in \Wh$. 
For the degrees of freedom of $\Bh(\x,0)$ at the initial time we simply get  
\begin{equation}
    \B_c(0) = \nabla_c^q \times \A_q(0). 
\end{equation}
Likewise, to initialize for example a discretely curl-free velocity field $\v_h \in \Uh$, one must compute it as the gradient of a discrete scalar potential $\tilde{Z}_h \in \Wh$, i.e. 
\begin{equation}
    \v_h(\x,0) = \nabla \tilde{Z}_h\left(\x,0\right), 
\end{equation}
or, in terms of the degrees of freedom
\begin{equation}
    \v_c(0) = \nabla_c^q Z_q(0). 
\end{equation}

\subsection{Exact pointwise local conservation and preservation of involutions} 
Let us again consider the scheme \eqref{eqn.cgdg} which we rewrite, for all elements $T_k$, as
\begin{equation}\label{eq:remi1}\int \limits_{T_k}\phi_i\Big( \partial_t\mathbf{u}_h+\nabla\cdot\tilde{\mathbf{f}}_h\Big ) \dx=0. 
\end{equation}
Here $\mathbf{u}_h\in \Uh$ and $\tilde{\mathbf{f}}_h\in \Wh$. Clearly, $\nabla\cdot \tilde{\mathbf{f}}_h\in \Uh$,  so that  we have, \textbf{strongly} on each $T_k$,
\begin{equation}\label{eqn:remi2}
\partial_t\mathbf{u}_h+\nabla\cdot\tilde{\mathbf{f}}_h=0.
\end{equation}
This shows pointwise conservation in each element. Local conservation is guaranteed because for each element, the characteristic function of $T_k$ belongs to $\Uh$, i.e.
\begin{equation}
    \int \limits_{T_k} \partial_t\mathbf{u}_h\dx+\int \limits_{T_k}\nabla\cdot\tilde{\mathbf{f}}_h \dx=0,
\end{equation}
and using again the continuity of $\tilde{\mathbf{f}}_h$ across the boundary of $T_k$, we get
\begin{equation}
\label{eqn:remi:3}
    \dfrac{\dd}{\dt}\int \limits_{T_k}\mathbf{u}_h\dx+\int \limits_{\partial T_k}\tilde{\mathbf{f}}_h\cdot \mathbf{n} \dgamma=0.
\end{equation}
We now show that 
\begin{proposition}\label{prop: remi}
    For any volume $V$ included in $\Omega$, we have the conservation relation:
\begin{equation}\label{eq:generalConservation}
\dfrac{\dd}{\dt}\int \limits_V \mathbf{u}_h \dx+\int \limits_{\partial V} \tilde{\mathbf{f}}_h\cdot \mathbf{n}\dgamma =0.
\end{equation}
\end{proposition}
This relation is more general than the usual one because the volume $V$ can be arbitrary.
\begin{proof}
Since $\tilde{\mathbf{f}}_h\in \Wh\subset H^1(\Omega)$, $\nabla\cdot \tilde{\mathbf{f}}_h\in L^2(\Omega)$ and writing
\begin{equation}V=\bigcup_k ( V\cap T_k ),\end{equation}
we get from \eqref{eqn:remi2}
\begin{equation}\dfrac{\mathtt{d}}{\dt}\int \limits_V \mathbf{u}_h \dx+\sum_{T_k}\int \limits_{V\cap T_k} \nabla\cdot \tilde{\mathbf{f}}_h\dx=0.\end{equation}
Then, \begin{equation}\int \limits_{V\cap T_k} \nabla\cdot \tilde{\mathbf{f}}_h\dx=\int \limits_{\partial\big (V\cap T_k\big )}\tilde{\mathbf{f}}_h\cdot \mathbf{n}\dgamma.\end{equation}
The boundary $\partial\big (V\cap T_k\big )$ has two parts: one part that is contained in the skeleton of the mesh, and a part that is made of the pieces of $\partial V$ that are contained in the interior of the elements. We will call this second part in the following the second skeleton. Because of the continuity of the reconstructed flux $\tilde{\mathbf{f}}_h$, the contribution on the skeleton vanishes. The union of the pieces contained in the second skeleton is simply $\partial V$, so that in the end
\begin{equation}\sum_{T_k}\int \limits_{V\cap T_k} \nabla\cdot \tilde{\mathbf{f}}_h\dx=\int \limits_{\partial V} \tilde{\mathbf{f}}_h\cdot \mathbf{n}\dgamma.\end{equation}
This ends the proof of Proposition~\ref{prop: remi}.
    \qed\end{proof}

    Thanks to \eqref{eq:remi1}, we also respect involutions everywhere. Assume that the system \eqref{eqn.pde} satisfies the involution
    \begin{equation}\dfrac{\dd}{\dt}\text{curl} L(\mathbf{q})=0\end{equation} where $L$ is a linear operator. This means that \eqref{eqn.pde} satisfies
    \begin{equation}\text{curl}\big (L(\nabla \cdot \mathbf{f})\big )=0\end{equation} or in other words that
    \begin{equation}L(\nabla \cdot \mathbf{f})=\nabla G(\mathbf{q})\end{equation} for some functional $G$.
    Similarly, we could have involutions of the kind
    \begin{equation}\dfrac{\dd}{\dt}\text{div} L(\mathbf{q})=0\end{equation}
    i.e.
    \begin{equation}\text{div}\big (L(\nabla \cdot \mathbf{f})\big )=0\end{equation} that is that
    \begin{equation}L(\nabla \cdot \mathbf{f})=\text{ curl }G(\mathbf{q})\end{equation} for some $G$.
    Typical examples are:
    \begin{itemize}
        \item For the acoustic equations, where $\mathbf{q}=(\mathbf{v},p)^T$, the operator $L$ is
        $L(\mathbf{q})=\mathbf{v}$ and then $G(\mathbf{q})=p/\rho$.
        \item For the Maxwell equations $\mathbf{q}=\big ( \mathbf{E}, \mathbf{B}\big )^T$, then we may consider
        \begin{equation}L\big (\mathbf{q}\big )=\mathbf{E} \text{ or }L\big (\mathbf{q}\big )=\mathbf{B}\end{equation} so that
        $\mathbf{G}(\mathbf{q})=-\mathbf{B}$ or $\mathbf{G}(\mathbf{q})=\mathbf{E}$, respectively.
    \end{itemize}
    What matters is that $L\big ( \nabla\cdot \tilde{\mathbf{f}}_h\big )=\nabla \tilde{G}_h$ in the first case and that $L\big ( \nabla\cdot \tilde{\mathbf{f}}_h\big )=\text{curl } \tilde{G}_h$ in the second case. This  is true for the two examples given previously. Using \eqref{eq:remi1}, we see that pointwise, we will have
    \begin{equation}\dfrac{\dd}{\dt} \text{curl }L\big (\mathbf{q}_h\big )=0 \qquad \text{ or } \qquad \dfrac{\dd}{\dt}\text{div} L\big (\mathbf{q}_h\big )=0.\end{equation}
        The question now is what happens  across the edges and faces.
    \begin{itemize}
        \item For the curl operator. We need that whatever $\boldsymbol{\tau}$ tangential to the face $f$, $L\big (\mathbf{q}\big )\cdot \boldsymbol{\tau}$ is continuous across $f$. More precisely, we show that for any polynomial $p_N$ of degree $N$,
        \begin{equation}\dfrac{d}{dt}\int \limits_f p_N\big [L\big (\mathbf{q}\big )\cdot \boldsymbol{\tau}\big ]\dgamma=0.\end{equation}
        \begin{proof}
    The face $f$ is the intersection of two elements, $T$ and $T'$. Since from $T$, we have
    \begin{equation}
        \begin{split}
    \dfrac{d}{dt}\int \limits_{f\cap T}p_N L\big (\mathbf{q}_h\big )\cdot \boldsymbol{\tau}\dgamma&=\int \limits_{f\cap T}p_N \big (\dfrac{d}{dt}
    L\big (\mathbf{q}_h\big )\big )\cdot \boldsymbol{\tau}\dgamma=-\int \limits_{f\cap T}p_N\nabla \tilde{G}_h\cdot\boldsymbol{\tau} \\&=
    -\int \limits_{f\cap T}p_N\partial_{\boldsymbol{\tau}}{\tilde{G}_h}
    \end{split}
    \end{equation}Seen from $T'$, we have
    \begin{equation}\dfrac{d}{dt}\int_{f\cap T'}p_N L\big (\mathbf{q}_h\big )\cdot \boldsymbol{\tau}\dgamma=
    -\int_{f\cap T'}p_N\partial_{\boldsymbol{\tau}}{\tilde{G}_h}\end{equation}
    Since $\partial_{\boldsymbol{\tau}}{\tilde{G}_h}$ is continuous across $f$, we then see that
    \begin{equation}0=\dfrac{d}{dt}\int \limits_f p_N\big [L\big (\mathbf{q}\big )\cdot \boldsymbol{\tau}\big ]\dgamma=
    \int \limits_f p_N\dfrac{\dd}{\dt}\big [L\big (\mathbf{q}\big )\cdot \boldsymbol{\tau}\big ]\dgamma\end{equation}
    Because  $\big [L\big (\mathbf{q}\big )\cdot \boldsymbol{\tau}\big ]$ is a polynomial of degree $N$, this means that pointwise, we have
    \begin{equation}\dfrac{d}{dt}\big [L\big (\mathbf{q}\big )\cdot \boldsymbol{\tau}\big ]=0\end{equation}
        \qed\end{proof}
        \item For the divergence operator, we can also show that
        \begin{equation}\dfrac{d}{dt}\int \limits_f p_N\big [L\big (\mathbf{q}\big )\cdot \mathbf{n}\big ]\dgamma=0.\end{equation}
        The proof is similar and comes from
        \begin{equation}p_N \text{curl }\tilde{\mathbf{G}}\cdot \mathbf{n}=p_N\text{ div }\big (\tilde{\mathbf{G}}\times \mathbf{n}\big )\end{equation}
        Then using coordinates in the positively oriented orthonormal frame $\{\mathbf{n}, \boldsymbol{\tau}_1, \boldsymbol{\tau}_2\}$ where $\boldsymbol{\tau}_1, \boldsymbol{\tau}_2$ is a basis of the face, we first write $\mathbf{G}=(G_n, G_{\tau_1}, G_{\tau_2})$ and see that
        \begin{equation}\mathbf{G}\times \mathbf{n}=\begin{pmatrix} 0\\ G_{\tau_2}\\ -G_{\tau_1}\end{pmatrix}\end{equation}
        and then 
        \begin{equation}\text{ div }\big ( \mathbf{G}\times \mathbf{n}\big )=\partial_{\tau_1}G_{\tau_2}-\partial_{\tau_2}G_{\tau_1}\end{equation}
        and this quantity, on the face, only depends on the $\tau_1$ and $\tau_2$ coordinates since the $\mathbf{n}$ coordinate is constant. This shows that the jump $\big[ L(\mathbf{q}_h) \cdot \mathbf{n} \big]$ is constant in time, because it is a polynomial of degree $N$.
        \end{itemize}

        We have shown the following 
        \begin{proposition}\label{prop:involutions}
            Assume that the system $\partial_t\mathbf{q}+\nabla\cdot \mathbf{f}=0$ admits the involution
            \begin{equation}\dfrac{\mathtt{d}}{\mathtt{d}t}\text{curl} L(\mathbf{q})=0 \qquad \textnormal{ or } \qquad \dfrac{\mathtt{d}}{\mathtt{d}t}\text{div} L(\mathbf{q})=0\end{equation}
        i.e. $L(\nabla\cdot \mathbf{f})=\nabla G$ or $L(\nabla\cdot \mathbf{f})=\text{curl }\mathbf{G}$ for a suitable $G$ or $\mathbf{G}$. If we have $L(\nabla\cdot \widetilde{\mathbf{f}})=\nabla \widetilde{G}$ or $L(\nabla\cdot \widetilde{\mathbf{f}})=\text{curl }\widetilde{\mathbf{G}}$, then the scheme respects pointwise  the involution.
        \end{proposition}

\subsection{Relation with the continuous finite element method}
\label{sec:equivalence}
We assume that $\PW$ and $\PU$ are $L^2$ projections on $\Wh$ and $\Uh$ respectively. With this assumption, the reconstructed variable $\wh$ solves a standard continuous finite element discretization.
\begin{lemma}
 \label{lemma:reduction}
 Let $\u_h$ be the solution of \eqref{eqn.cgdg} and define $\wh \doteq \PW \u_h$. Then $\wh$ is the solution to the standard Continuous Galerkin problem: find $\wh\in \Wh$ such that
 \begin{equation}
\label{eq:FEM}
\int \limits_\Omega\partial_t \wh\psi_i\,d\x-  \int \limits_\Omega \fh \cdot \nabla \psi_i  d\x= 0, \qquad \forall \psi_i \in \Wh, 
 \end{equation}
 subject to the initial condition $\wh(0) = \PW \u_h(0).$
\end{lemma}
\begin{proof}
 We choose the test function $\phi_i = \PU \psi_i$ in \eqref{eqn.cgdg} with $\psi_i \in \Wh$. Since $\nabla \cdot \fh\in \Uh$ and $\PU$ is self-adjoint with respect to the $L^2$ product, we obtain
\begin{eqnarray}
0 & = & \int \limits_\Omega\partial_t \u_h\cdot \PU  \psi_i\,d\x + \int \limits_\Omega \nabla \cdot \fh \PU \psi_i \,d\x \nonumber \\
 & = & \int \limits_\Omega \partial_t (\PU \u_h) \psi_i \,d\x + \int \limits_\Omega (\PU \nabla \cdot \fh) \psi_i\,d\x \nonumber \\
 & = & \int \limits_\Omega \partial_t  \u_h \psi_i \,d\x + \int \limits_\Omega  \nabla \cdot \fh \psi_i\,d\x \nonumber \\
 & = & \int \limits_\Omega \partial_t (\PW \u_h) \psi_i \,d\x + \int \limits_\Omega \nabla \cdot \fh \psi_i \,d\x.
\end{eqnarray}
The proof is concluded with the definition $\wh \doteq \PW\u_h$ and integration by parts.
\qed\end{proof}
\begin{remark}\label{rem:cg.equivalence}
	
Although our new scheme is equivalent to a standard continuous Galerkin method in the case when the projector $\PW$ is the standard $L^2$ projection and no upwind/stabilization is present, what makes it actually unique and novel is the \textit{additional DG variable} $\uh$, which is evolved with the globally continuous discrete flux field $\fh$ and therefore satisfies \textit{exact pointwise conservation} and integral conservation on \textit{arbitrary control volumes} for general initial data, together with the \textit{exact preservation} of the \textit{involutions}, provided that $\uh(0)$ is initialized in a compatible manner as described in Section~\ref{sec.schwarz}. We emphasize that this novelty is of \textit{structural nature}.

We further stress that it is \textit{not} the aim of this paper to obtain a reduction of the computational cost with respect to the classical continuous Galerkin method. In the case of the $L^2$ projection, the global sparse linear system that has to be solved for $\wh$ in each Runge-Kutta stage has exactly the same matrix as the one of a classical continuous finite element method with consistent mass matrix, and the only additional effort is the element-local update of the DG variable $\uh$ via \eqref{eqn.cgdg.dof}, which involves only the block-diagonal DG mass matrix and is therefore computationally cheap.

However, the classical continuous finite element solution  evolves according to $\partial_t \wh = -\PW (\nabla \cdot \fh)$, i.e.  with the $L^2$ projection of the flux divergence onto $\Wh$, so that in general $\partial_t \wh + \nabla \cdot \fh \neq 0$ and the involutions are not preserved exactly, whereas our additional DG variable $\uh$ satisfies \eqref{eqn:remi2} pointwise and therefore possesses both exact pointwise local conservation and exact preservation of the involutions, see Propositions~\ref{prop: remi} and~\ref{prop:involutions}.

We note that local conservation of classical continuous Galerkin methods can be recovered on elements or unions of elements by means of suitably post-processed fluxes, see \cite{HughesCGcons,CockburnCGcons}, while in the new CG-DG scheme exact pointwise conservation holds directly for the evolved variable $\uh$, with the same continuous flux field $\fh$, on arbitrary control volumes and without any post-processing.

Moreover, we remark that many other choices for $\PW$ are possible in our general framework, and some of them may be preferable to the $L^2$ projection because they naturally incorporate upwinding or stabilization and do not require solving a global sparse linear system at each time step. We present one such alternative reconstruction in Section~\ref{sec:alternative_reconstruction}. However, at present, the $L^2$ projection is the only reconstruction for which we can establish stability and error estimates, at least in the linear case, due to the equivalence to the continuous finite element method.

\end{remark}

\subsection{Quadratic energy conservation, nonlinear stability and error estimates for linear symmetric hyperbolic systems} 

In this section, we assume that $\f(\q) = \A\q$ with $\A = A_{ijk}$ a constant 
tensor of rank 3, \textbf{symmetric} in the first two indices, so that \eqref{eqn.pde} can be written in index notation as
\begin{equation}
\label{eq:quasi-linear}
    \partial_t q_i + A_{ijk}\partial_kq_j = 0.
\end{equation}
Note that the symmetry implies the integration-by-parts formula
\begin{equation}
    \int \limits_{\Omega}\nabla\cdot (\A(\q \otimes \p))\,d\x = \int \limits_{\Omega}(\A:\nabla \q)\p\,d \x + \int \limits_{\Omega}\q\cdot(\A:\nabla \p)\,d\x,
    \label{eq:int-by-parts}
\end{equation}
for smooth vector fields $\q$ and $\p$. For the ease of notation we introduce the $L^2$ norm of a vector field $\q$ as 
\begin{equation}
    \lVert \q\rVert \doteq \sqrt{ \int \limits_{\Omega} \lvert \q\rvert^2\,d\x}.
\end{equation}
Then, we readily obtain conservation of energy. \begin{lemma}[Continuous energy conservation]
\label{lem:cont_energy_cons}
    Let $\q$ be the exact solution of \eqref{eq:quasi-linear}. If $\q$ is sufficiently smooth and vanishes on the boundary of $\Omega$, then it satisfies the energy conservation
    \begin{equation}
        \frac{1}{2}\frac{d}{dt} \norm{\q}^2 = 0. 
    \end{equation}
\end{lemma} 
\begin{proof}
    Multiply \eqref{eq:quasi-linear} by $\q$ and integrate on $\Omega$. Then, identity \eqref{eq:int-by-parts} gives
    \begin{equation}
    \frac{1}{2} \left(\frac{d}{dt} \norm{\q}^{2} + \int_{\Omega}\nabla \cdot (\A(\q \otimes \q))\,d\x \right) = 0.
    \label{eq:cont_energy_proof}
    \end{equation}
The proof is concluded applying the divergence theorem.
\qed\end{proof}
The numerical method becomes: find $\u_h = \u_h(\x, t)\in \Uh$, satisfying
\begin{equation}
\label{eq:weak_form}
    \int \limits_{\Omega} \partial_t \u_h\cdot \dq_h\,d \x + \int \limits_{\Omega} (\A :\nabla \wh) \cdot \dq_h\,d\x = 0, \qquad \forall \dq_h\in \Uh.
\end{equation}
As in Section \ref{sec:equivalence}, we assume that $\PW$ and $\PU$ are $L^2$ projections on $\Wh$ and $\Uh$ respectively. Then, Lemma \ref{lemma:reduction} implies that $\wh = \PW \uh$ satisfies 
\begin{equation}
    \int \limits_\Omega\partial_t \wh\cdot\dqt\,d\x+  \int \limits_\Omega(\A : \nabla \wh)\cdot\dqt\, d\x= 0, \qquad \forall \dqt \in \Wh, 
    \label{eq:FEM_linear}
 \end{equation}
 subject to the initial condition $\wh(0) = \PW \u_h(0).$ As a consequence, we readily obtain energy conservation and error estimates for $\wh$. 
\begin{lemma}[Discrete energy conservation]\label{lem:disc_energy_cons}
Assume periodic or vanishing boundary conditions. Let $\uh$ be the solution of \eqref{eq:weak_form}. Then $\wh = \PW \uh$ satisfies the energy conservation
\begin{equation}
 \label{eq:energy_cons}
 \frac{1}{2} \frac{d}{dt} \lVert \wh \rVert^2 = 0.
\end{equation}
\end{lemma}
\begin{proof}
Take $\dqt = \wh$ in \eqref{eq:FEM_linear}. Then, the proof proceeds as the one of Lemma \ref{lem:cont_energy_cons}.
\qed\end{proof}
We continue now with the error analysis of the scheme, assuming that the mesh is quasi-uniform. Let $\mathbf{H}^s(\Omega)$ be the standard vector-valued Sobolev space of $\mathbf{L}^2(\Omega)$ functions with derivatives up to order $s$ in $\mathbf{L}^2(\Omega)$, equipped with norm and seminorm $\lVert \cdot \rVert_{\H^s(\Omega)}$ and $\seminorm{\cdot}_{\H^s(\Omega)}$, see e.g. \cite[Proposition 2.9]{EG1}.
\begin{lemma}
 \label{lemma:cg_error}
 Let $\q \in L^\infty(0,T; \H^{N+2}(\Omega))$ be the exact solution of \eqref{eqn.pde} with $\mathbf{f}(\mathbf{q})=\mathbf{A}\mathbf{q}$ and $\mathbf{A}$ constant. 
 Assume that the initial data satisfies $\lVert \u_h(0) - \q(0) \rVert \lesssim h^{N+ 1}\seminorm{\q(0)}_{\H^{N+ 1}(\Omega)}$ and that either periodic boundary conditions are used, or $\q$ vanishes on $\partial \Omega$ and $\PW$ is the $L^2$ projection onto the subspace $\mathcal{W}_{h,0}^{N+1} \subset \Wh$ of functions that vanish on $\partial \Omega$, so that also $\wh$ and $\PW \q$ vanish on $\partial \Omega$.
 Then, the solution $\wh$ to \eqref{eq:FEM} satisfies the error estimate
 \begin{equation}
\norm{\wh(t) - \q(t)} = \mathcal{O}(h^{N+1}).
 \end{equation}
\end{lemma}
\begin{proof}
The proof uses standard FEM techniques, but we report it for completeness.
 Subtracting the weak form satisfied by the exact solution from \eqref{eq:FEM_linear} with test function $\dqt = \boldsymbol{e}_h \doteq \wh - \PW \q \in \Wh$, we obtain the error equation:
 \begin{eqnarray}
\int \limits_\Omega \partial_t \boldsymbol{e}_h\cdot \boldsymbol{e}_h \,d\x + \int \limits_\Omega (\A : \nabla \boldsymbol{e}_h)\cdot \boldsymbol{e}_h \,d\x & = & \int \limits_\Omega (\partial_t \q - \partial_t \PW \q)\cdot \boldsymbol{e}_h \,d\x \nonumber \\
 & & + \int \limits_\Omega (\A: \nabla ( \q - \PW\q))\cdot \boldsymbol{e}_h \,d\x.
 \end{eqnarray}
 Observing that $\int \limits_\Omega \partial_t \boldsymbol{e}_h\cdot \boldsymbol{e}_h \,d\x = \frac{1}{2}\frac{d}{dt} \norm{\boldsymbol{e}_h}^2$ and that \eqref{eq:int-by-parts} implies $\int \limits_\Omega ( \A: \nabla \boldsymbol{e}_h)\cdot \boldsymbol{e}_h \,d\x = 0$, we analyze the right-hand side.

 The first term on the right-hand side vanishes due to the $L^2$-orthogonality of $\PW$:
 \begin{equation}
\int \limits_\Omega (\partial_t  \q - \partial_t\PW \q)\cdot \boldsymbol{e}_h \,d\x = \int \limits_\Omega (\partial_t \q - \PW(\partial_t \q))\cdot \boldsymbol{e}_h \,d\x = 0.
 \end{equation}
 For the second term, we employ the H\"older inequality and standard approximation estimates:
 \begin{eqnarray}
\int \limits_\Omega (\A : \nabla (\q - \PW\q)) \cdot \boldsymbol{e}_h \,d\x & \leq & \seminorm{\A}_\infty \norm{\nabla(\q - \PW\q)} \norm{\boldsymbol{e}_h} \nonumber \\
 & \lesssim & h^{N+ 1} \seminorm{\q}_{\H^{N+2}(\Omega)} \norm{\boldsymbol{e}_h}.
\end{eqnarray}
 
Substituting this into the error equation, we obtain 
\begin{equation}
    \frac{1}{2}\frac{d}{dt} \norm{\boldsymbol{e}_h(t)}^2 \leq Ch^{N+1} \seminorm{\q(t)}_{\H^{N+2}(\Omega)}\norm{\boldsymbol{e}_h(t)}.
        \label{eq:energy_diff_bound}
\end{equation}
Let $M = \max_{t\in [0,T]} \norm{\boldsymbol{e}_h(t)}$. Integrating \eqref{eq:energy_diff_bound} from $0$ to $t \leq T$ and using H\"older's inequality in time, we obtain
\begin{equation}
    \norm{\boldsymbol{e}_h(t)}^2 \leq \norm{\boldsymbol{e}_h(0)}^2 + Ch^{N+1} M \norm{\q}_{L^1(0,T;\H^{N+2}(\Omega))} \qquad \forall t \in [0,T].
        \label{eq:energy_integrated}
\end{equation}
Taking the maximum over $t$ on the left-hand side and applying Young's inequality to the last term yields
\begin{equation}
    M^2 \leq \norm{\boldsymbol{e}_h(0)}^2 + Ch^{2(N+1)}\norm{\q}_{L^1(0,T;\H^{N+2}(\Omega))}^2 + \frac{1}{2}M^2,
        \label{eq:energy_young}
\end{equation}
and therefore
\begin{eqnarray}
    \norm{\boldsymbol{e}_h(t)} & \leq & M \leq \sqrt{2}\norm{\boldsymbol{e}_h(0)} + Ch^{N+1}\norm{\q}_{L^1(0,T;\H^{N+2}(\Omega))} \nonumber \\
    & \leq & \sqrt{2}\norm{\boldsymbol{e}_h(0)} + Ch^{N+1}\,T \norm{\q}_{L^\infty(0,T;\H^{N+2}(\Omega))}.
        \label{eq:energy_final}
\end{eqnarray}

For the error at the initial time we have 
\begin{eqnarray}
    \norm{\boldsymbol{e}_h(0)} & = & \norm{\wh(0) - \PW\q(0)} \nonumber \\
    & = & \norm{\PW (\uh(0) - \q(0))} \nonumber \\
    & \leq & \norm{\uh(0) - \q(0)} \nonumber \\
    & \lesssim & h^{N+1} \seminorm{\q(0)}_{\H^{N+1}(\Omega)},
    \label{eq:initial_error}
\end{eqnarray}
implying 
\begin{equation}
    \norm{\boldsymbol{e}_h(t)} = O(h^{N+1}).
    \label{eq:bound_eh}
\end{equation}
Finally the triangle inequality $\norm{\wh - \q} \leq \norm{\boldsymbol{e}_h} + \norm{\PW \q - \q}$ concludes the proof. 
\qed\end{proof}

\begin{theorem}
 \label{thm:main_error}
 Let $\q \in L^\infty(0,T; \mathbf{H}^{N+2}(\Omega))$ be the exact solution, let the assumptions of Lemma~\ref{lemma:cg_error} hold, and let the mesh be quasi-uniform. 
 Then, the semi-discrete solution $\u_h$ to \eqref{eq:weak_form} satisfies the error estimate:
 \begin{equation}
\norm{\u_h(t) - \q(t)} = 
\mathcal{O}(h^N).
 \end{equation}
\end{theorem}

\begin{proof}
 We proceed via the triangle inequality using the projected variable $\wh = \PW \u_h$:
 \begin{equation}
\label{eq:triangle_split}
\norm{\u_h(t) - \q(t)} \leq \norm{\u_h(t) - \wh(t)} + \norm{\wh(t) - \q(t)}.
 \end{equation}
 The second term is bounded by Lemma \ref{lemma:cg_error} as:
 \begin{equation}
\label{eq:bound_term2}
\norm{\wh(t) - \q(t)} =\mathcal{O}(h^{N+1}).
 \end{equation}
Then we only need to bound the first term. We define the projection offset $\boldsymbol{\delta}_h(t) \doteq \u_h(t) - \wh(t)$. Differentiating with respect to time and utilizing the discrete evolution equation $\partial_t \u_h = -\A :\nabla \wh$ (which holds in $\Uh$ since $\A$ is constant and $\nabla \wh \in \Uh$), we obtain:
 \begin{equation}
\partial_t \boldsymbol{\delta}_h = \partial_t \u_h - \PW (\partial_t \u_h) = (I - \PW) \partial_t \u_h = -(I - \PW)(\A : \nabla \wh).
\label{eq:offset}
 \end{equation}
We integrate from $0$ to $t$, obtaining
\begin{equation}
    \boldsymbol{\delta}_h(t) = \boldsymbol{\delta}_h(0) - \int_0^t(I-\PW)(\A:\nabla \wh(s))\,ds,
\end{equation}
which in turn implies 
\begin{eqnarray}
\norm{\boldsymbol{\delta}_h(t)} & \leq & \norm{\boldsymbol{\delta}_h(0)} + \int_0^t\norm{(I-\PW)(\A:\nabla \wh(s))}\,ds \nonumber \\
 & \leq & \norm{\boldsymbol{\delta}_h(0)} +  T \sup_{s\in [0,T]}\norm{(I-\PW)(\A:\nabla \wh(s))}.
\end{eqnarray}
 To bound the second term on the right-hand side, we introduce the gradient of the exact solution $\nabla \q$:
 \begin{eqnarray}
 \norm{(I - \PW) (\A : \nabla \wh)} & \leq & \norm{(I - \PW)[\A:(\nabla \wh - \nabla \q)]} + \norm{(I - \PW)(\A : \nabla \q)} \nonumber \\
 & \leq & \norm{\A : (\nabla \wh - \nabla \q)} + \norm{(I - \PW)(\A : \nabla \q)} \nonumber \\
 & \leq & \seminorm{\A}_\infty \norm{\nabla \wh - \nabla \q} + \norm{(I - \PW) (\A:\nabla \q)} \nonumber \\
 & \leq & \seminorm{\A}_{\infty}\norm{\nabla \boldsymbol{e}_h} + \seminorm{\A}_{\infty}\norm{\nabla (I - \PW) \q} + \norm{(I - \PW)(\A : \nabla \q)}.
\end{eqnarray}
 We bound the first term using an inverse inequality and \eqref{eq:bound_eh}:
 \begin{equation}
\seminorm{\A}_{\infty}\norm{\nabla \boldsymbol{e}_h} \lesssim h^{-1}\norm{\boldsymbol{e}_h} = O(h^N).
 \end{equation}
 
 The second and third terms can be bounded with standard approximation properties of $\PW$:
 \begin{equation}
\seminorm{\A}_{\infty}\norm{\nabla(I - \PW)\q} + \norm{(I - \PW)(\A:\nabla \q)} \lesssim h^{N}\seminorm{\q}_{\H^{N+1}(\Omega)}.
 \end{equation}
 
 Combining these two estimates, we obtain:
 \begin{equation}
\norm{ \boldsymbol{\delta}_h(t)}  \leq \norm{\boldsymbol{\delta}_h(0)} + O(h^N).
 \end{equation}
The following bound on the initial offset:
 \begin{equation}
\norm{\boldsymbol{\delta}_h(0)} = \norm{(I - \PW)\u_h(0)} \leq \norm{\u_h(0) - \q(0)} + \norm{(I-\PW)\q(0)} \lesssim h^{N+ 1} \seminorm{\q(0)}_{\mathbf{H}^{N+1}(\Omega)}
 \end{equation}

concludes the proof. 
\qed\end{proof}

\section{More efficient flux reconstruction(s)}
\label{sec.effi}
In order to make the global $L^2$ projection more efficient, we propose the following simple strategy to find a suitable \textit{initial guess} for the conjugate gradient method that is used to solve \eqref{eqn.L2proj.dof}. The approach is based on a quasi-interpolation procedure, similar to that introduced by Ern and Guermond in \cite{ErnGuermond17}. For each element $T_k$ define a reconstruction stencil $\mathcal{R}_k$ which contains all the elements that share a common vertex with $T_k$. Then solve the following \textit{local} $L^2$ problem exactly,  
\begin{equation}
	\int \limits_{\mathcal{R}_k} \psi_i \psi_j d\x \, \tilde{\w}^k_j(t) = 
	\int \limits_{\mathcal{R}_k} \psi_i \phi_j d\x \, \hat{\u}_j(t),  
	\label{eqn.loc.L2proj.dof}
\end{equation}
which is easy to solve, since the dimension of the stencil mass matrix to be inverted is now sufficiently small and does not depend on the total number of elements $|\mathcal{T}_h|$ in $\mathcal{T}_h$, but only on the number of elements $|\mathcal{R}_k|$ in the reconstruction stencil $\mathcal{R}_k$. The inverse of this localized FEM mass matrix can actually be precomputed for each $T_k$,  multiplied with the right-hand side matrix and stored for each element. Note that in \eqref{eqn.loc.L2proj.dof} the degrees of freedom $\tilde{\w}^k_j(t)$ are computed individually for each element $T_k$ and therefore global continuity is \textit{not} yet established. We then compute the \textit{initial guess} for the degrees of freedom of the globally continuous representation $\wh$ as
\begin{equation}
 \hat{\w}^0_j(t) = \frac{1}{|\mathcal{N}_j|} \sum \limits_{T_k \in \mathcal{N}_j} \tilde{\w}^k_j(t),
 \label{eqn.average}
\end{equation}
where $\mathcal{N}_j$ is the set of triangles $T_k$ that contain node number $j$. For internal degrees of freedom, i.e. those inside $T_k$, one clearly has $|\mathcal{N}_j|=1$, for nodes in the interior of edges (2D) / faces (3D) one has $|\mathcal{N}_j|=2$, and for nodes in the interior of edges (3D) or on vertices (both 3D and 2D) one has in general a much larger number of elements, corresponding to those elements that share a common edge (3D) or vertex (3D and 2D). With the initial guess \eqref{eqn.average} one then solves the full problem \eqref{eqn.L2proj.dof} via the matrix-free conjugate gradient method up to a tolerance of $10^{-14}$. 

\subsection{Genuinely multi-dimensional upwinding using the N matrix} 
\label{sec:alternative_reconstruction}

The idea outlined here is a local alternative to the $L^2$ projection and is based on an averaging technique that 
has recently also been used in semi-discrete active flux (AF) / PAMPA schemes \cite{Abgrall_AF,AbgrallBoscheriLiu,AbgrallLiuLin,AbgrallLiuDg}. This  is inspired by the so-called N-scheme of the residual distribution (RD) methods \cite{Roe:90,RoeSidilkover,vanderweide,energie,Mario}.

We evaluate $\tilde{\w}$ from $\hat{\u}$ by
\begin{equation}
 \hat{\w}_j(t) = \mathbf{N}_j^{-1} \sum \limits_{T_k \in \mathcal{N}_j} \mathbf{K}_{k,j}^+ \, \tilde{\w}^k_j(t),
 \label{eqn.Nscheme}
\end{equation}
with 
\begin{equation}
 \mathbf{N}_j = \sum \limits_{T_k \in \mathcal{N}_j} \mathbf{K}_{k,j}^+ 
\end{equation}
and 
\begin{equation}
 \mathbf{K}_{k,j}^+ = \mathbf{R}_n \big (\max \left( \boldsymbol{\Lambda}_n, 0 \right)+\varepsilon \mathbf{I}\big ) \mathbf{R}^{-1}_n  
\end{equation}
with $\mathbf{R}_n$ and $\boldsymbol{\Lambda}_n$, the matrix of right eigenvectors and the diagonal matrix of eigenvalues of the Jacobian matrix 
\begin{equation}\A_n = \frac{\partial (\f \cdot \n_{k,j})}{\partial \q}\end{equation}
of the hyperbolic system \eqref{eqn.pde} in normal direction $\n_{k,j}$, which is the outward-pointing normal vector for degree of freedom $j$ in element $k$, whose magnitude is a length in 2D and an area in 3D, and where $\A_n$ is evaluated at the element-local state $\tilde{\w}^k_j$. The term $\varepsilon \mathbf{I}$ is here to avoid any singularity of $\mathbf{N}_j$. This kind of singularity may occur in cases where the flow is aligned with one of the faces.  In practice we take $\varepsilon = \frac{1}{2} |T_k|$, i.e. half of the volume of the element $T_k$.  For internal degrees of freedom we simply set $\mathbf{K}_{k,j}^+ = \mathbf{I}$ with $\mathbf{I}$ the identity matrix.  Of course there $\varepsilon=0$. For degrees of freedom at the vertices of an element we use the corner normal \cite{Despres2005,Maire2007,Despres2009,ShashkovCellCentered,Lipnikov2014,Maire2020,HTCLagrange,Mario,energie,AbgrallBoscheriLiu,AbgrallLiuLin,AbgrallLiuDg}, for degrees of freedom in the interior of an edge/face we use the edge/face normal. In the future we will also consider the use of composite finite volume schemes 
\cite{HochComposite1,HochComposite2,HochComposite3} for the calculation of  numerical fluxes in all  degrees of freedom of the finite element space $\Wh$, which can be either internal, or on faces, on edges or on vertices. 

\section{Numerical results}
\label{sec.results}

In the following we show numerical results obtained with the new CG-DG scheme applied to three different linear and nonlinear hyperbolic PDE systems, namely the equations of linear acoustics, the vacuum Maxwell equations and the nonlinear compressible Euler equations. In all test problems, the semi-discrete scheme is integrated in time with the classical fourth-order Runge-Kutta method for $N \leq 3$ and with a six-stage fifth-order Runge-Kutta scheme of Butcher, see e.g. \cite{butcher}, for $N=4$. The time step is chosen with a CFL number of $0.4$, and the reconstruction $\wh$ is always computed via the $L^2$ projection. Since the pointwise conservation law \eqref{eqn:remi2} and the discrete involution constraints are linear in $\uh$, and since each stage of an explicit Runge-Kutta scheme is a linear combination of increments of the form $-\nabla \cdot \fh$, both properties are preserved exactly also by the fully discrete scheme in every Runge-Kutta stage, see also Table~\ref{tab.EPerrors} for the conservation and Figures~\ref{fig.smoothacoustics}, \ref{fig.res.ac.expl}, \ref{fig.smoothmaxwell} and~\ref{fig.res.max.expl} for the involutions, while the quadratic energy conservation of Lemma~\ref{lem:disc_energy_cons} holds only at the semi-discrete level and is therefore affected by the time discretization error. The total energy shown in all figures of this section is the one of the continuous reconstruction $\wh$, which is the quantity conserved according to Lemma~\ref{lem:disc_energy_cons}. The use of energy-conserving time integrators, see e.g. \cite{Brugnano1,Brugnano2}, is possible within our framework but beyond the scope of the present paper.    

\subsection{Equations of linear acoustics}

The equations of linear acoustics read 
\begin{eqnarray}    
    \frac{\partial \rho  \v}{\partial t} + \nabla p & = & 0, \label{eqn.ac.v} \\
    \frac{\partial p}{\partial t} + \rho c^2 \, \nabla \cdot \v & = & 0, \label{eqn.ac.p}      
\end{eqnarray}
with constant fluid density $\rho$, velocity vector $\v$, pressure $p$ and constant sound speed $c$. In all test problems of this section we set $\rho=1$ and $c=1$. The above system satisfies the following extra conservation law for the total energy density $\mathcal{E} = \halb \rho \v^2 + \frac{1}{2 \rho c^2} p^2$: 
\begin{equation}
\frac{\partial \mathcal{E}}{\partial t} + \nabla \cdot \left( \v p \right) = 0.
    \label{eqn.ac.extra}
\end{equation}
From the momentum equation \eqref{eqn.ac.v} it is obvious that the velocity field remains curl-free if it was so initially, i.e. the following involution holds 
\begin{equation}
   \nabla \times \v = 0.
   \label{eqn.ac.inv}
\end{equation}
This means that in the new CG-DG scheme introduced in this paper, the discrete velocity field $\v_h \in \Uh$ must be initialized as the discrete gradient of a scalar potential $Z_h \in \Wh$ so that \eqref{eqn.ac.inv} holds at the discrete level. In the presence of discontinuities and in order to keep the velocity field curl-free, the artificial viscosity is chosen such that its structure is compatible with the equations. Following \cite{Sidilkover2025}, the pressure and the velocity  fields used in the numerical flux field are modified as follows: 
\begin{equation}
    \tilde \v_h = \v(\wh) - \epsilon \tilde{\nabla} p_h,
    \qquad 
    \tilde p_h  = p(\wh) - \epsilon \tilde{\nabla} \cdot \v_h,    
\end{equation}
with $\v(\wh)$ and $p(\wh)$ the averaged states given by $\wh$ and $\epsilon$ the artificial viscosity coefficient given already above. 

\subsubsection{Smooth acoustic wave propagation}

The first test for linear acoustics is taken from \cite{CompatibleDG1}, with initial condition  
\begin{equation}
  \v(\x, 0)  = 0, \qquad 
  p(\x, 0)  = p_0 \exp \left( -\halb |\x|^2/\sigma^2 \right) 
	\label{eqn.ic.gauss.max.acoustic}
\end{equation}
and $p_0=1$ and $\sigma=0.05$. 
Simulations are carried out in the domain $\Omega=[-\halb,+\halb]^d$ until a final time of $t=10$, without any artificial viscosity. We use an unstructured triangular mesh composed of a total number of 2028 triangles with $N_x=N_y=30$ edges along each side of the computational domain. We use a CG-DG scheme of polynomial approximation degree $N=3$. 
The computational results obtained are depicted in Figure \ref{fig.smoothacoustics}. 
As expected, the method conserves the total energy up to the small error of the time discretization even over long times, and the discrete velocity field remains curl-free up to machine precision, i.e. $\tilde{\nabla} \times \v_h = 0$ in the sense of \eqref{eqn.disc.rotgrad}. The linear growth in the curl error is due to the accumulation of roundoff errors in the computer. The growth of the total energy conservation error is due to the time discretization which is not exactly energy conservative and our results hold only at the semi-discrete level so far. For this reason, future research will also concern the use of exactly energy conservative time integrators.    

\begin{figure}[!htbp]
	\begin{center}
		\begin{tabular}{cc} 
			\includegraphics[width=0.45\textwidth]{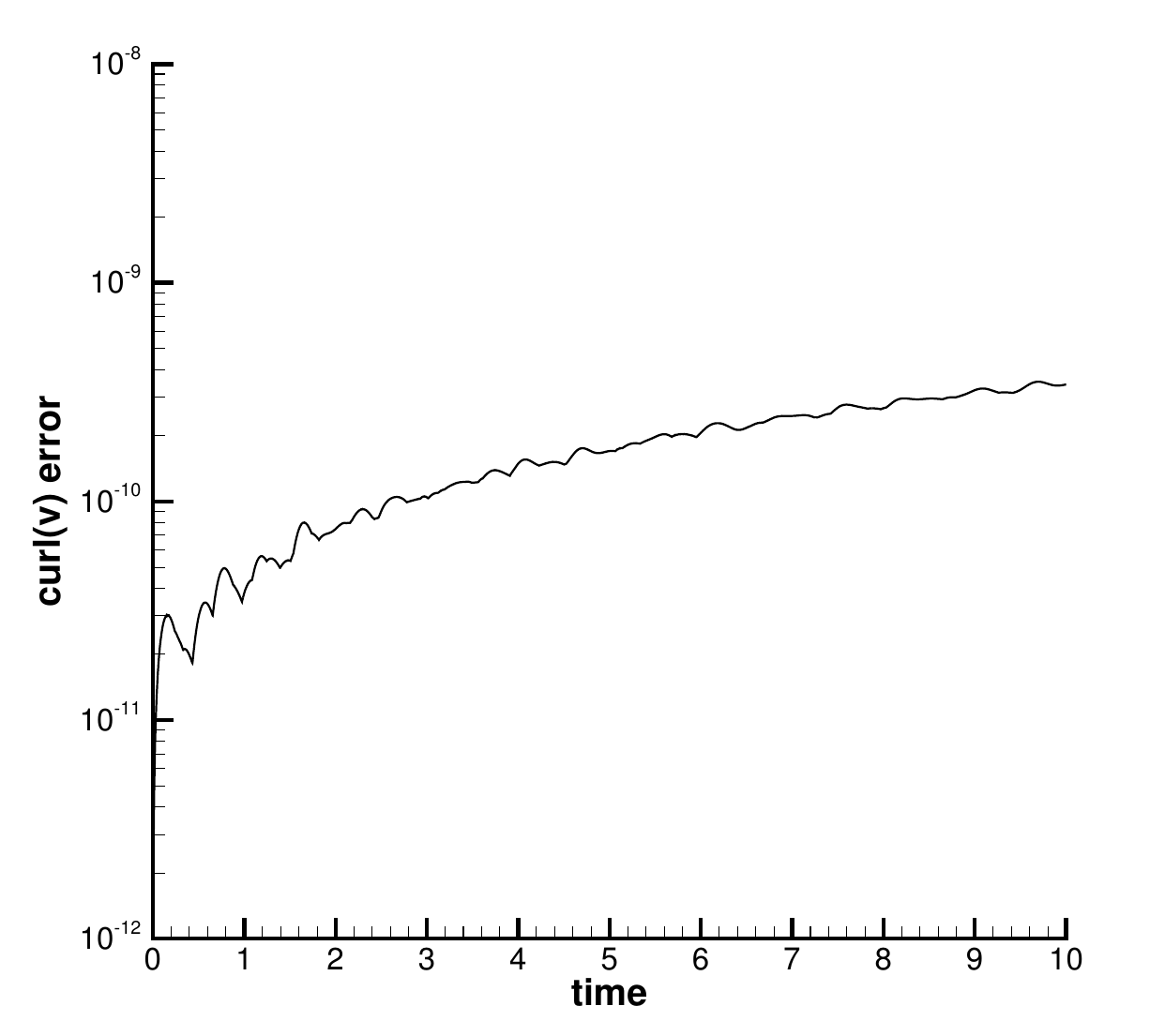}   & 
			\includegraphics[width=0.45\textwidth]{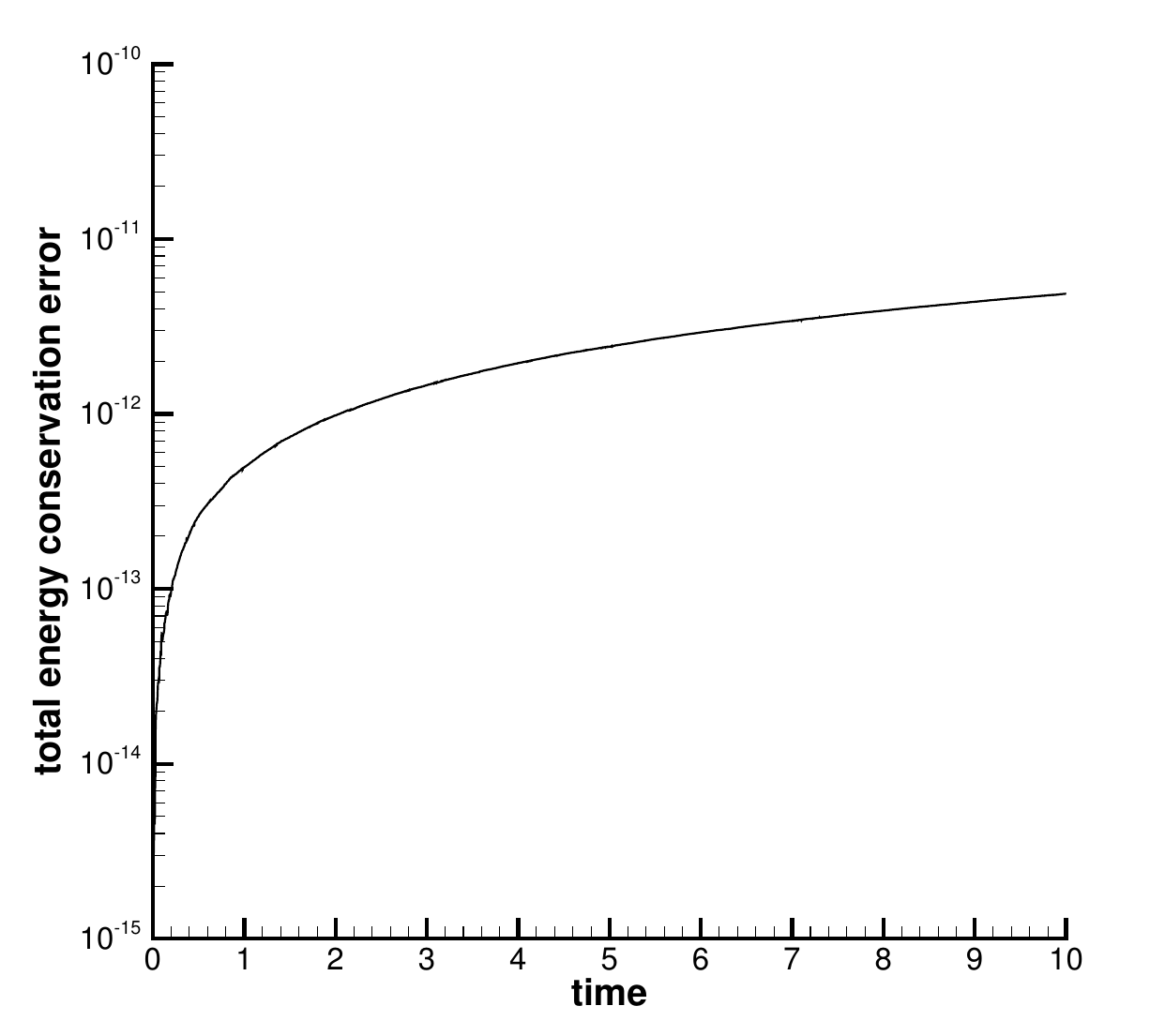}     
		\end{tabular} 
		\caption{Smooth acoustic wave propagation. Time evolution of the $L^\infty$ error of the curl of the velocity field (left) and time evolution of the total energy conservation error (right). } 
		\label{fig.smoothacoustics}
	\end{center}
\end{figure}

\subsubsection{Acoustic waves with discontinuities}
Next, we consider a  test problem that contains a nontrivial initial velocity field and an initially discontinuous pressure profile. 
The initial velocity is computed as the gradient of a scalar potential $Z$ and reads  
\begin{equation}
    \v(\x,0) = \nabla Z, \qquad \textnormal{ with } \qquad 
    Z = 0.04 \sin(2 \pi x) \sin(2 \pi y), 
\end{equation}
while the initial pressure is 
\begin{equation}
    p(\x,0) = \left\{   \begin{array}{ll} 
    1 & \qquad \textnormal{if} \qquad r \leq R, \\
    0 & \qquad \textnormal{if} \qquad r > R,
    \end{array} \right. 
\end{equation}
with $r = |\x|$ and $R=0.25$. 
The computational domain is $\Omega = [-\halb,+\halb]^2$ with periodic boundaries everywhere. The unstructured mesh is composed of 3608 triangles, with $40$ edges along each side of the square domain. 
Simulations are carried out until $t=0.1$ with the new CG-DG scheme with polynomial approximation degree $N=3$ and the compatible artificial viscosity activated, since the solution presents discontinuities.  
The computational results are depicted in Figures \ref{fig.res.ac.expl} and \ref{fig.res.ac.expl3D}. 
One can observe that the curl errors remain of the order of machine precision, despite a linear growth that is due to the accumulation of roundoff errors in the computer. The energy is decreasing since the method is dissipative due to the presence of the artificial viscosity. 

\begin{figure}[!htbp]
	\begin{center}
		\begin{tabular}{cc} 
			\includegraphics[width=0.45\textwidth]{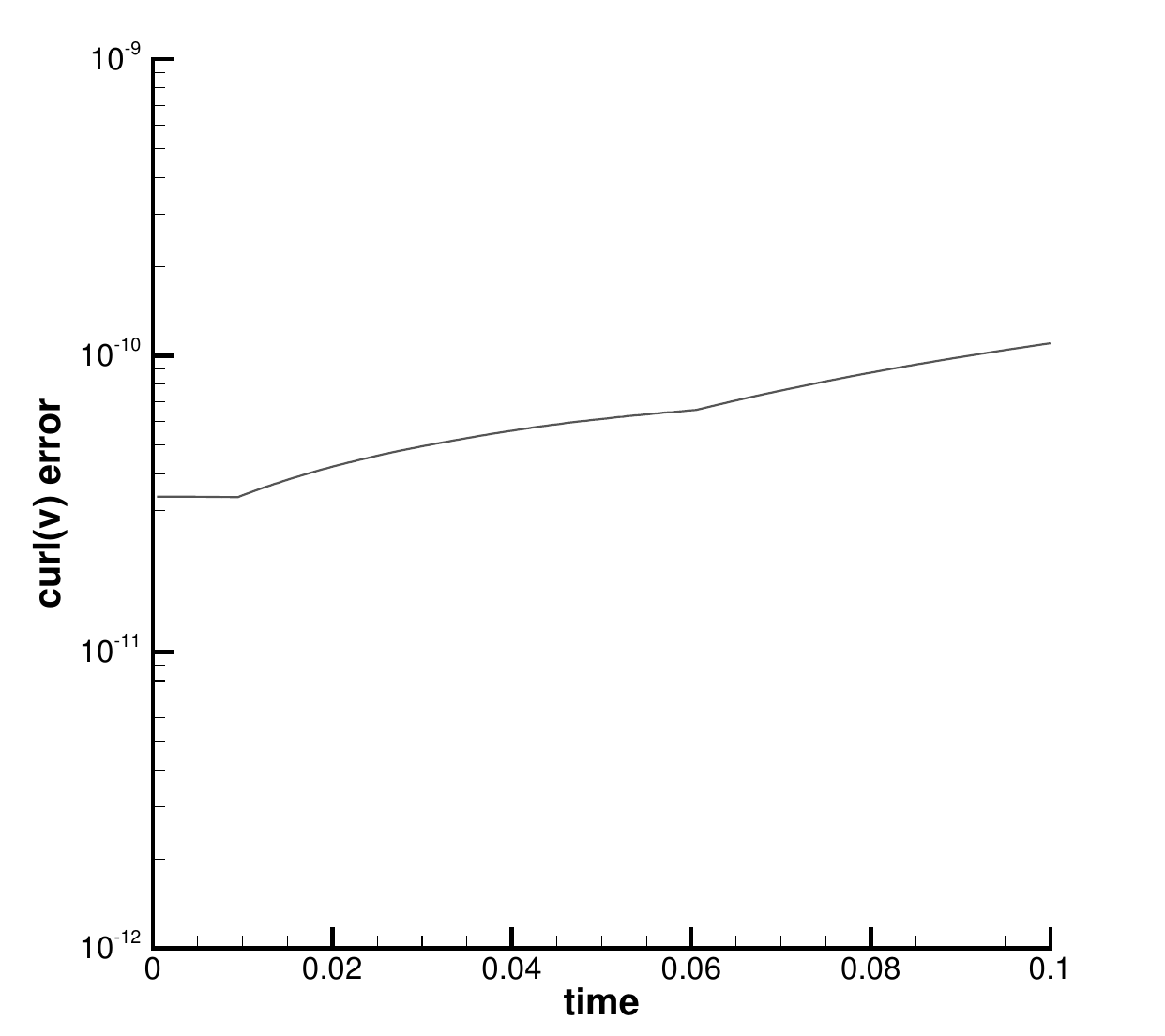}   & 
			\includegraphics[width=0.45\textwidth]{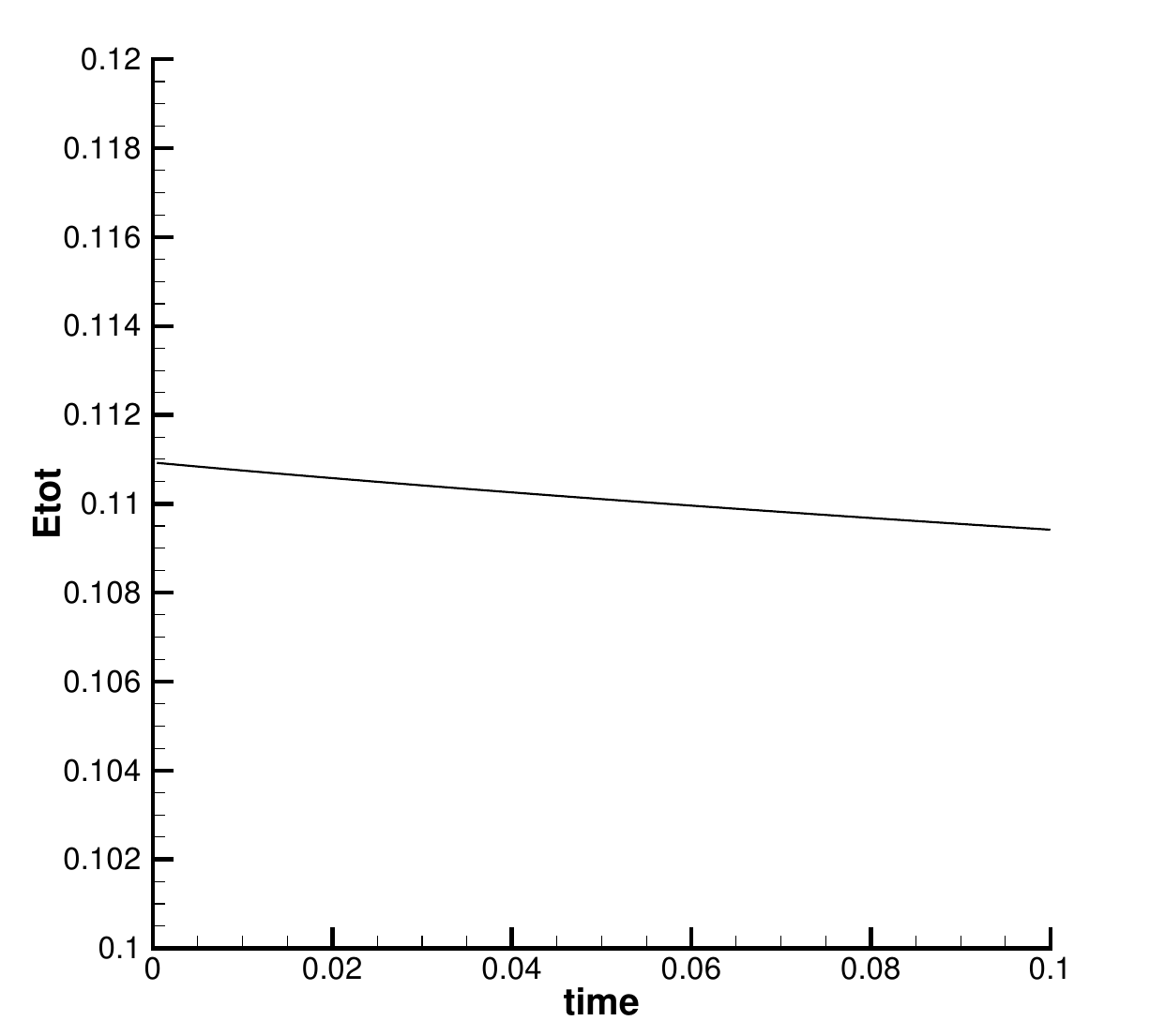}     
		\end{tabular} 
		\caption{Acoustic waves with discontinuities. Time evolution of the $L^\infty$ error of the curl of the velocity field (left) and time evolution of the total energy (right). The method is dissipative since the artificial viscosity for the treatment of discontinuities is switched on.} 
		\label{fig.res.ac.expl}
	\end{center}
\end{figure}

\begin{figure}[!htbp]
	\begin{center}
			\includegraphics[trim=10 10 10 10,clip,width=0.75\textwidth]{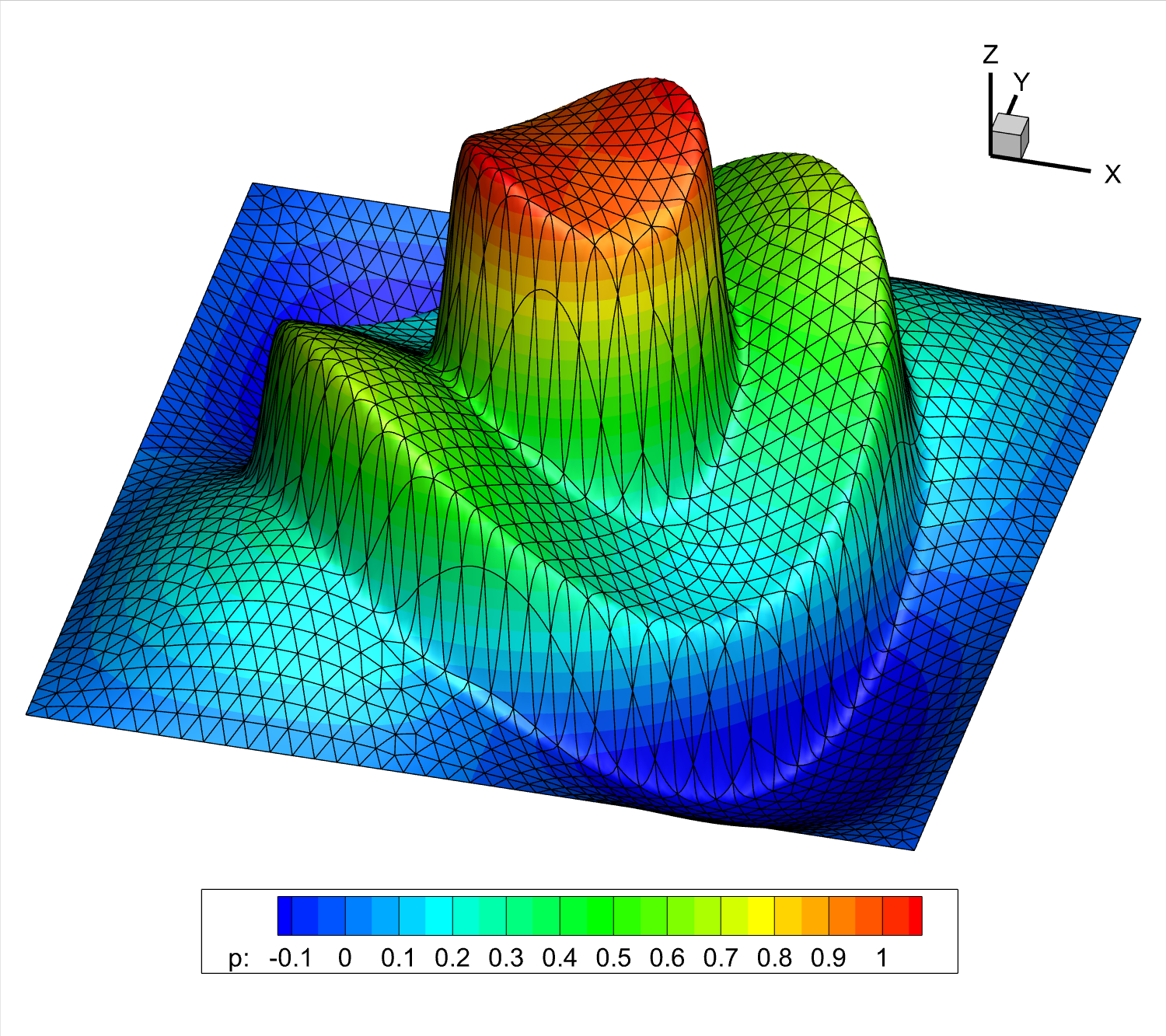}       
		\caption{3D plot of the pressure and the computational mesh at time $t=0.1$.} 
		\label{fig.res.ac.expl3D}
	\end{center}
\end{figure}

\subsection{Vacuum Maxwell equations}

The vacuum Maxwell equations with unit speed of light and in the absence of charges are  
\begin{eqnarray}    
    \frac{\partial \B}{\partial t} + \nabla \times \E & = & 0, \label{eqn.max.B} \\
    \frac{\partial \E}{\partial t} - \nabla \times \B & = & 0. \label{eqn.max.E}      
\end{eqnarray}
Here, $\E$ and $\B$ are the electric and magnetic field, respectively. The system satisfies the following extra conservation law for the total energy density $\mathcal{E} = \halb \E^2 + \halb \B^2$: 
\begin{equation}
\frac{\partial \mathcal{E}}{\partial t} + \nabla \cdot \left( \E \times \B \right) = 0.
    \label{eqn.max.extra}
\end{equation}
From the Maxwell equations \eqref{eqn.max.B} and \eqref{eqn.max.E} it becomes immediately clear that both the magnetic and the electric field must remain divergence-free for all times, if they were initially divergence-free. This means that the following two involutions must hold:
\begin{equation}
   \nabla \cdot \B = 0, \qquad \textnormal{ and } \qquad 
   \nabla \cdot \E = 0.
   \label{eqn.max.inv}
\end{equation}
Hence, for the CG-DG scheme the initial electric and magnetic fields in $\Uh$ must be computed as the discrete curl of two vector potentials in $\Wh$ 
in order to have \eqref{eqn.max.inv} satisfied at the discrete level for all times. 
In order to treat solutions with discontinuities and to maintain at the same time the divergence-free property of the magnetic and the electric field, the artificial viscosity must have a structure that is compatible with the equations. Therefore, the electric and magnetic fields used in the fluxes are modified as follows: 
\begin{equation}
    \tilde \E_h = \E(\wh) + \epsilon \tilde{\nabla} \times \B_h,
    \qquad 
    \tilde \B_h = \B(\wh) - \epsilon \tilde{\nabla} \times \E_h,    
\end{equation}
with $\E(\wh)$ and $\B(\wh)$ the electric and magnetic fields contained in $\wh$ and $\epsilon$ the artificial viscosity coefficient. 

\subsubsection{Smooth electromagnetic wave propagation}

To test the properties of our new CG-DG scheme, we solve the following test problem with initial data   
\begin{equation}
  \B(\x, 0)  = \mathbf{B}_0 \exp \left( -\halb |\x|^2/\sigma^2 \right), \qquad 
  \E(\x, 0)  = \mathbf{E}_0 \exp \left( -\halb |\x|^2/\sigma^2 \right).
	\label{eqn.ic.gauss.max.maxwell}
\end{equation}
Simulations are run in the domain $\Omega=[-\halb,+\halb]^2$ until a final time of $t=10$, using a polynomial  approximation degree of $N=3$. The mesh contains 2028 triangles with $N_x=N_y=30$ edges along each side of the computational domain. For the initial data we set $\B_0 = 0$ and $\E_0 = (0,0,1)$. The half width is chosen as $\sigma=0.05$.
The computational results are shown in Figure \ref{fig.smoothmaxwell}. As expected, the method conserves the total energy up to the small error of the time discretization, and both the magnetic and the electric field remain divergence-free up to machine precision, i.e. $\tilde{\nabla} \cdot \B_h = \tilde{\nabla} \cdot \E_h = 0$ in the sense of \eqref{eqn.disc.divrot}. The linear growth in time in the divergence errors is due to the accumulation of roundoff errors in the computer, while the linear growth in time of the total energy conservation error is due to time discretization errors as our energy conservation property only holds on the semi-discrete level.  

\begin{figure}[!htbp]
	\begin{center}
		\begin{tabular}{ccc} 
			\includegraphics[width=0.3\textwidth]{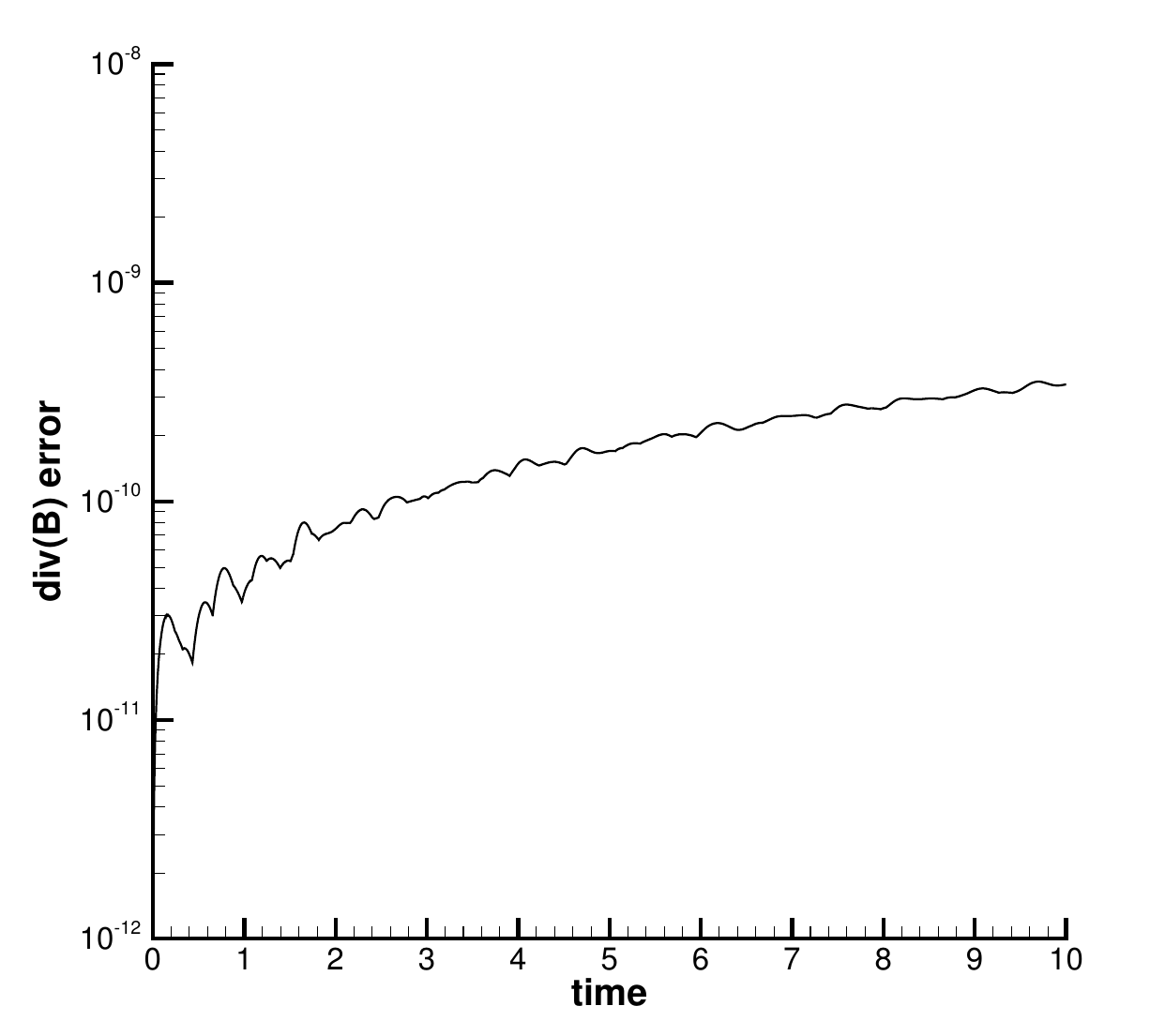}   & 
			\includegraphics[width=0.3\textwidth]{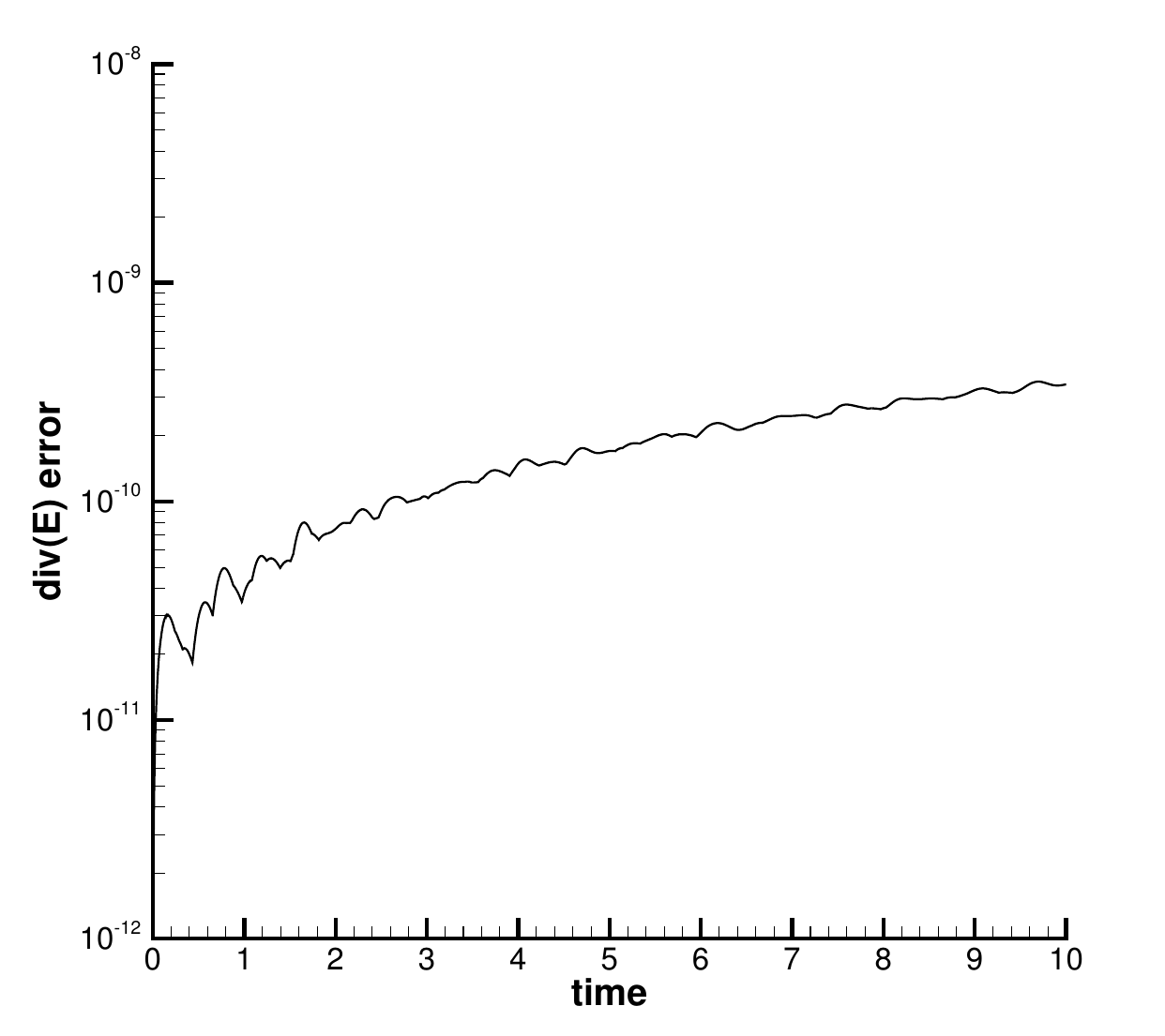}   & 
			\includegraphics[width=0.3\textwidth]{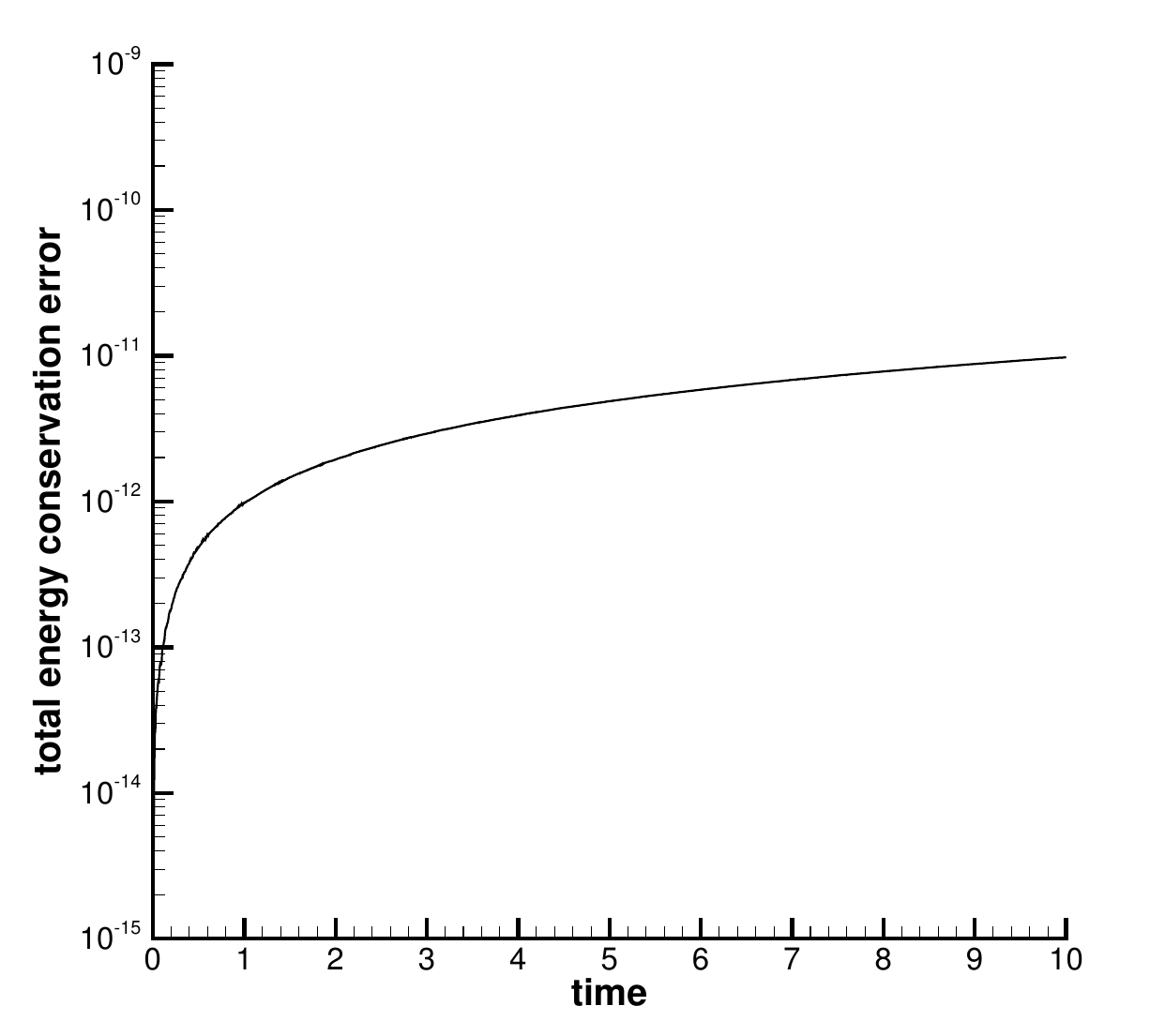}     
		\end{tabular} 
		\caption{Smooth electromagnetic wave propagation. Time evolution of the $L^\infty$ error of the divergence of the magnetic field (left), of the divergence of the electric field (center) and time evolution of the total energy conservation error (right). } 
		\label{fig.smoothmaxwell}
	\end{center}
\end{figure}

\subsubsection{Electromagnetic waves with discontinuities}

We consider the following test problem with nontrivial initial magnetic and electric field, including discontinuities. 
The $x$ and $y$ components of the initial electric and magnetic field are computed as the curl of a vector potential $\A$:   
\begin{equation}
    \B(\x,0) = \nabla \times \A,
    \quad 
    \E(\x,0) = \nabla \times \A,
    \quad \textnormal{ with } \quad 
    \A = \left( 0,0,0.04 \sin(2 \pi x) \sin(2 \pi y) \right), 
\end{equation}
while the $z$ components of the initial electric and magnetic fields are initialized by  
\begin{equation}
    E_3(\x,0), B_3(\x,0)  = \left\{   \begin{array}{ll} 
    1 & \qquad \textnormal{if} \qquad r \leq R, \\
    0 & \qquad \textnormal{if} \qquad r > R,
    \end{array} \right. 
\end{equation}
with $r = |\x|$ and $R=0.25$. 
The computational domain is $\Omega = [-\halb,+\halb]^2$ with periodic boundaries everywhere. The unstructured mesh is composed of 3608 triangles, with $40$ edges along each side of the square domain. 
Simulations are carried out until $t=0.1$ with the new CG-DG scheme with polynomial approximation degree $N=3$ and the compatible artificial viscosity activated.  
The computational results are depicted in Figures \ref{fig.res.max.expl} and \ref{fig.res.max.expl3D}.
The divergence errors of both the magnetic and the electric field remain of the order of machine precision, as expected. The observable linear growth in the errors is due to the accumulation of roundoff errors in the computer. In this test the energy is decreasing, since the method is dissipative due to the presence of the artificial viscosity. 

\begin{figure}[!htbp]
	\begin{center}
		\begin{tabular}{ccc} 
			\includegraphics[width=0.3\textwidth]{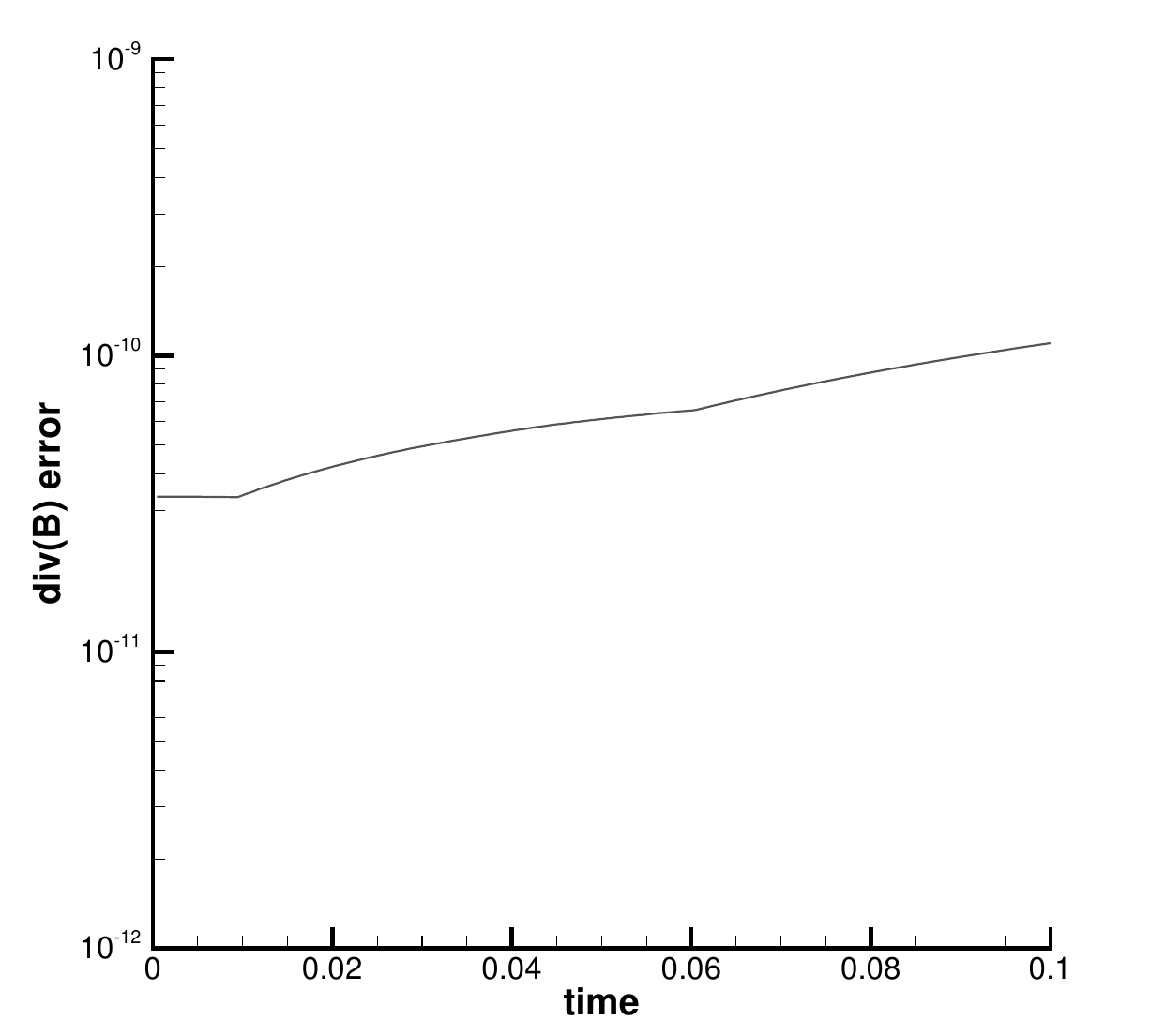}   & 
			\includegraphics[width=0.3\textwidth]{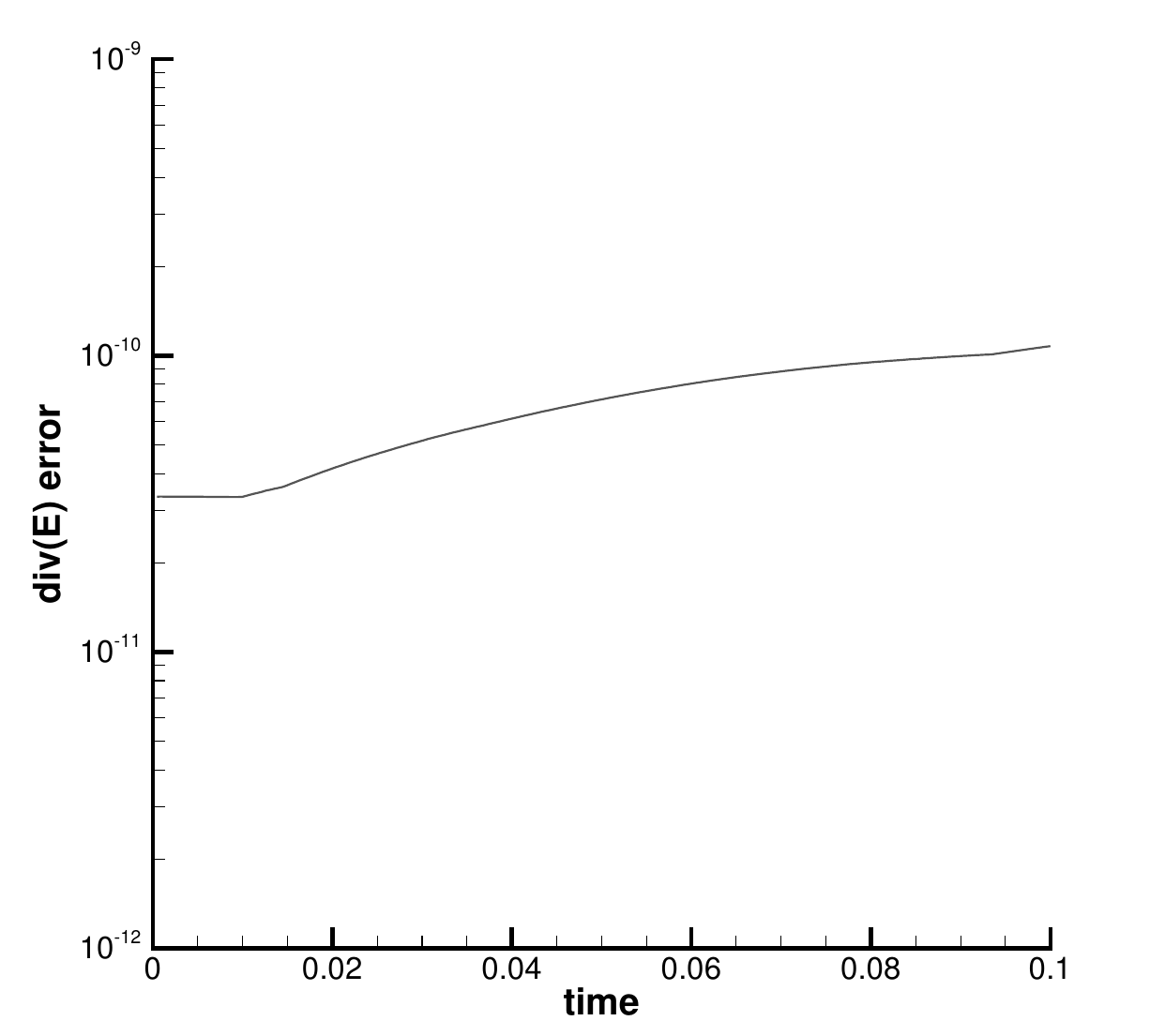}   &      
			\includegraphics[width=0.3\textwidth]{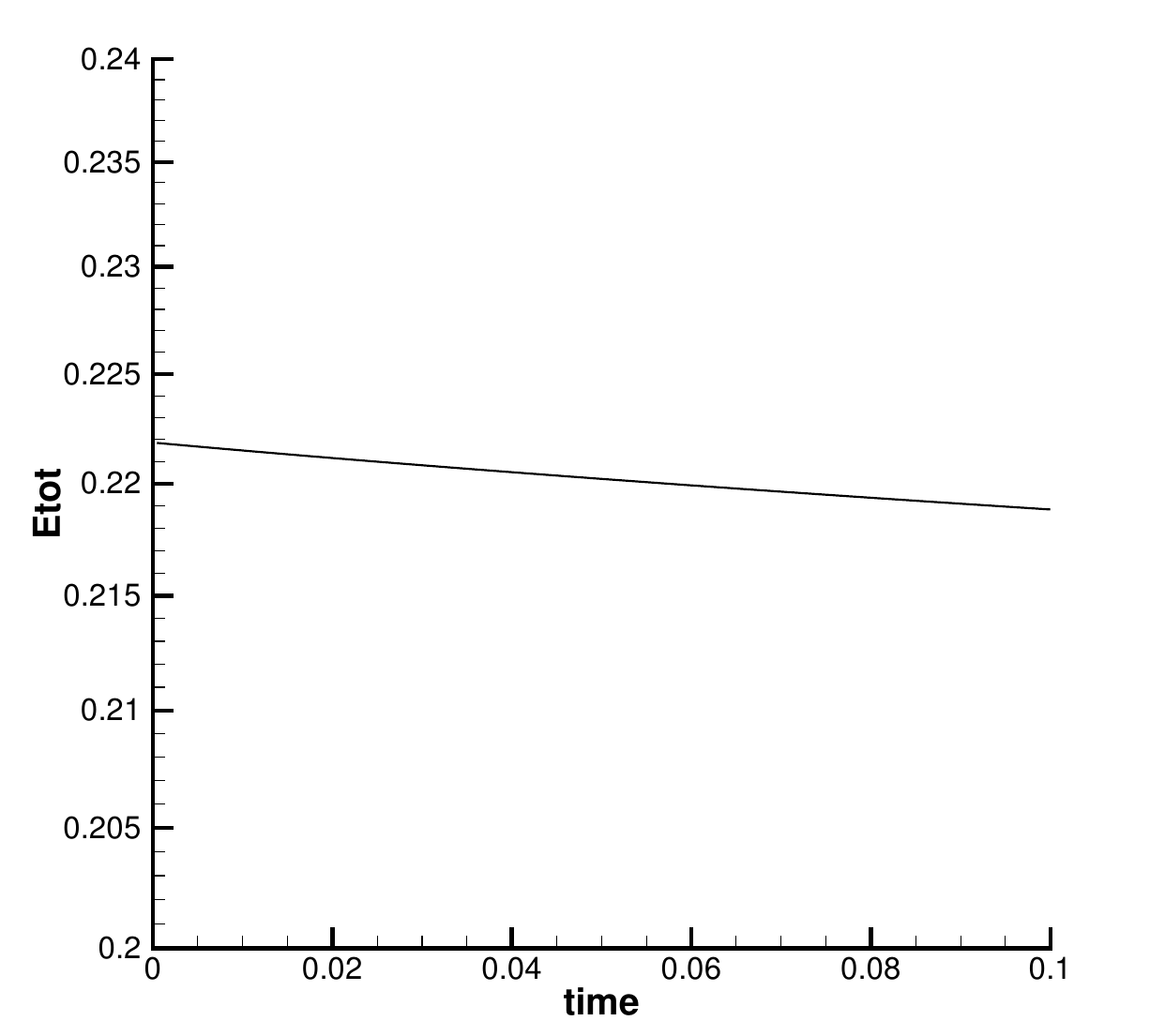}          
		\end{tabular} 
		\caption{Electromagnetic waves with discontinuities. Time evolution of the $L^\infty$ errors of the divergence of the magnetic field (left) and of the divergence of the electric field (center) as well as the time evolution of the total energy (right). The method is dissipative since the artificial viscosity for the treatment of discontinuities is switched on.} 
		\label{fig.res.max.expl}
	\end{center}
\end{figure}

\begin{figure}[!htbp]
	\begin{center}
			\includegraphics[trim=10 10 10 10,clip,width=0.75\textwidth]{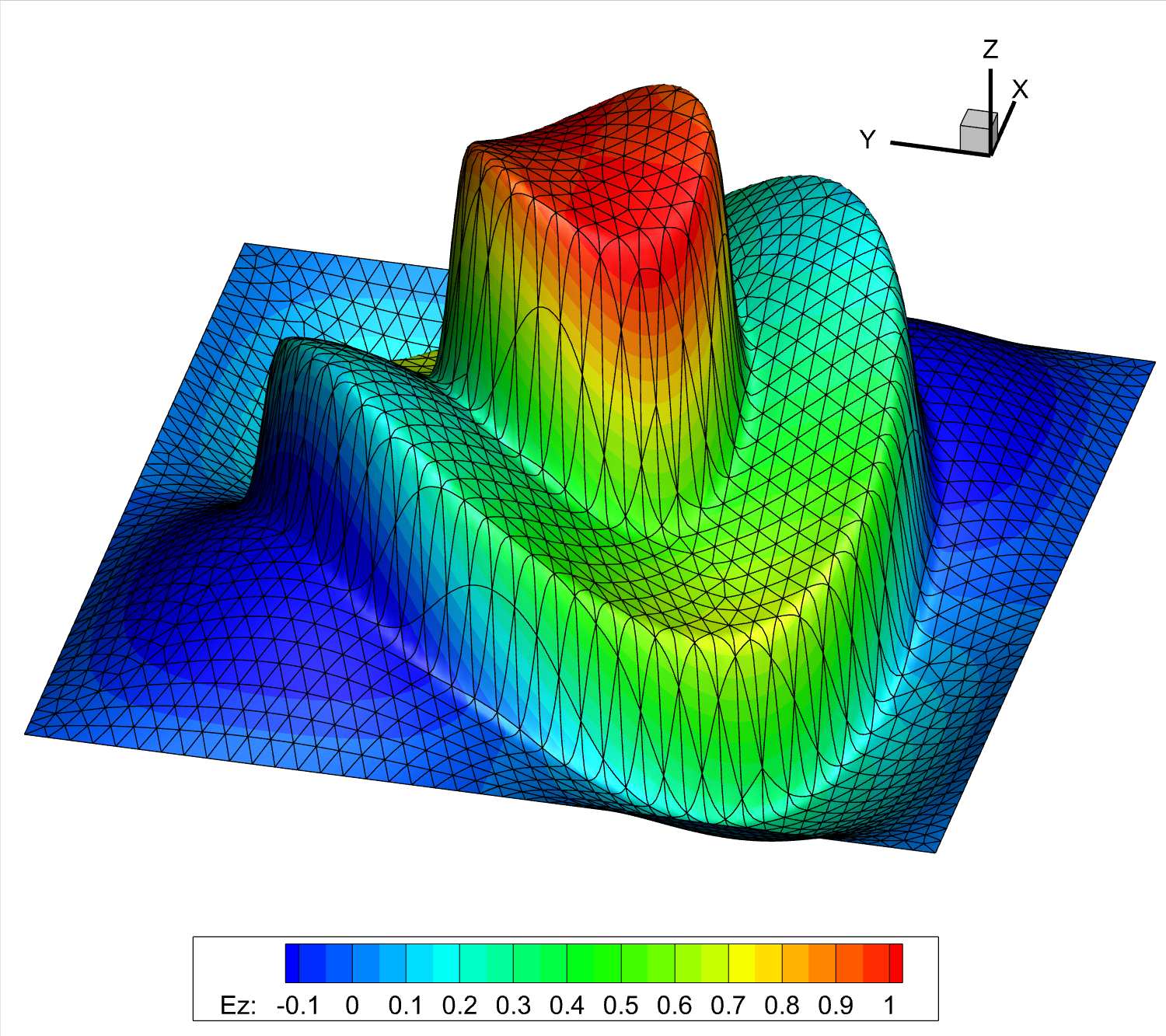}       
		\caption{3D plot of the $z$ component of the magnetic field, coloured by the $z$ component of the electric field, together with the computational mesh at time $t=0.1$.} 
		\label{fig.res.max.expl3D}
	\end{center}
\end{figure}

\subsection{Nonlinear compressible Euler equations}

So far all PDE systems considered in this section were linear and did not present shock waves. Here we apply our new CG-DG scheme also to a nonlinear system of conservation laws, namely the compressible Euler equations. They read
\begin{eqnarray}
    \frac{\partial \rho}{\partial t} + \nabla \cdot (\rho \v) & = & 0, \label{eqn.euler.mass} \\
    \frac{\partial (\rho \v)} {\partial t} + \nabla \cdot (\rho \v \otimes \v ) + \nabla p & = & 0, \label{eqn.euler.mom} \\
    \frac{\partial \mathcal{E}} {\partial t} + \nabla \cdot \left( \v ( \mathcal{E} + p ) \right)  & = & 0, \label{eqn.euler.ene}     
\end{eqnarray}
with $\rho$ the fluid density, $\v$ the velocity vector, $p$ the pressure, $\mathcal{E} = \rho E = \halb \rho \v^2 + \rho e$ the total energy density and $e$ the specific internal energy. For ideal gases the internal energy density $\rho e$ is given by the following equation of state:  
\begin{equation}
\rho e = \frac{p}{\gamma-1},
\label{eqn.eos}
\end{equation}
with $\gamma = c_p/c_v$ the ratio of specific heats,
where $c_p$ is the specific heat at constant pressure  and $c_v$ is the specific heat at constant volume. For diatomic gases it typically assumes the value $\gamma = 1.4$.
The Euler equations satisfy the following entropy inequality
\begin{equation}
\frac{\partial (\rho S)}{\partial t} + \nabla \cdot \left( \v \rho S \right) \ge 0,
\end{equation}
with the specific entropy $S$ that is linked to the other state variables again via the equation of state 
\begin{equation}
    e = \frac{\rho^{\gamma-1}}{\gamma-1} \exp\left( S/c_v \right). 
\end{equation}
However, in this paper the entropy inequality is not explicitly considered at the discrete level. Constructing a provably entropy stable version of our new CG-DG scheme is the topic of future research and beyond the scope of this work. 
Nevertheless, some basic ideas on how to achieve this in the special case $N=0$ have been outlined in Section \ref{sec:Lagrangian}. 

\subsection{Numerical convergence study for the isentropic vortex}\label{sec.vortex}

The order of accuracy of the new CG-DG scheme is  experimentally verified via the well-known isentropic vortex problem, see \cite{HuShuTri}.  
The problem is solved in a periodic domain $\Omega = [0,10]^2$ until a final time of $t=0.2$. In this isentropic flow the entropy is constant. The velocity, temperature, density and pressure profiles read 
\begin{equation}
    \v  =  
    \frac{\varepsilon}{2\pi}e^{\frac{1-r^2}{2}} \left(\begin{array}{c} 5-y \\ x-5 \end{array}\right), 
    \qquad
	\rho  =  (1+\delta T)^{\frac{1}{\gamma-1}}, 
	\qquad 
	p = \left( 1+\delta T \right)^{\frac{\gamma}{\gamma-1}},
	\label{ShuVortDelta}
\end{equation} 
with the temperature fluctuation 
\begin{equation}
 \delta T = -\frac{(\gamma-1)\varepsilon^2}{8\gamma\pi^2}e^{1-r^2}, 
\end{equation}
$r^2=(x-5)^2+(y-5)^2$ and $\varepsilon=5$. 
The above vortex is 
a stationary solution of the compressible Euler equations, hence 
$\q(\x,t) = \q(\x,0)$. The numerical convergence study of the CG-DG schemes is carried out with different polynomial approximation degrees $N$ on a sequence of successively refined unstructured triangular meshes with $N_x$ and $N_y$ edges along the boundaries of $\Omega$ in the $x$ and $y$ direction, respectively. In Tables \ref{tab.conv.DG} and \ref{tab.conv.CG} we report the $L^2$ error norms computed for the DG solution $\uh \in \Uh$ and for the CG solution $\wh \in \Wh$, respectively, at the final time, together with the numerical convergence rates of the scheme. Results are shown for the density $\rho$, the momentum density component $\rho v_1$ and the energy density $\mathcal{E}$. We observe that for $N \geq 1$ the DG solution $\uh$ essentially converges with an order close to $N$ and the CG solution $\wh$ with an order close to $N+1$, which is in agreement with Theorem~\ref{thm:main_error} and Lemma~\ref{lemma:cg_error}, although these estimates have been proven only for linear symmetric hyperbolic systems, while the isentropic vortex is a solution of the nonlinear compressible Euler equations.     
\begin{table}[!ht]  
	\caption{$L^2$ error norms and numerical convergence results for the DG approximation $\uh \in \Uh$ of density $\rho$, momentum density $\rho v_1$ and total energy density $\mathcal{E}$ at time $t=0.2$ obtained with the CG-DG scheme applied to the isentropic vortex of the compressible Euler equations. } 
	\begin{center} 
			\renewcommand{\arraystretch}{1.1}
			\begin{tabular}{ccccccc} 
				\hline
				$N_x$ & $\rho$ & $\rho v_1$ & $\mathcal{E}$ & $\mathcal{O}(\rho)$ & $\mathcal{O}(\rho v_1)$ & $\mathcal{O}(\mathcal{E})$ \\ 
				\hline
				\multicolumn{7}{c}{CG-DG scheme - $N=0$ - errors for the DG solution $\uh \in \Uh$}  \\
				\hline
				40  & 1.676E-01	& 2.294E-01	& 5.817E-01	&      &      &      \\
				60  & 1.220E-01	& 1.720E-01	& 4.275E-01	& 0.8 & 0.7 & 0.8 \\
				80  & 9.259E-02	& 1.293E-01	& 3.227E-01	& 1.0 & 1.0 & 1.0 \\
				100 & 8.805E-02	& 1.083E-01	& 3.116E-01	& 0.2 & 0.8 & 0.2 \\
				120 & 7.370E-02	& 9.528E-02	& 2.607E-01	& 1.0 & 0.7 & 1.0 \\
				\hline
				\multicolumn{7}{c}{CG-DG scheme - $N=1$ - errors for the DG solution $\uh \in \Uh$ }  \\
				\hline 
				20  & 5.271E-02	& 6.763E-02	& 1.899E-01	&      &      &      \\
				40  & 2.382E-02	& 2.295E-02	& 8.305E-02	& 1.1 & 1.6 & 1.2 \\
				60  & 1.510E-02	& 1.313E-02	& 5.250E-02	& 1.1 & 1.4 & 1.1 \\
				80  & 1.112E-02	& 9.522E-03	& 3.846E-02	& 1.1 & 1.1 & 1.1 \\
				100 & 9.110E-03	& 7.264E-03	& 3.167E-02	& 0.9 & 1.2 & 0.9 \\
				\hline
				\multicolumn{7}{c}{CG-DG scheme - $N=2$ - errors for the DG solution $\uh \in \Uh$}  \\
				\hline
				20  & 5.101E-03	& 5.506E-03	& 1.859E-02	&      &      &      \\
				30  & 1.765E-03	& 1.800E-03	& 6.713E-03	& 2.6 & 2.8 & 2.5 \\
				40  & 8.982E-04	& 8.456E-04	& 3.462E-03	& 2.3 & 2.6 & 2.3 \\
				60  & 3.850E-04	& 3.048E-04	& 1.452E-03	& 2.1 & 2.5 & 2.1 \\
				80  & 2.038E-04	& 1.496E-04	& 7.746E-04	& 2.2 & 2.5 & 2.2 \\
				\hline
				\multicolumn{7}{c}{CG-DG scheme - $N=3$ - errors for the DG solution $\uh \in \Uh$}  \\
				\hline
				10  & 5.454E-03	& 6.893E-03	& 2.024E-02	&      &      &      \\
				20  & 4.936E-04	& 5.561E-04	& 1.925E-03	& 3.5 & 3.6 & 3.4 \\
				30  & 1.314E-04	& 1.450E-04	& 5.173E-04	& 3.3 & 3.3 & 3.2 \\
				40  & 5.332E-05	& 5.854E-05	& 2.122E-04	& 3.1 & 3.2 & 3.1 \\
				60  & 1.903E-05	& 1.702E-05	& 7.337E-05	& 2.5 & 3.0 & 2.6 \\
				\hline
				\multicolumn{7}{c}{CG-DG scheme - $N=4$ - errors for the DG solution $\uh \in \Uh$}  \\
				\hline
				8   & 2.517E-03	& 2.995E-03	& 9.072E-03	&      &      &      \\
				12  & 7.583E-04	& 9.180E-04	& 2.697E-03	& 3.0 & 2.9 & 3.0 \\
				16  & 2.012E-04	& 2.365E-04	& 7.809E-04	& 4.6 & 4.7 & 4.3 \\
				20  & 7.821E-05	& 9.048E-05	& 3.206E-04	& 4.2 & 4.3 & 4.0 \\
				30  & 1.461E-05	& 1.150E-05	& 5.476E-05	& 4.1 & 5.1 & 4.4 \\
				\hline 
			\end{tabular}
	\end{center}
	\label{tab.conv.DG}
\end{table} 

\begin{table}[!ht]  
	\caption{$L^2$ error norms and numerical convergence results for the CG approximation $\wh \in \Wh$ of density $\rho$, momentum density $\rho v_1$ and total energy density $\mathcal{E}$ at time $t=0.2$ obtained with the CG-DG scheme applied to the isentropic vortex of the compressible Euler equations. } 
	\begin{center} 
			\renewcommand{\arraystretch}{1.1}
			\begin{tabular}{ccccccc} 
				\hline
				$N_x$ & $\rho$ & $\rho v_1$ & $\mathcal{E}$ & $\mathcal{O}(\rho)$ & $\mathcal{O}(\rho v_1)$ & $\mathcal{O}(\mathcal{E})$ \\ 
				\hline
				\multicolumn{7}{c}{CG-DG scheme - $N=0$ - errors for the CG solution $\wh \in \Wh$}  \\
				\hline
				40  & 9.625E-02	& 1.349E-01	& 3.244E-01	&      &      &      \\
				60  & 6.301E-02	& 9.073E-02	& 2.137E-01	& 1.0 & 1.0 & 1.0 \\
				80  & 4.529E-02	& 7.504E-02	& 1.491E-01	& 1.1 & 0.7 & 1.3 \\
				100 & 3.847E-02	& 5.568E-02	& 1.292E-01	& 0.7 & 1.3 & 0.6 \\
				120 & 3.197E-02	& 5.094E-02	& 1.043E-01	& 1.0 & 0.5 & 1.2 \\
				\hline
				\multicolumn{7}{c}{CG-DG scheme - $N=1$ - errors for the CG solution $\wh \in \Wh$ }  \\
				\hline 
				20  & 3.306E-02	& 5.727E-02	& 1.220E-01	&      &      &      \\
				40  & 8.238E-03	& 1.431E-02	& 2.985E-02	& 2.0 & 2.0 & 2.0 \\
				60  & 3.670E-03	& 6.360E-03	& 1.334E-02	& 2.0 & 2.0 & 2.0 \\
				80  & 2.031E-03	& 3.536E-03	& 7.392E-03	& 2.1 & 2.0 & 2.1 \\
				100 & 1.274E-03	& 2.270E-03	& 4.642E-03	& 2.1 & 2.0 & 2.1 \\
				\hline
				\multicolumn{7}{c}{CG-DG scheme - $N=2$ - errors for the CG solution $\wh \in \Wh$}  \\
				\hline
				20  & 2.426E-03	& 3.112E-03	& 6.509E-03	&      &      &      \\
				30  & 5.931E-04	& 9.700E-04	& 1.948E-03	& 3.5 & 2.9 & 3.0 \\
				40  & 2.164E-04	& 4.134E-04	& 9.153E-04	& 3.5 & 3.0 & 2.6 \\
				60  & 8.094E-05	& 1.128E-04	& 2.125E-04	& 2.4 & 3.2 & 3.6 \\
				80  & 2.636E-05	& 5.057E-05	& 1.059E-04	& 3.9 & 2.8 & 2.4 \\
				\hline
				\multicolumn{7}{c}{CG-DG scheme - $N=3$ - errors for the CG solution $\wh \in \Wh$}  \\
				\hline
				10  & 2.783E-03	& 4.339E-03	& 8.058E-03	&      &      &      \\
				20  & 1.733E-04	& 2.695E-04	& 5.448E-04	& 4.0 & 4.0 & 3.9 \\
				30  & 3.117E-05	& 5.471E-05	& 1.250E-04	& 4.2 & 3.9 & 3.6 \\
				40  & 9.115E-06	& 1.706E-05	& 3.366E-05	& 4.3 & 4.1 & 4.6 \\
				60  & 1.808E-06	& 4.773E-06	& 6.330E-06	& 4.0 & 3.1 & 4.1 \\
				\hline
				\multicolumn{7}{c}{CG-DG scheme - $N=4$ - errors for the CG solution $\wh \in \Wh$}  \\
				\hline
				8   & 1.368E-03	& 1.667E-03	& 3.834E-03	&      &      &      \\
				12  & 3.587E-04	& 5.092E-04	& 9.444E-04	& 3.3 & 2.9 & 3.5 \\
				16  & 6.430E-05	& 9.533E-05	& 2.075E-04	& 6.0 & 5.8 & 5.3 \\
				20  & 1.998E-05	& 2.907E-05	& 5.951E-05	& 5.2 & 5.3 & 5.6 \\
				30  & 2.608E-06	& 5.183E-06	& 6.691E-06	& 5.0 & 4.3 & 5.4 \\
				\hline 
			\end{tabular}    
	\end{center}
	\label{tab.conv.CG}
\end{table} 

\subsection{2D circular Sod problem}\label{sec.ep2d}

As a last test case we consider a circular version of the Sod problem. The computational domain is $\Omega=[-0.5,+0.5]^2$ and the initial condition is given by
\begin{equation}
	     \q(\x,0)=\begin{cases}
		\q_L \quad \text{if} \quad r\leq R \\
		\q_R \quad \text{if} \quad r>R, \\
	\end{cases}
\end{equation}
with $\q_L$ and $\q_R$ the inner and outer states, respectively, and $r=\sqrt{x^2+y^2}$ being the radial coordinate. The initial discontinuity is placed at a radius of $R=0.25$.
For the inner state we set the density and the pressure to 
$\rho_{L}=1$ and $p_{L}=1$, while in the outer state we impose $\rho_{R}=0.125$ and $p_{R}=0.1$. The initial velocity is set to zero everywhere and the final time is $t=0.1$. We solve the problem with the new CG-DG scheme with polynomial approximation degree of $N=3$ using a very coarse unstructured triangular mesh of characteristic mesh spacing $h=1/40$ leading to a total number of 3608 triangles. In this test, the artificial viscosity regularization is needed, and we compare the results obtained with $\chi_j = 1$ everywhere with those obtained with the indicator $\chi_j$ based on the OE sensor of \cite{AbgrallLiuDg}.    
The reference solution is obtained by solving a 1D radial Euler system with source terms using a classical second-order MUSCL-Hancock scheme on a very fine mesh, as described in detail in \cite{toro-book}.  
In Figure \ref{fig.ep2d1dcut} a comparison of the reference solution with the numerical results obtained at the aid of the new CG-DG scheme on a 1D cut along the axis $y=0$ is provided for density, radial velocity and pressure. One observes overall a good agreement. With the OE sensor, the head and the tail of the rarefaction fan and the peak of the radial velocity are clearly better resolved than with $\chi_j = 1$, while the shock is captured equally sharply in both cases. The contact discontinuity is only slightly sharper with the OE sensor and remains smeared on this coarse mesh. On the finer mesh with $h=1/100$, i.e. with 22476 triangles, the results obtained with the OE sensor are in very good agreement with the reference solution, and also the contact discontinuity is resolved sharply. In the left panel of Figure \ref{fig.ep2d} a 3D plot of the coarse mesh used in our simulation as well as the density color contours are shown at the final time.  

\begin{figure}[!htbp]
	\begin{center}
			\includegraphics[width=0.6\textwidth]{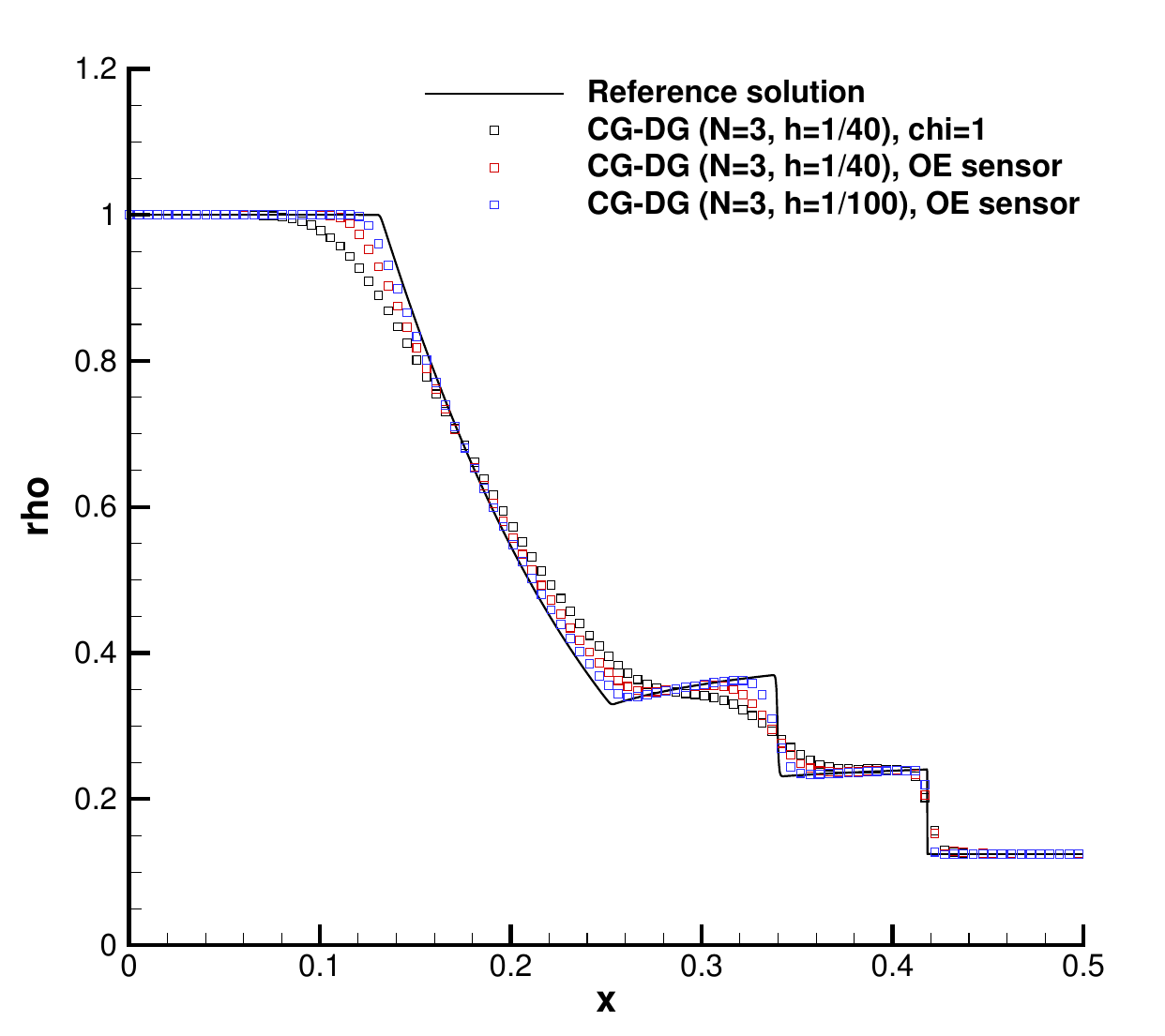}   \\ 
			\includegraphics[width=0.6\textwidth]{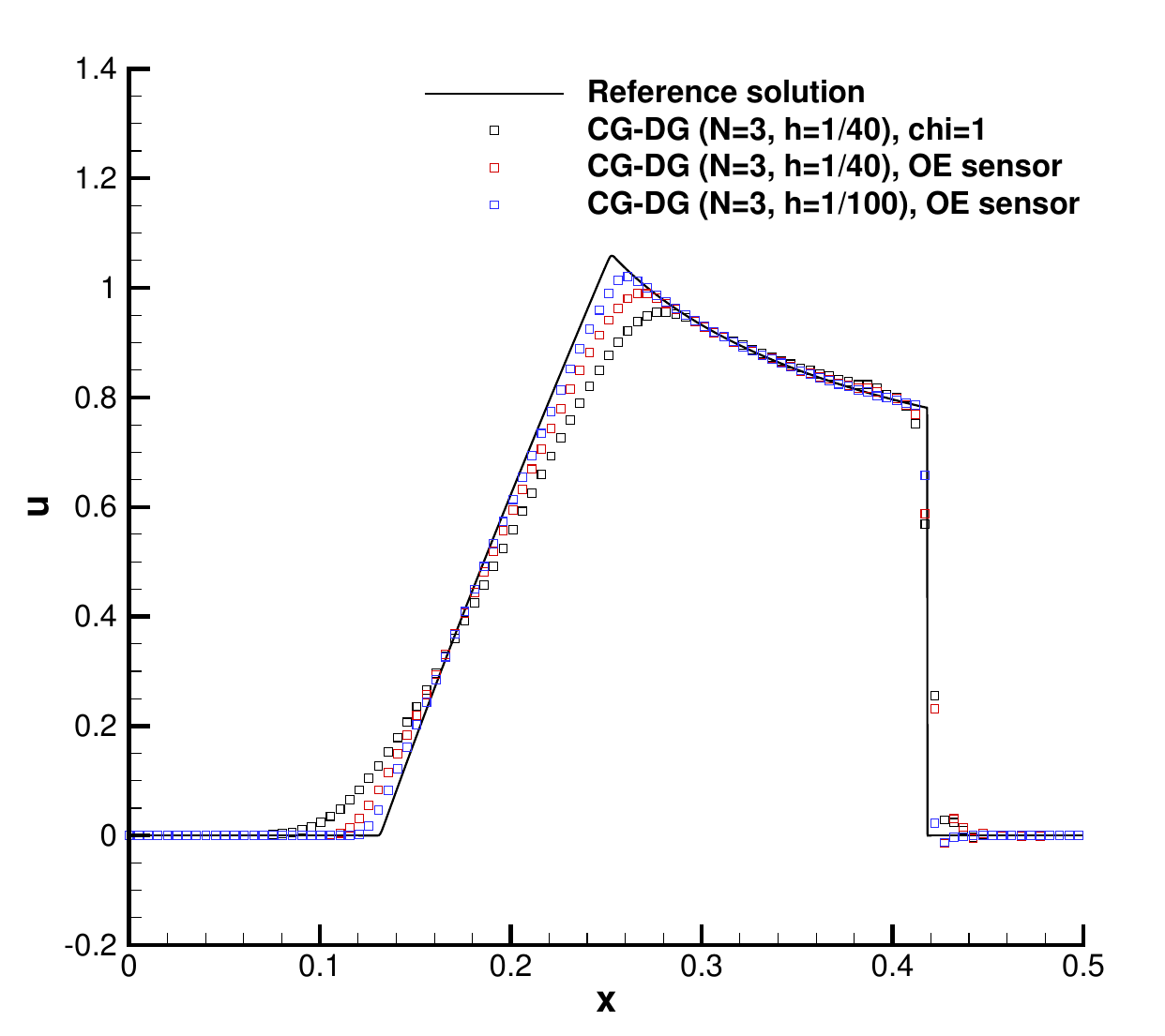}     \\  
			\includegraphics[width=0.6\textwidth]{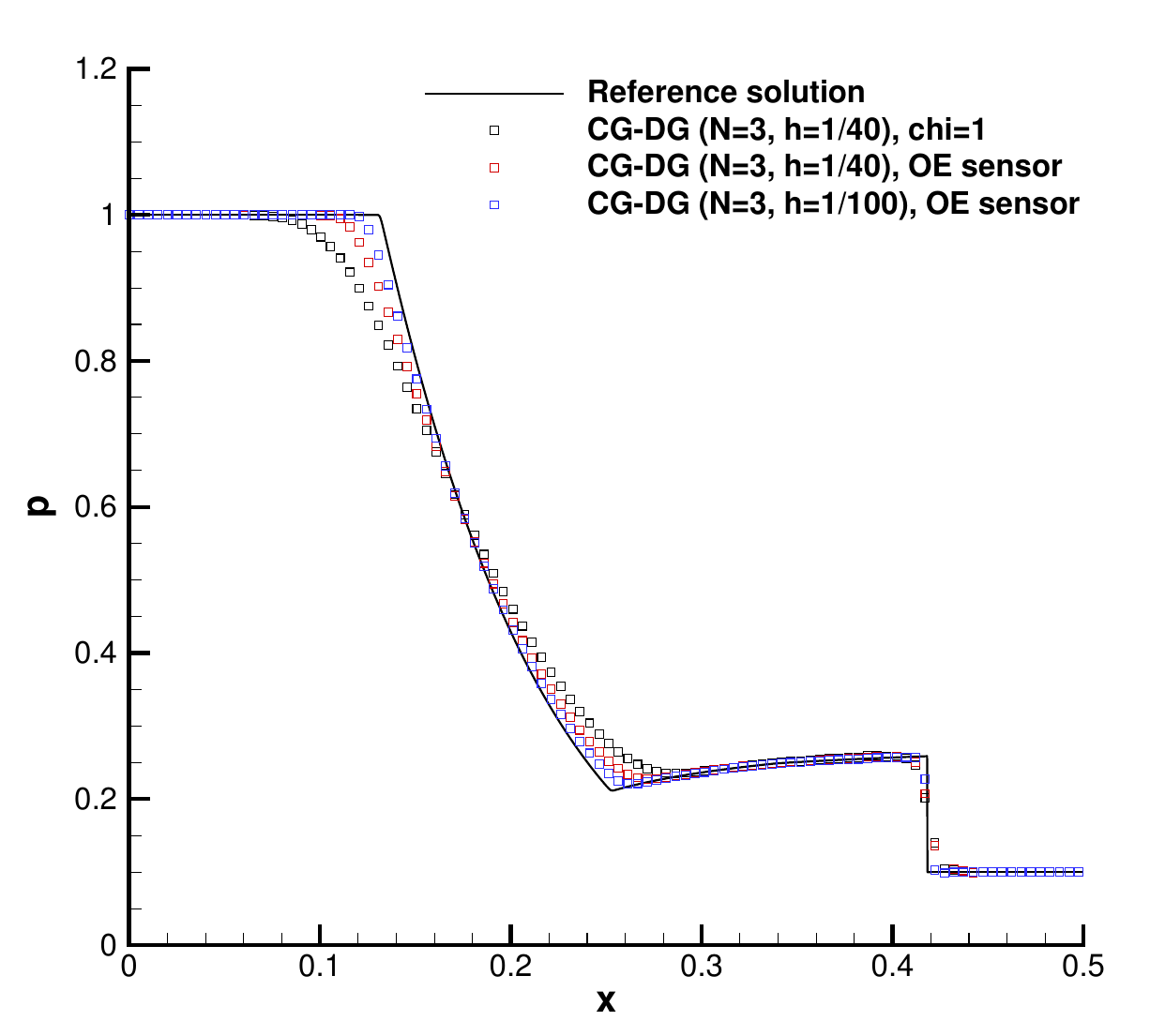}         
		\caption{Cut through the numerical solution along the axis $y=0$ at the final time $t=0.1$ obtained with $\chi_j = 1$ and with the OE sensor on the coarse mesh ($h=1/40$), and with the OE sensor on the fine mesh ($h=1/100$), and comparison with the reference solution for density (top), radial velocity (middle) and pressure (bottom).} 
		\label{fig.ep2d1dcut}
	\end{center}
\end{figure}

\begin{figure}[!htbp]
	\begin{center}
		\begin{tabular}{cc} 
			\includegraphics[trim=10 10 10 10,clip,width=0.5\textwidth]{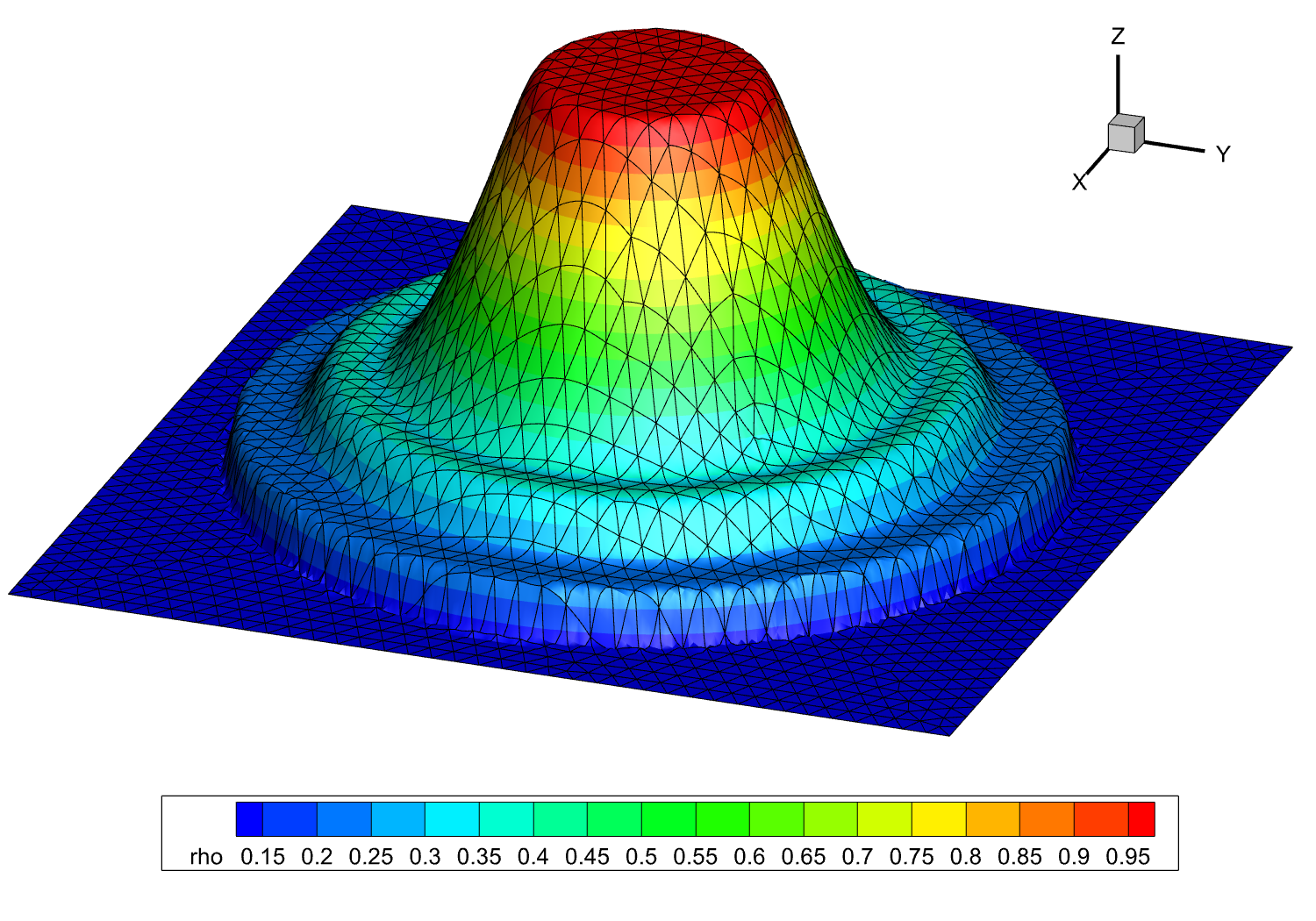}   & 
			\includegraphics[width=0.4\textwidth]{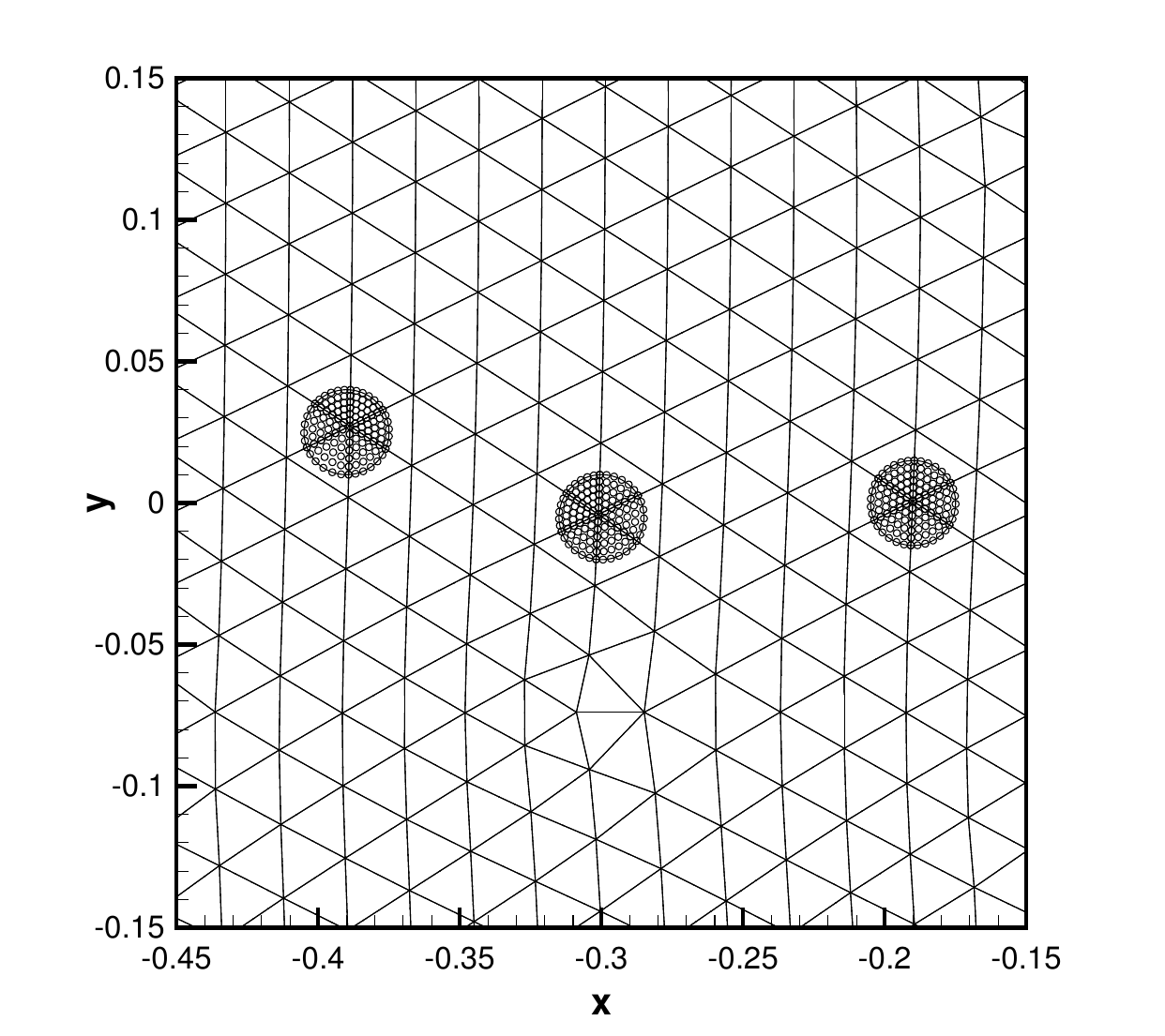}       
		\end{tabular} 
		\caption{2D circular Sod problem. Mesh and 3D view of the numerical solution $\wh$ for density at $t=0.1$ (left) using the OE sensor \cite{AbgrallLiuDg}. Zoom into the mesh and visualization of three circular control volumes $V_1$, $V_2$ and $V_3$ and their approximation via the union of curved isoparametric triangles of degree $n_c=6$ (right).} 
		\label{fig.ep2d}
	\end{center}
\end{figure}

In order to verify the new pointwise conservation property \eqref{eqn:remi2} we place 4 observation points $\x_o$ in the domain $\Omega$. The observation points deliberately do \textit{not} coincide with any mesh point or with any location of the degrees of freedom of $\wh$. We also define 3 circular control volumes $V_1$, $V_2$ and $V_3$ of radius $r_c = 0.015$ with centers $\x_c$ in order to check the integral form of the conservation law over generic control volumes $V$, see \eqref{eq:generalConservation}. The integrals appearing in the integral form \eqref{eq:generalConservation} over arbitrary control volumes $V$ are computed by approximating the circular control volumes via a set of isoparametric triangles of approximation degree $n_c = 6$ for the geometry and using suitable one- and two-dimensional Gaussian quadrature formulas \cite{stroud} in the subvolumes and on the boundary segments. We use $m_c=n_c+2$ quadrature points in 1D and $m_c^2$ quadrature points in 2D.    
In Table \ref{tab.EPerrors} we report the coordinates $\x_o$ and $\x_c$ together with the coordinates of the nearest mesh point $\x_n$ with $\x_o \neq \x_n$ and $\x_c \neq \x_n$, as well as the $L^{\infty}$ errors of \eqref{eqn:remi2} and \eqref{eq:generalConservation} observed for all Runge-Kutta stages and time steps during the entire simulation for all state variables $\q$. 
As expected, the error is of the order of machine precision, providing numerical evidence that the theoretical properties also hold in our practical implementation.        

\begin{table}  
	\caption{Locations of the observation points $\x_o$ used to verify \eqref{eqn:remi2} and centers $\x_c$ of the circular control volumes $V_i$ used to verify \eqref{eq:generalConservation}. The nearest mesh point $\x_n$ is also given to show that $\x_o \neq \x_n$
		and $\x_c \neq \x_n$. The $L^\infty$ errors observed in all time steps and Runge-Kutta stages during the entire simulation with $\chi_j = 1$ are reported. } 	
	\begin{center} 
		\renewcommand{\arraystretch}{1.1}
		\begin{tabular}{cccccc} 
			\hline 
			$\x_o$ & $\x_n$ & $L^\infty(\rho)$ & $L^\infty(\rho v_1)$ & $L^\infty(\rho v_2)$ & $L^\infty(\rho E)$ \\ 
			\hline 
(-3.00E-01, 0.00E+00) & (-3.01E-01, -4.03E-03) & 3.47E-13 & 6.04E-13 & 2.51E-13 & 1.06E-12 \\
(-2.50E-01, 0.00E+00) & (-2.57E-01, -7.77E-03) & 7.39E-14 & 2.03E-13 & 7.98E-13 & 2.28E-13 \\
(-2.00E-01, 0.00E+00) & (-1.90E-01, 7.43E-04) & 2.39E-13 & 7.53E-13 & 3.36E-13 & 7.43E-13 \\
(-1.50E-01, 0.00E+00) & (-1.46E-01, -2.62E-03) & 1.55E-13 & 8.77E-13 & 2.95E-13 & 5.38E-13 \\
			\hline 
			$\x_c$ & $\x_n$ & $L^\infty(\rho)$ & $L^\infty(\rho v_1)$ & $L^\infty(\rho v_2)$ & $L^\infty(\rho E)$ \\ 								
			\hline 
(-3.00E-01, -5.00E-03) & (-3.01E-01, -4.03E-03) & 1.04E-11 & 2.13E-11 & 7.49E-12 & 5.66E-11 \\
(-1.90E-01, 0.00E+00) & (-1.90E-01, 7.43E-04) & 8.60E-13 & 1.08E-12 & 1.11E-12 & 3.38E-12 \\
(-3.90E-01, 2.50E-02) & (-3.89E-01, 2.72E-02) & 5.03E-10 & 8.30E-10 & 2.38E-10 & 2.06E-09 \\
			\hline 
		\end{tabular}
	\end{center}
	\label{tab.EPerrors}
\end{table}

\section{Conclusions}
\label{sec.concl} 

In this paper we have introduced a new class of continuous-discontinuous Galerkin (CG-DG) finite element schemes for linear and nonlinear hyperbolic systems of partial differential equations on unstructured simplex meshes. The key idea of the method relies on two compatible discrete ansatz spaces: a \textit{discontinuous one} (DG) $\Uh$ based on piecewise polynomials of degree $N$ that is used to discretize the \textit{state vector} of the system, and a \textit{globally continuous one} $\Wh$ using piecewise polynomials of degree $N+1$ for the representation of the \textit{discrete flux field}. 
Indeed, the introduction of a globally continuous discrete flux field is a major pillar in our new method.
The continuous discrete flux field is reconstructed from the discontinuous state vector represented and evolved in the DG scheme. The proposed scheme is by construction locally pointwise conservative and satisfies both vector calculus identities at the discrete level exactly pointwise everywhere. The discrete vector calculus identities are an immediate consequence of the fact that the two discrete nabla operators appearing in the CG-DG scheme satisfy a discrete Schwarz theorem.       

Our approach can also be seen as a direct and natural extension of cell-centered finite volume schemes with vertex-based fluxes 
\cite{Despres2005,Maire2007,Despres2009,Maire2020,PHRaph2,MultidOsher,HTCLagrange,HTCLagrangeGPR} to arbitrarily high order of accuracy in space. 

Thanks to a suitable artificial viscosity regularization, our method maintains the exact pointwise conservation and the exact preservation of the involutions also in the presence of shock waves and discontinuities. 

Future work will concern the extension to more general unstructured meshes and to a discretely thermodynamically compatible formulation for nonlinear HTC systems, see e.g. \cite{HTCDG,HTCAbgrall}. We furthermore plan to extend the method to the MHD equations, the equations of nonlinear hyperelasticity in Eulerian coordinates \cite{PeshRom2014,GPRmodel,GPRmodelMHD} as well as to compressible multi-phase flows with surface tension \cite{SIST}, all of which contain divergence and/or curl constraints for some vector fields of the system. Since the new CG-DG scheme automatically preserves divergence and curl involutions exactly at the discrete level, it may be an interesting method to be tried in the context of numerical general relativity, in particular in the context of first order hyperbolic reductions of the Einstein field equations \cite{ADERCCZ4,FOCCZ4GLM,FOZ4,FOBSSNOK,FOBSSNOKDG}, which come along with a large number of curl involutions.  
We also want to investigate the use of more sophisticated and provably bound-preserving limiters, as well as energy-conserving time integrators \cite{Brugnano1,Brugnano2}.  

\section*{Acknowledgments}

M.D. was funded by the Fondazione Caritro via the project SOPHOS and by the European Research Council (ERC) under the European Union's Horizon Europe research and innovation programme via the project SOPHOS, grant agreement no. ERC-ADG-2025-101265878-SOPHOS. M.D. also received funding from a scientific collaboration with the ERC project ERC-ADG-2021-101052956-BEYOND. 
Views and opinions expressed are however those of the authors only and do not necessarily reflect those of the European Union or the European Research Council Executive Agency. Neither the European Union nor the granting authority can be held responsible for them.

M.D. and E.Z. are members of the INdAM GNCS group. 

The authors would like to thank the two anonymous referees for the thorough review and for their constructive comments, which greatly helped to improve the quality and clarity of this manuscript.

\appendix
\section{Application of the CG-DG scheme to the Poisson equation} 
In this short appendix we show the application of the original staggered CG-DG method introduced in \cite{CompatibleDG1} to the Poisson equation. 
Consider the model problem 
\begin{eqnarray}
	\v & = & \nabla u, \label{eqn.lap.v} \\
	\nabla \cdot \v & = & f(\x), \label{eqn.lap.div} 
\end{eqnarray}
with $u = 0$ on $\partial \Omega$, 
which is nothing else than a first order reformulation of the Poisson equation 
\begin{equation}
	\nabla^2 u = f(\x), \label{eqn.laplace} 
\end{equation}
with $u = 0$ on $\partial \Omega$. 
Since \eqref{eqn.lap.v} is a simple definition, following the choice made in \cite{CompatibleDG1}, according to which variables that are simple consequences of definitions are defined in the DG space, we choose $\vh \in \Uh$. Since the Poisson equation \eqref{eqn.laplace} can be obtained from a variational principle, following \cite{CompatibleDG1} we choose $\tilde{u}_h \in \Wh$. With the boundary condition $u = 0$ on $\partial \Omega$, which we impose by choosing $\tilde{u}_h$ and the test functions $\psi_j$ in the subspace $\mathcal{W}_{h,0}^{N+1} \subset \Wh$ of functions that vanish on $\partial \Omega$, the weak formulation of the mixed problem above leads to 
  \begin{eqnarray}
  	\int \limits_{\Omega} \phi_i \vh d\x & = & 
  	\int \limits_{\Omega} \phi_i \nabla \tilde{u}_h d\x, \label{eqn.lap.vh} \\
  	-\int \limits_{\Omega} \nabla \psi_j \cdot \vh d\x & = & \int \limits_{\Omega}  \psi_j f(\x) \, d\x. \label{eqn.lap.divh} 
  \end{eqnarray}
But since $\vh \in \Uh$ and $\tilde{u}_h \in \Wh$ the first identity \eqref{eqn.lap.vh} holds even strongly, i.e. $\vh = \nabla \tilde{u}_h$. Inserting in \eqref{eqn.lap.divh} leads to 
\begin{equation}
	-\int \limits_{\Omega} \nabla \psi_j \cdot \nabla \tilde{u}_h \, d\x = \int \limits_{\Omega}  \psi_j f(\x) \, d\x,
\end{equation} 
which is simply the \textit{classical} continuous finite element discretization of the Poisson equation \cite{ZienTaylorZhu}. Note that the ansatz spaces of the mixed CG-DG method are chosen differently from the Brezzi-Douglas-Marini (BDM) elements \cite{BDM}. 

\section*{Data Availability}

The data can be obtained from the authors on reasonable request. 

\section*{Conflict of interest} 
The authors declare that they have no conflict of interest.

\bibliographystyle{plain}
\bibliography{biblio,ra}
\end{document}